\documentclass[12pt]{amsart}

\usepackage[T1]{fontenc}
\usepackage[utf8]{inputenc}
\usepackage{lmodern}

\usepackage[american]{babel}
\usepackage[babel, final]{microtype}
\usepackage[all]{xy}
\usepackage{ifpdf}
\usepackage{multirow}

\usepackage[pdftex, dvipsnames]{xcolor}
\usepackage[pdftex, final]{graphicx}
\usepackage{pinlabel}
\usepackage[pdftex,%
    lmargin=1in,%
    rmargin=1in,%
    tmargin=1in,%
    bmargin=1in,%
    nomarginpar,%
    includehead,%
    includefoot%
]{geometry}

\usepackage{amsmath,amsthm,amssymb,amsfonts,mathrsfs,tikz,tikz-cd,bm,verbatim,mathtools,caption,enumitem,csquotes,float,setspace,subcaption,graphicx,comment,framed}

\usepackage[normalem]{ulem}

\usepackage{hyperref}
\hypersetup{
    colorlinks=true,
    linkcolor=NavyBlue,
    citecolor=NavyBlue
}

\usepackage[hang,flushmargin]{footmisc} 

\usepackage[backend=biber,style=alphabetic,sorting=nyt, url=false,doi=false,isbn=false,backref=true,maxnames=10]{biblatex}
\renewbibmacro{in:}{}
\DefineBibliographyStrings{english}{
  backrefpage={$\uparrow$},
  backrefpages={$\uparrow$}
}
\AtBeginBibliography{\small}
\addto\extrasamerican{%
}

\numberwithin{equation}{section}

\newtheorem{theorem}{Theorem}[section]

\newtheorem{proposition}{Proposition}[section]
\newtheorem{lemma}{Lemma}[section]

\newtheorem{conjecture}{Conjecture}[section]

\newtheorem*{theorem*}{Theorem}

\theoremstyle{definition}
\newtheorem{definition}{Definition}[section]
\newtheorem{example}{Example}[section]
\newtheorem{problem}{Problem}[section]
\newtheorem{remark}{Remark}[section]

\newtheorem{question}{Question}[section]
\newtheorem{answer}{Answer}[section]

\makeatletter
\let\c@conjecture=\c@theorem
\let\c@corollary=\c@theorem
\let\c@proposition=\c@theorem
\let\c@lemma=\c@theorem
\let\c@definition=\c@theorem
\let\c@example=\c@theorem
\let\c@remark=\c@theorem
\let\c@equation=\c@theorem
\let\c@question=\c@theorem
\let\c@fact=\c@theorem
\let\c@problem=\c@theorem
\makeatother

\def\makeautorefname#1#2{\expandafter\def\csname#1autorefname\endcsname{#2}}
\makeautorefname{proposition}{Proposition}
\makeautorefname{lemma}{Lemma}
\makeautorefname{definition}{Definition}
\makeautorefname{example}{Example}
\makeautorefname{question}{Question}
\makeautorefname{conjecture}{Conjecture}
\makeautorefname{section}{Section}
\makeautorefname{claim}{Claim}
\makeautorefname{equation}{Equation}
\makeautorefname{remark}{Remark}
\makeautorefname{corollary}{Corollary}
\makeautorefname{fact}{Fact}
\makeautorefname{problem}{Problem}

\makeatletter
\def\@tocline#1#2#3#4#5#6#7{\relax
  \ifnum #1>\c@tocdepth 
  \else
    \par \addpenalty\@secpenalty\addvspace{#2}%
    \begingroup \hyphenpenalty\@M
    \@ifempty{#4}{%
      \@tempdima\csname r@tocindent\number#1\endcsname\relax
    }{%
      \@tempdima#4\relax
    }%
    \parindent\z@ \leftskip#3\relax \advance\leftskip\@tempdima\relax
    \rightskip\@pnumwidth plus4em \parfillskip-\@pnumwidth
    #5\leavevmode\hskip-\@tempdima
      \ifcase #1
       \or\or \hskip 1em \or \hskip 2em \else \hskip 3em \fi%
      #6\nobreak\relax
    \dotfill\hbox to\@pnumwidth{\@tocpagenum{#7}}\par
    \nobreak
    \endgroup
  \fi}
\makeatother

\newcommand{\hra}{\hookrightarrow}
\newcommand{\wt}{\widetilde}
\newcommand{\eps}{\varepsilon}
\newcommand{\Da}{D_\alpha}
\newcommand{\Db}{D_\beta}
\newcommand{\Z}{\mathbb{Z}}

\newcommand{\R}{\mathbb{R}}
\newcommand{\Tcal}{\mathcal{T}}

\newcommand{\Dcal}{\mathcal{D}}

\newcommand{\T}{\overline{T}}

\newcommand{\interior}{\textnormal{int}}

\title{Bridge position of $3$-manifolds embedded in the $5$-sphere}

\author[Aranda]{Rom\'an Aranda}
\address{Department of Mathematics, University of Nebraska-Lincoln, Lincoln, NE 68588}
\email{\href{romanaranda123@gmail.com}{romanaranda123@gmail.com}}

\author[Blackwell]{Sarah Blackwell}
\address{Department of Mathematics, University of Virginia, Charlottesville, VA 22902}
\email{\href{mailto:blackwell@virginia.edu}{blackwell@virginia.edu}}

\author[Kim]{Geunyoung Kim}
\address{Department of Mathematics and Statistics\\McMaster University\\Hamilton, ON L8S 4L8}
\email{\href{kimg68@mcmcaster.ca}{kimg68@mcmaster.ca}}

\author[Naylor]{Patrick Naylor}
\address{Department of Mathematics and Statistics\\McMaster University\\Hamilton, ON L8S 4L8}
\email{\href{patrick.naylor@mcmaster.ca}{patrick.naylor@mcmaster.ca}}

\author[Pongtanapaisan]{Puttipong Pongtanapaisan}
\address{Pitzer College, Claremont, CA 91711}
\email{\href{puttip@pitzer.edu}{puttip@pitzer.edu}}

\begin{document}

\begin{abstract}
    We introduce and study bridge decompositions for $3$–manifolds embedded in the $5$–sphere. These generalize both the classical notion of bridge position for knots in the $3$–sphere and the bridge trisections of surfaces in the $4$–sphere due to Meier and Zupan. Our main technical tool is the multisections of $5$–manifolds introduced by Aribi, Courte, Golla, and Moussard. We prove that every embedded $3$–manifold admits such a decomposition; in particular, any such embedding is encoded by four trivial tangle diagrams. We also present a range of explicit examples, including $S^2$-spun knots and ribbon $3$-knots.
\end{abstract}
\maketitle
 
\vspace{-1cm}
\section{Introduction}
\thispagestyle{empty}

A central challenge in the study of high-dimensional knotting lies in finding effective diagrammatic descriptions of embeddings of manifolds. In classical knot theory, knots and links in $S^3$ are encoded by planar diagrams, and many invariants and constructions can be computed directly from such a depiction. However, in higher dimensions, similar diagrammatic descriptions are much harder to obtain. A natural goal is therefore to represent embeddings of higher-dimensional manifolds using collections of lower-dimensional diagrams.

In this paper, we develop such a description for embeddings of $3$-manifolds in the $5$-sphere. Our approach encodes an embedded $3$-manifold using a collection of four simple tangle diagrams satisfying certain compatibility conditions. Roughly speaking, the data consists of four trivial tangle diagrams with the property that any pair determines an unlink and any triple determines a bridge trisection of a collection of unknotted $2$-spheres in $S^4$. We call such a representation a \emph{bridge quadrisection diagram}; see \autoref{fig:RP3_4plane} for an example. These diagrams provide a combinatorial description of embeddings of $3$-manifolds in $S^5$ analogous to the role played by planar diagrams in classical knot theory, or bridge trisection diagrams in knotted surface theory. 

\begin{figure}[ht]
    \centering
    \includegraphics[width=.7\textwidth]{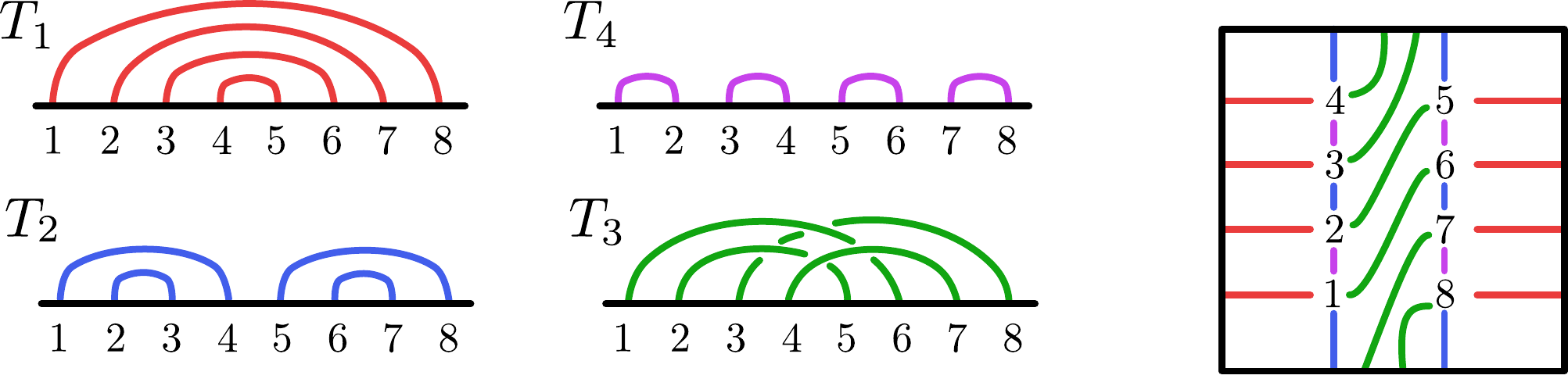}
    \caption{Left: a 4-plane diagram encoding a bridge 4-section of an embedded torus in $S^4$. In fact, the tuple also describes a bridge quadrisected embedding of $\mathbb{RP}^3$ in $S^5$, as described in \autoref{ex:RP3_explained}. Right: an embedding of the graph $\textcolor{red}{T_1}\cup\textcolor{blue}{T_2}\cup\textcolor{ForestGreen}{T_3}\cup\textcolor{purple}{T_4}$ 
    in a torus. The curves $\textcolor{red}{T_1}\cup\textcolor{ForestGreen}{T_3}$ and $\textcolor{blue}{T_2}\cup\textcolor{purple}{T_4}$ form a Heegaard diagram for $\mathbb{RP}^3$ (with some duplication of curves).}
    \label{fig:RP3_4plane}
\end{figure}

The geometric framework underlying these diagrams comes from decompositions of manifolds into simple pieces. \emph{Trisections}, introduced by Gay and Kirby \cite{gay16}, decompose a smooth, closed, orientable $4$-manifold into three $4$-dimensional $1$-handlebodies with simple intersections. Trisections provide a diagrammatic description of $4$-manifolds and have since been generalized in several ways: to decompositions with more than three pieces, called \emph{multisections} or \emph{$n$-sections} \cite{islambouli20}; to higher-dimensional PL manifolds \cite{rubinsteintillmann2020}; and to non-orientable $4$-manifolds \cite{millernaylor2024}. Trisection theory has also proved particularly useful for studying knotted surfaces. Meier and Zupan \cite{meier17,meier18} introduced the notion of \emph{bridge position} for a knotted surface $S\subset S^4$ (or an arbitrary $4$-manifold), producing a \emph{bridge trisection} in which the surface decomposes into three trivial disk systems subordinate to the trivial trisection of $S^4$. These decompositions yield diagrammatic descriptions of knotted surfaces and lead to effective diagrammatic moves relating different representations.

Our work extends this perspective to embeddings of $3$-manifolds in $S^5$. To do so, we make use of the notion of a \emph{quadrisection} of a closed, orientable $5$-manifold introduced by Aribi, Courte, Golla, and Moussard \cite{aribi23}. A quadrisection decomposes a $5$-manifold into four $5$-dimensional $1$-handlebodies so that any intersection of the pieces is again a $1$-handlebody, except for the total intersection, which is a surface. An important feature of quadrisections is that a quadrisected $5$-manifold admits a \emph{quadrisection diagram}, and conversely such a diagram uniquely determines a smooth $5$-manifold equipped with a quadrisection. 

Using the trivial quadrisection of $S^5$, we introduce the corresponding decompositions of embedded $3$-manifolds: a \emph{bridge quadrisection} of an embedded $3$-manifold $Y\subset S^5$ is a decomposition of $Y$ into boundary-parallel disk systems organized with respect to the trivial quadrisection of $S^5$. Our main structural result is an existence theorem for these decompositions, which leads to the tangle diagrams described above.

\newtheorem*{existence}{\autoref{thm:existance_bridge_4sections}}
\begin{existence}
    Every embedded 3-manifold $Y\subset S^5$ admits a bridge quadrisection.  
\end{existence}

We construct diagrams for a range of embeddings, including simple embeddings of lens spaces, spun knotted spheres, and ribbon $3$-spheres, and in \autoref{sec:unknotted_examples} and \autoref{sec:knotted_examples}, we describe techniques for finding these diagrams in practice. One benefit of our approach is that it provides a relatively straightforward way to generate such diagrams. Thus, this complementary perspective to existing techniques (e.g., \cite{Lomonaco81_5dim_knot_theory}) is a potentially fruitful way of producing interesting examples. It reduces, in principle, the problem of generating $3$-knots in $S^5$ to finding four tangles satisfying certain combinatorial conditions. 

A natural question is how different diagrams representing the same embedding are related. In dimension four, Hughes, Kim, and Miller showed that bridge trisections of isotopic knotted surfaces in 4-manifolds are related by sequences of perturbations and deperturbations \cite{hughes_kim_miller20}. In \autoref{subsec:uniqueness_bridge_3mans}, we introduce analogous moves for bridge quadrisections and formulate a corresponding conjectural uniqueness statement.

\newtheorem*{uniqueness}{\autoref{conj:uniqueness_bridge_3mans}}
\begin{uniqueness}
    Any two 4-plane diagrams describing isotopic 3-manifolds in $S^5$ are related by a finite sequence of interior Reidemeister moves, mutual braid moves, and 3-manifold perturbations. 
\end{uniqueness}

Beyond providing a diagrammatic framework for studying embeddings of $3$-manifolds in $S^5$, bridge quadrisections also allow effective computation of invariants directly from diagrams. In \autoref{subsec:branched_covers}, we show how to adapt work of Cahn, Mati\'c, and Ruppik \cite{cahn23} to compute invariants arising from branched covers. Code implementing these computations is available on GitHub \cite{github_link}.

\begin{remark}\label{rem:general_quadrisections}
The curious reader may wonder whether bridge quadrisections for 3-manifolds in arbitrary 5-manifolds exist, just as 2-manifolds can be bridge trisected inside any trisected 4-manifold~\cite{meier18}. This statement is in production by Courte, Moussard, Ren, and Zhou~\cite{Moussard_personal}. Here, we focus on the case of the 5-sphere, but we believe that one should be able to use the ideas in this work to prove a more general existence theorem by upgrading \autoref{lem:uniqueness_disk_fillings} and \autoref{lem:abstract_perturbations_lift} appropriately.
\end{remark}

\subsection*{Organization} The paper is organized as follows. In \autoref{sec:decompositions_of_manifolds}, we briefly review bridge multisections of surfaces, and define bridge quadrisections of 3-manifolds in $S^5$. In \autoref{sec:calculus_quadrisected_surfaces}, we discuss some important properties of perturbations which will be necessary for the upcoming proof. In \autoref{sec:existence}, we prove \autoref{thm:existance_bridge_4sections}. In \autoref{sec:unknotted_examples} and \autoref{sec:knotted_examples}, we present a variety of examples and techniques for finding bridge quadrisection diagrams. In \autoref{sec:future_directions}, we discuss computations of invariants from these diagrams, as well as some questions. Finally, in \autoref{sec:appendix}, we include further details about our Sage code.

\subsection*{Acknowledgments} 
We thank Sylvain Courte, Daniel Hartman, Slava Krushkal, Jeffrey Meier, Delphine Moussard, Qiuyu Ren, Xiaozhou Zhou, and Alexander Zupan for helpful discussions. While this project did not originate at a workshop, the authors acknowledge many years of participation in Trisectors Workshops, which brought them together to explore the ideas presented in this paper, and thank Jeffrey Meier, Maggie Miller, Laura Starkston, and Alexander Zupan for organizing such wonderful collaboration spaces.

RA and PP were partially supported by an AMS-Simons Travel Grant. SB was supported by the NSF Postdoctoral Research Fellowship DMS-2303143. PN was supported by an NSERC Discovery Grant and a CRM-Simons Scholar Grant. 

\tableofcontents
 \makeatletter
\providecommand\@dotsep{5}

\section{Decompositions of low-dimensional manifolds}\label{sec:decompositions_of_manifolds}

We will work in the smooth category throughout. Manifolds are compact and connected unless stated otherwise, but not necessarily orientable. We reserve the term \emph{$n$-knot} for a embedding $S^n\hookrightarrow S^{n+2}$. 

\subsection{Trivial tangles}\label{subsec:trivial_tangles}

A \emph{trivial tangle} is a pair $(B^3,T)$, or simply $T\subset B^3$, where $T$ is a collection of properly embedded arcs such that, fixing the endpoints of $T$, we may isotope $T$ into $\partial B^3$. Every link $L$ in $S^3$ can be written as the union of two trivial tangles along their boundary, i.e., $(S^3,L)=(B_1, T_1)\cup (B_2,T_2)$. Such a decomposition is called a \emph{$b$-bridge splitting} of $L$ if each $T_i$ has $b$ components. The symbol $\T$ will denote the mirror image of a trivial tangle $T$.

The higher-dimensional version of a trivial tangle is called a \emph{trivial disk system}. A collection of properly embedded $(n-2)$-dimensional disks $\Dcal\subset B^n$ is called a \emph{$c$-patch trivial $(n-2)$-disk system} if $\Dcal$ is boundary parallel (rel. boundary) and has $c$ connected components. The boundary of a trivial $(n-2)$-disk system is an unlink of $(n-3)$-spheres in $S^{n-1}$. The following results, due to Livingston for $n=4$ and Powell for $n=5$, guarantee that trivial disk systems are unique up to isotopy.

\begin{theorem}[\cite{Liv82,powell24}]\label{lem:uniqueness_disk_fillings}
    Let $n=4$ or $n=5$ and let $\Dcal_1$ and $\Dcal_2$ be two trivial $(n-2)$-disk systems in $B^n$. If $\partial\Dcal_1=\partial \Dcal_2$, then $\Dcal_1$ is isotopic (rel. boundary) to $\Dcal_2$. 
\end{theorem}

The rest of this section will discuss higher-dimensional analogs of bridge position for links in $S^3$. These notions can be defined for embeddings of closed $(n-2)$-manifolds in arbitrary smooth closed $n$-manifolds, but we only include the case of embeddings in $n$-spheres, which is what we will need.

\subsection{Surfaces in \texorpdfstring{$4$}{4}-space}\label{subsec:surfaces_in_4-space}

We say that an orientable surface $F\subset S^4$ is \emph{unknotted} if it can be embedded in the equatorial $S^3\subset S^4$. An \emph{unlink of $2$-knots} is one which bounds a collection of embedded $3$-balls in $S^4$ with pairwise disjoint interiors. 

\begin{definition}[\cite{islambouli2022toric}]\label{def:multisections_surfaces}
    Let $\Sigma$ be an embedded surface in $S^4$ and let $m\geq 3$. A \emph{$(b;c)$-bridge $m$-section} of $\Sigma$, where $c=(c_1,\dots,c_m)$, is a decomposition
    \[ 
    (S^4,\Sigma)=(X_1, D_1)\cup (X_2, D_2)\cup \dots\cup (X_m, D_m) , 
    \]
    such that for each $i$ (with indices taken modulo $m$), 
    \begin{enumerate}
        \item $D_i=\Sigma\cap X_i$ is a $c_i$-patch trivial 2-disk system in a 4-ball $X_i$, 
        \item $(B_i, T_i)=(X_i,D_i)\cap (X_{i-1},D_{i-1})$ is a $b$-bridge trivial tangle inside a 3-ball, and 
        \item $\bigcap_{i=1}^m (X_i,D_i)$ is a 2-sphere with $2b$-punctures. 
    \end{enumerate}
\end{definition}

The decomposition of $S^4=\bigcup_{i=1}^m X_i$ into $m$ 4-balls is called a \emph{genus-zero} $m$-section of $S^4$~\cite{islambouli20}; it may be obtained by mapping $S^4$ to an $m$-sected 2-disk and pulling back the pieces. See \autoref{def:sphere_decomposition} for a more general description of this decomposition. The quantity $b$ is called the \emph{bridge number} of the decomposition. Consecutive tangles (mod $m$) glue to $c_i$-component unlinks $T_i\cup \T_{i+1}$. We will mainly work with bridge $m$-multisections with $m=3$ or $m=4$, which we call (four dimensional) \emph{bridge trisections} and \emph{bridge quadrisections}, respectively. 
For an arbitrary bridge quadrisection, the links $T_1\cup\T_3$ and $T_2\cup\T_4$ may not be trivial. The tuple of tangles $\Tcal=(T_1,T_2,T_3,T_4)$ is called \emph{the spine} or \emph{$4$-plane diagram} of $\Sigma$, and by an application of \autoref{lem:uniqueness_disk_fillings} this data is enough to determine the bridge $m$-sected surface~\cite{meier17}. 

\subsection{\texorpdfstring{$3$}{3}-manifolds in \texorpdfstring{$5$}{5}-space}\label{subsec:3-manifolds_in_5-space}

Now we turn to the 5-dimensional setting. First, we need a standard decomposition of the $n$-sphere.

\begin{definition}[\cite{aribi23}]\label{def:sphere_decomposition}
    For $n>2$, let $S^n\subset \mathbb{R}^{n+1}$, and let $p:S^n\to \mathbb{R}^{n-2}$ be the map which forgets the last three coordinates. Identify the unit $(n-2)$-ball $\text{Im}(p)$ with the interior of an $n-2$ simplex $\Delta$, and take a subdivision of $\Delta$ (and hence $\text{Im}(p)$) into $n-1$ pieces induced by the cone on $\partial \Delta$. The \emph{genus zero multisection} of $S^n$ is the decomposition of $S^n$ obtained by pulling back these pieces using $p$.
\end{definition}

When $n=3$, this yields a genus-zero Heegaard splitting of $S^3$. When $n=4$, this yields the \emph{genus zero trisection} of $S^4$, and when $n=5$, the \emph{genus zero quadrisection} of $S^5$. See \cite[Example 2.4]{aribi23} for more details. 

\begin{figure}[ht]
    \centering
    \includegraphics[width=.5\textwidth]{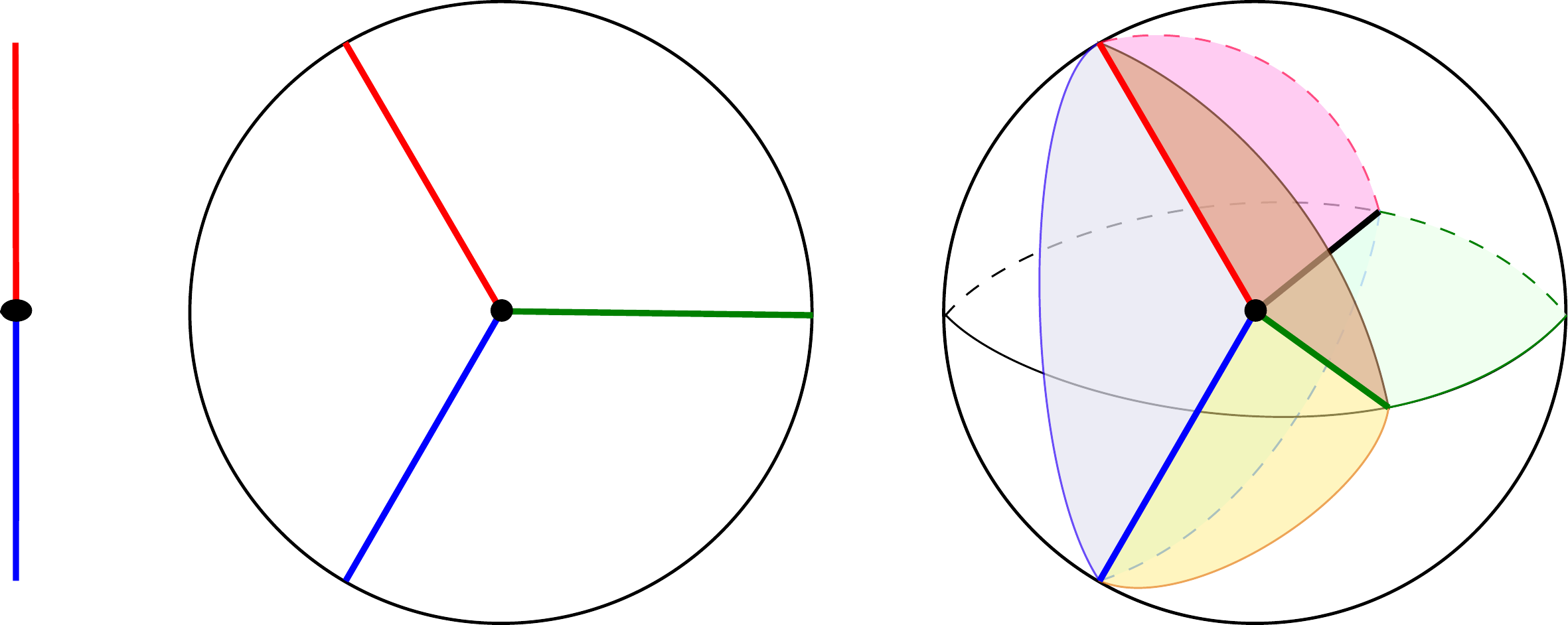}
    \caption{Decompositions of the $(n-2)$-dimensional ball into $(n-1)$ pieces.}
    \label{fig:simplicial_model}
\end{figure}

\begin{remark}
    An important subtlety of these decompositions is that different angles at the codimension zero, cornered submanifolds used in a decomposition of $\Delta$ may yield non-diffeomorphic decompositions of $S^n$; see  \cite[Figure 3]{aribi23}. Theorem 3.2 of \cite{aribi23} implies that for manifolds of dimension at most $5$, different choices of angles on each cornered piece yield diffeomorphic spaces.
\end{remark}

\begin{definition}\label{def:bridge_4section_3man}\label{def:multisections_3mans}
    Let $Y$ be a 3-manifold embedded in $S^5$. A $b$-bridge quadrisection of $Y$ is a decomposition
    \[ 
    (S^5,Y)=(W_1, E_1)\cup (W_2, E_2)\cup (W_3, E_3)\cup (W_4, E_4) , 
    \]
    where $S^5=\bigcup_{i=1}^4 W_i$ is a genus zero quadrisection, and for any permutation $\{i,j,k,\ell\}$ of $\{1,2,3,4\}$,
    \begin{enumerate}
    \item $E_i=Y\cap W_i$ is a $s_i$-patch trivial 3-disk system in the 5-ball $W_i$, 
    \item $(X_{ij}, D_{ij})=(W_i,E_i)\cap (W_{j},E_{j})$ is a $c_{ij}$-patch trivial 2-disk system in the 4-ball $X_{ij}$,
    \item $(B_{\ell}, T_{\ell})=(W_i,E_i)\cap (W_{j},E_{j})\cap (W_{k},E_{k})$ is a $b$-bridge trivial tangle in the 3-ball $B_\ell$,  
    \item $\bigcap_{i=1}^4 (W_i,E_i)$ is a 2-sphere with $2b$ marked points. 
    \end{enumerate}
\end{definition}

If necessary, we will refer to such a decomposition as a $(b;c,s)$-bridge quadrisection of $Y$, where $c=\{c_{ij}:i\neq j\}$ and $s=\{s_i\}$. The quantity $b$ is called the \emph{bridge number} of the quadrisection. Note that in a bridge quadrisection, the boundary of each 3-ball system $E_i$ is an unlink of 2-knots equipped with the bridge trisection 
\[
(\partial W_i, \partial E_i)=
(X_{ij},D_{ij})\cup (X_{ik},D_{ik})\cup (X_{i\ell},D_{i\ell}).
\]
The \emph{spine} of a bridge quadrisection of $Y$ is the tuple of tangles $(T_1,T_2,T_3,T_4)$. Note that any permutation of the spine of a quadrisected 3-manifold is also a 4-plane diagram for some knotted surface in $S^4$; different permutations may lead to non-isotopic surfaces. The following lemma states that the spine of a bridge quadrisected 3-manifold $Y$ determines the embedding $Y\hookrightarrow S^5$.

\begin{lemma}\label{lem:spine_determines_3man}
    Let $S^5=\bigcup_{i=1}^4 W_i$ be a genus-zero quadrisection of $S^5$ with intersections labeled as in \autoref{def:multisections_3mans}. Suppose that $(T_1,T_2,T_3,T_4)$ is a tuple of tangles, $T_\ell\subset B_\ell$, such that 
    for any three-element subset $\{i,j,k\}\subset \{1,2,3,4\}$, the tuple $(T_i,T_j,T_k)$ is a triplane diagram for an unlink of $2$-knots. 
    Then up to isotopy, there is a unique embedding $Y\hookrightarrow S^5$ of a 3-manifold $Y$ which intersects the splitting $S^5=\bigcup_{i=1}^4 W_i$ in a bridge quadrisection with spine $(T_1,T_2,T_3,T_4)$.  
\end{lemma}

\begin{proof}
Using the notation in \autoref{def:multisections_3mans}, the union $B_i\cup \overline{B}_j$ is a genus-zero Heegaard splitting for $\partial X_{k\ell}$. By \autoref{lem:uniqueness_disk_fillings}, the unlink $T_i\cup \T_j$ bounds a unique 2-disk system $D_{ij}$ inside $X_{ij}$. 
One dimension higher, the union $(X_{ij},D_{ij})\cup (X_{ik},D_{ik})\cup (X_{i\ell}, D_{i\ell})$ is a bridge trisection of a surface $F_i$ in $\partial W_i$. Note that the spine of such a bridge trisection is the tuple $(T_\ell,T_j,T_k)$. Thus, $F_i$ is an unlink of 2-knots in $\partial W_i$. By \autoref{lem:uniqueness_disk_fillings} (in the case $n=5$), there is a unique 3-disk system $E_i$ in $W_i$. Thus, we have built a unique (up to isotopy) bridge quadrisected 3-manifold $Y=\bigcup_{i=1}^4 E_i$, as desired. 
\end{proof}

\subsection{Heegaard splittings from bridge quadrisections}\label{subsec:heegaard_splittings}

Recall that a \emph{Heegaard splitting} of a closed $3$-manifold $Y$ is a decomposition $Y=H_1\cup_\Sigma H_2$ into $3$-dimensional handlebodies with common boundary a closed surface $\Sigma$. If $Y$ is orientable and connected, the surface $\Sigma$ is orientable. 
If $Y$ is non-orientable and connected, then the surface $\Sigma$ is homeomorphic to the connected sum of an even number of projective planes. A \emph{Heegaard diagram} is a tuple $(\Sigma;\alpha, \beta)$ where $\alpha$ and $\beta$ are sets of pairwise disjoint simple closed curves bounding disks in $H_1$ and $H_2$, respectively, such that $\Sigma\setminus \alpha$ and $\Sigma\setminus \beta$ are planar surfaces. At times, we will make use of Heegaard diagrams given by a collection of curves which are homologically dependent, but still prescribe a compression body, and we will call such a diagram an \emph{extended Heegaard diagram} if needed. 

The following result explains how to extract Heegaard splittings from bridge quadrisected knotted 3-manifolds in $S^5$. See \autoref{ex:RP3_explained} for an implementation of this process. This result also motivated the notion of bridge position for embeddings of Heegaard splittings into $S^5$ in \autoref{def:heegaard_complex_bridge_position}, which is an essential ingredient in the proof of existence of bridge quadrisections.

\begin{proposition}\label{prop:Heegaard_splitting_from_4section}
    Let $(S^5,Y^3)=\bigcup_{i=1}^4(W_i, E_i)$ be a bridge quadrisected \emph{connected} $3$-manifold as in \autoref{def:bridge_4section_3man}. For any permutation $\{i,j,k,\ell\}$ of $\{1,2,3,4\}$, there is a Heegaard splitting 
    \[
    Y^3=\left(E_i\cup E_k\right) \cup_\Sigma \left(E_j\cup E_\ell\right),
    \] 
    where $\Sigma\subset S^4$ is a Heegaard surface described by the $4$-plane diagram $(T_i, T_j, T_k, T_\ell)$. In particular, the tuple $(\Sigma;T_i\cup \T_k, T_j\cup \T_\ell)$ is an extended Heegaard diagram of $Y$.
\end{proposition}

\begin{proof}
    For distinct $a,b\in \{1,2,3,4\}$, $E_a\cup E_b$ is the union of 3-dimensional balls along $E_a\cap E_b=D_{ab}$ , a set of disjoint disks on their boundary. Hence, $E_a\cup E_b$ is a 3-dimensional handlebody. 
\end{proof}

If $Y$ is disconnected, then each component of $Y$ inherits a Heegaard splitting. If a component of $Y$ is non-orientable, the handlebodies and the Heegaard surface $\Sigma$ are non-orientable.

\begin{example}[Embedding $\mathbb{RP}^3$ in $S^5$]\label{ex:RP3_explained}
The left image of \autoref{fig:RP3_4plane} shows a 4-plane diagram $(T_1,T_2,T_3,T_4)$ with the property that any three tangles form a triplane diagram for an unlink of 2-spheres.\footnote{This fact is left as an exercise.} Thus, by \autoref{lem:spine_determines_3man}, it determines an embedding of a 3-manifold in $S^5$. To determine the homeomorphism type of $Y$, we invoke \autoref{prop:Heegaard_splitting_from_4section} as follows: we first consider the 4-colored graph $\Gamma=T_1\cup T_2\cup T_3\cup T_4$, with eight vertices and twelve edges equal to the strands of the four tangles. Then we embed $\Gamma$ in a closed surface $F$ such that $F\setminus \Gamma$ is a disjoint union of bicolored polygons as in the right side of \autoref{fig:RP3_4plane}, where the boundary of each polygon lies in $T_i\cup T_{i+1}$. \autoref{prop:Heegaard_splitting_from_4section} states that the collections of curves $T_1\cup T_3$ and $T_2\cup T_4$ form a Heegaard diagram for $Y$. In this case, $Y$ is homeomorphic to $\mathbb{RP}^3$.
\end{example}

\section{Calculi of quadrisected surfaces}\label{sec:calculus_quadrisected_surfaces}

In the upcoming proof in \autoref{sec:existence}, it will be convenient to place the surface of a Heegaard splitting of an embedded 3-manifold into a kind of \emph{bridge position} (\autoref{def:heegaard_complex_bridge_position}). Having done so, we will need to perform various kinds of stabilization operations on quadrisected surfaces.  In this section, we develop language for modifying quadrisection diagrams of surfaces in $S^4$ as well as quadrisections of abstract surfaces (i.e., not necessarily embedded in $S^4$).

\subsection{Surfaces in \texorpdfstring{$S^4$}{the 4-sphere}}\label{sec:calculus_embedded_surfaces}

Classically, band surgery on a link $L$ in $S^3$ is a modification that attaches a band to $L$, and replaces two segments of the link with the other two edges of the band. For bridge multisections, \emph{band surgeries} are modifications that increase the bridge number by one, while changing the quadrisected surface in a controlled way~\cite[\S 4]{aranda24}. In this paper, we will work with band surgeries that either preserve the surfaces isotopy class or add a 1-handle. 
At the level of 4-plane diagrams, band surgery changes each tangle $T$ of $\Tcal$ in one of the following two ways.

\begin{itemize}
    \item Type 0: Adding a small one-bridge strand near a point in the boundary sphere.
    \item Type 1: Dragging a strand of $T$ towards the boundary sphere and breaking the strand in two. 
\end{itemize}

The dragging effect of a type 1 modification of $T$ can be codified with a band $\rho$ that has one side in $\interior(T)$ and the opposite side in the boundary sphere. For a 4-plane diagram, we ensure that the endpoints of the new strands are the same across the tangles, so we can glue them together. In particular, if $T_1$ and $T_2$ are modified with different types of modifications, the union of the resulting tangles $T'_1\cup \T'_2$ is isotopic to $T_1\cup \T_2$. If both tangles suffered a type 0 modification, $T_1\cup \T_2$ gained a one-bridge unknot. And, if both modifications were of type 1, then $T'_1\cup \T'_2$ is the result of a classical band surgery on $T_1\cup \T_2$. In \autoref{fig:Kinoshita_Terasaka_2knot}, we see how subsets of the framed bands describe band surgeries on 4-plane diagrams. 

\subsubsection{Perturbations}\label{sec:perturbations}

They correspond to isotopies of a quadrisected surface $(S^4,F)=\bigcup_{i=1}^4 (X_i,D_i)$ that drag parts of disks in $D_i\cup \dots \cup D_{i+k}$ to the rest of the 4-dimensional sectors $X_j$. At the level of 4-plane diagrams, there are three kinds of perturbations of bridge quadrisections, depending on the number of pairs of consecutive tangles suffering a modification of type 1. 

Let $\Tcal=(T_1,T_2,T_3,T_4)$ be a 4-plane diagram. Fix $0\leq k<3$ and $i\in \Z_4$. Let $\rho_i, \dots, \rho_{i+k}$ be bands inducing modifications of type 1 in the tangles $T_i, \dots, T_{i+k}$, respectively. Let $\Tcal'=(T_1',T'_2,T'_3,T'_4)$ be the tuple obtained by performing type 1 modifications on the tangles $T_j$ for $i\leq j \leq j+k$, and type 0 modifications on the rest. 

\begin{lemma}[{{\cite[Lemma 5.1]{aranda24}}}]\label{lem:sector-perturbation}
    Suppose that for each $i\leq j\leq i+k-1$, there is a $2$-sphere $S_j$ in $B_j\cup \overline{B}_{j+1}$ that intersects $\partial B_j=\partial B_{j+1}$ in one loop, and contains both the band $\rho_j$ and exactly one component of $T_j\cup \T_{j+1}$.
    Then $\Tcal'$ represents the same surface as $\Tcal$. We call $\Tcal'$ a \emph{$k$-sector perturbation} of $\Tcal$.
\end{lemma}

The sphere condition ensures that the new link $T'_j\cup T'_{j+1}$, obtained by classical link band surgery along $\rho_j\cup \rho_{j+1}$, is an unlink with one more component than $T_j\cup \T_{j+1}$. The sphere condition may be overwhelming to verify in practice.   
An alternative description is for the tangles and type 1 bands $T_{j}$, $\rho_j$, $T_{j+1}$, and $\rho_{j+1}$ to be isotopic to the diagram as in \autoref{fig:sphere_condition} after interior Reidemeister moves and mutual braid moves. These resemble the original definition of perturbation in \cite[Fig 27]{meier17} more closely. In our examples, we will choose $\rho$-arcs near punctures to make sure this condition is easily verifiable.

\begin{figure}[ht]
    \centering
    \includegraphics[width=.4\textwidth]{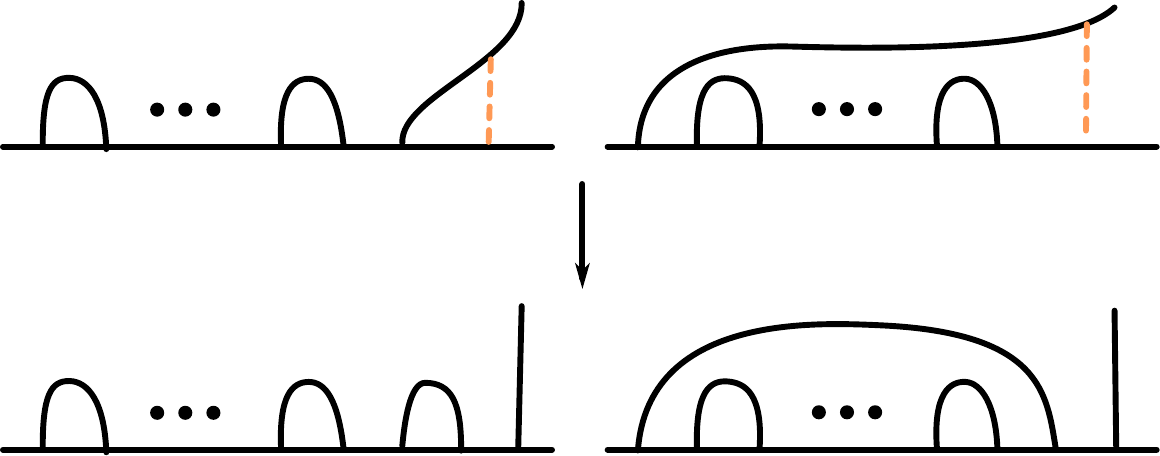}
    \caption{Top: the sphere condition in \autoref{lem:sector-perturbation} and \autoref{lem:3man_perturbation} is equivalent to each pair of tangles with type 1 bands looking as follows, after braid moves and interior Reidemeister moves. Bottom: the tangles after band surgery.}
    \label{fig:sphere_condition}
\end{figure}

Sometimes we will want to keep track of which tangles are being popped (type 0) or dragged (type 1). We will write \emph{$I$-perturbations} to refer to $(|I|-1)$-sector perturbations where the tangles $T_i$ ($i\in I$) are those suffering a type 1 modification. For instance, the 2-sector perturbations in panels (B-C) and (D-E) in \autoref{fig:double_ribbon_lemma_1} are a 123-perturbation and a 341-perturbation, respectively. 

For context, 1-sector perturbations were first introduced by Meier and Zupan in \cite{meier17} using a local model for tangles being dragged. The language of band surgeries was introduced by the first author and Engelhardt to account for multiple sectors in \cite{aranda24}; though a local model version should be achievable for $k\geq 2$. On the other hand, 0-sector sector perturbations may not always be explained with a local model, as they may not yield trivial tangles; see \autoref{remark:0-sector_warning}. 

\begin{remark}[A warning about 0-sector perturbations]\label{remark:0-sector_warning}
The result of a 0-sector perturbation on a 4-plane diagram may not be strictly a 4-plane diagram. The reason for this is that the dragged tangle $T'_j$ may not be trivial, as the band $\rho_j$ may create some local knotting in $T_j$. That said, one can see that the tuple $\Tcal'$ still satisfies the condition that each consecutive union $T'_i\cup \T'_{i+1}$ is an unlink. Thus, the tuple $\Tcal'$ still determines an embedding of the same surface, and \autoref{lem:sector-perturbation} still holds. This remark is relevant to this work, as some intermediate tuples obtained in our processes may not be a 4-plane diagram in the strict sense. If the reader desires to work with real 4-plane diagrams, they may need to perform more 0-sector perturbations to get rid of the local minima of $T'_j$. 
\end{remark}

\subsubsection{Tubings}\label{sec:tubings}

For surfaces in $S^4$, \emph{$1$-handle addition} is the result of replacing two small disks in the surface with a thin tube connecting them~\cite{book_knotted_surfaces}. Precisely, if $D^2\times I$ is an embedded 1-handle for $F\subset S^4$ with $F\cap (D^2\times I)=D^2\times \{0,1\}$, 1-handle addition of $F$ is the new surface $F'=F\setminus(D^2\times \{0,1\})\cup (S^1\times I)$. We refer to the core of the 1-handle $t=\{0\}\times I$ as the \emph{guiding arc} of the 1-handle. In this paper, we will exploit 1-handle additions that lie in the spine of the genus-zero quadrisection of $S^4$. 

\begin{lemma}[Examples 4.8 and 4.13 of \cite{aranda24}]
Let $\Tcal=(T_1,T_2,T_3,T_4)$ be a $4$-plane diagram representing a surface $F\subset S^4$. Fix $j\in \Z_4$ and let $\rho_j$ and $\rho_{j+2}$ be bands inducing modifications of type 1 in the tangles $T_j$ and $T_{j+2}$. Let $\Tcal'$ be the tuple obtained by performing type 1 modifications on the tangles $T_j$ and $T_{j+2}$ and type 0 modifications on the rest. Then the tuple $\Tcal'$ represents a surface $F'\subset S^4$ that is obtained by a $1$-handle addition to $F$. Moreover, the guiding arc of the $1$-handle is equal to the core of the band $\rho_j\cup \rho_{j+2}$. 
\end{lemma}

We will need the following important property.

\begin{lemma}\label{lem:tubing_4sec_with_bands}\label{lem:tubing_4sec}
    Let $\Tcal=(T_1,T_2,T_3,T_4)$ be a bridge $4$-section of $\Sigma\subset S^4$. Suppose that $t$ is a tubing arc with $t\subset B_1\cup \overline{B}_3$ and let $\rho$ be a framed band with core equal to $t$. Let $\Sigma'$ be the result of adding a $1$-handle to $\Sigma$ with guiding arc $t$. There is a bridge $4$-section $\Tcal'$ of $\Sigma'$ such that
    \begin{enumerate}
    	\item $L'_{13}$ is the result of band surgery of $L_{13}$ using the band $\rho$, and 
        \item $L'_{24}=L_{24}\cup \partial D$, where $D\subset B_2\cup \overline{B}_4$ is a set of pairwise disjoint disks away from $L_{24}$. 
    \end{enumerate}
    In fact, each disk in $D$ is a meridian of the tube $t$ with boundary a $1$-bridge unknot in $B_2\cup \overline{B}_4$. 
\end{lemma}
This follows from the proof of Proposition 6.1 in \cite{aranda24}, but for the reader's convenience, we include a sketch of the proof. 

\begin{proof}
    While fixing the tangles in $\Tcal$, isotope $t$ in $B_1\cup \overline{B}_3$ so that $t$ has endpoints in both tangles $T_1$ and $T_3$ and is transverse to the bridge sphere. An example of the resulting framed arc is shown in \autoref{fig:Kinoshita_Terasaka_2knot} (A). If $t$ crosses the bridge sphere exactly once, like in \autoref{fig:Kinoshita_Terasaka_2knot} (C), we can use $\rho$ to tube the bridge 4-sected surface as in \autoref{fig:Kinoshita_Terasaka_2knot} (D). If $t$ crosses the bridge sphere more than once, we can perform a 0-sector perturbation using a sub-band of $\rho\cap B_1$ with endpoint in $T_1$ as in \autoref{fig:Kinoshita_Terasaka_2knot} (A)-(B). In each of the steps above (0-perturbations and tubing), the link $T_2\cup \T_4$ gains a 1-bridge unknot bounding a meridian of the tube $t$; see \cite[Figure 13]{aranda24}.
\end{proof}

\begin{example}\label{ex:KT_double_as_tubes}
    The Kinoshita-Terasaka knot $11_{n42}$ bounds a ribbon disk $D_{KT}$ in $B^4$ shown as a banded presentation with a yellow band in $\textcolor{red}{T_1}\cup \textcolor{ForestGreen}{\T_3}$ as in \autoref{fig:Kinoshita_Terasaka_2knot} (A). Thus, the 2-knot $D_{KT}\cup \overline{D}_{KT}$ in $S^4$ is obtained by tubing a two-component unlink of spheres. \autoref{fig:Kinoshita_Terasaka_2knot} shows the process of drawing a 4-plane diagram for $D_{KT}\cup \overline{D}_{KT}$ described in \autoref{lem:tubing_4sec}.
\end{example}

\begin{figure}[ht]
    \centering
    \includegraphics[width=1\textwidth]{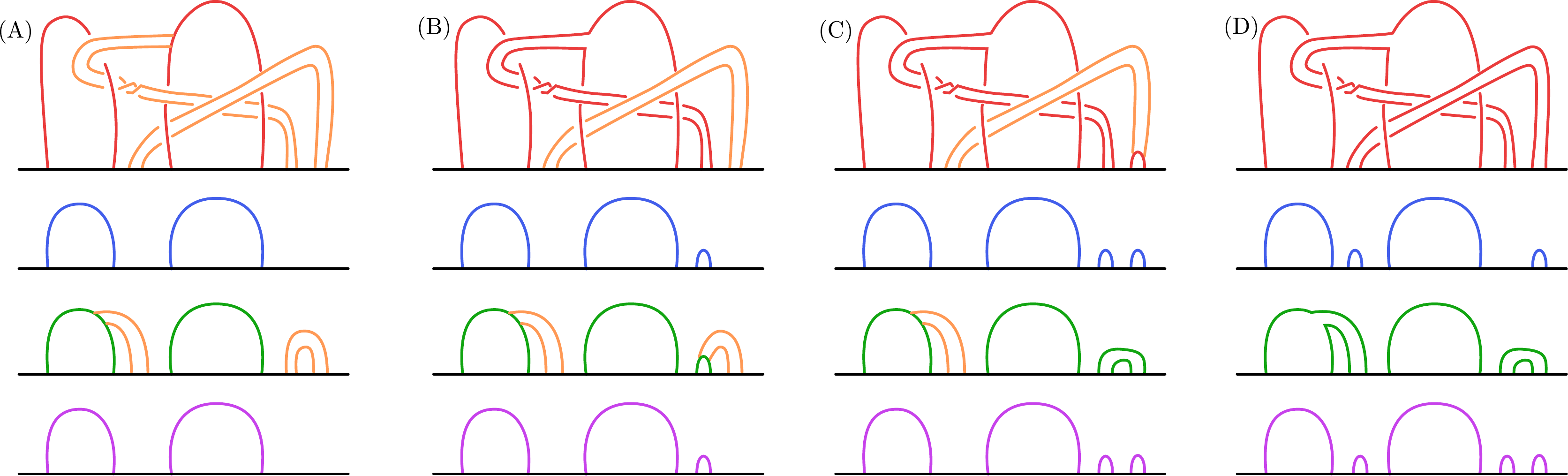}
    \caption{In (A), a 4-plane diagram of an unlink of 2-spheres; the yellow framed arc (or band) in $B_1\cup \overline{B}_3$ gives a ribbon presentation for the Kinoshita-Terasaka knot $11_{n_{42}}$. From (A) to (C), we perform 0-sector perturbations, as in \autoref{lem:sector-perturbation}, to shrink the yellow band while preserving the isotopy type of the bridge 4-sected surface. From (C) to (D), we tube as in \autoref{lem:tubing_4sec}. The result is a 4-plane diagram of the double of the ribbon disk bounding $11_{n_{42}}$.}
    \label{fig:Kinoshita_Terasaka_2knot}
\end{figure}

\subsection{Abstract surfaces}\label{sec:calculus_abstract_surfaces}

In this section, we develop terminology for working with curves on a Heegaard diagram of an abstract surface which is not necessarily embedded in $S^4$. 

\begin{definition}\label{def:multicurve}
    Let $S$ be a (possibly non-orientable) closed surface. A \emph{multicurve} is a collection of pairwise disjoint embedded loops in $S$ with annular neighborhoods. 
\end{definition}

Note that we allow multicurves to have trivial or parallel components. If $S$ is non-orientable, the definition of multicurve excludes cores of M\"obius bands inside $S$.

\begin{definition}\label{def:arc_surgery}
    Let $S$ be a surface, and suppose $x\subset S$ is a multicurve. Let $\rho\subset S$ be an embedded arc with interior disjoint from $x$ and both endpoints on $x$, i.e., $x\cap \rho=\partial \rho$. Let $u$ and $v$ be the components of $x$ connected by $\rho$. We define \emph{arc surgery} of $x$ along $\rho$, denoted by $x[\rho]$, to be the multicurve obtained by replacing the two boundary components $u$ and $v$ of $\partial \nu(u\cup v\cup \rho)$ in $x$ with the third boundary component. 
    
    In the case that $u$ and $v$ are equal, we define $x[\rho]$ to be the result of replacing the boundary component $u$ of $\partial\nu(u\cup \rho)$ with the two other boundary components. In this second case, we require that $\nu(u\cup \rho)$ is an orientable subsurface of $S$, so that this process produces two curves with annular neighborhoods. 
    See \autoref{fig:arc_surgery_model} for an illustration of these two cases.  
\end{definition}

\begin{figure}[ht]
    \centering
    \includegraphics[width=.3\textwidth]{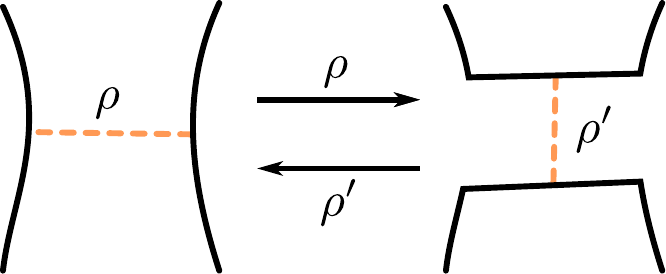}
    \caption{An illustration of the local model of an arc surgery.}
    \label{fig:arc_surgery_model}
\end{figure}

For example, in a pair of pants with boundaries $u$, $v$, and $w$, one can obtain $w$ from $u\cup v$ by one arc surgery along a seam. One can interpret arc surgery as a 2-dimensional 1-handle attachment to $\nu(u\cup v)$, and in particular, one arc surgery along $\rho$ can be undone by arc surgery along the cocore of $\nu(\rho)$; see \autoref{fig:arc_surgery_model}.

\begin{lemma}\label{lem:arc_surgery_nullhom}
    Let $x,y\subset S$ be two multicurves with $[x]=[y]=0$ in $H_1(S;\Z_2)$. There is a sequence of multicurves \[x=x_0, x_1, \dots, x_n=y\]
    such that $x_{i+1}$ is obtained from $x_i$ by one arc surgery.
\end{lemma}

\begin{proof}
    We will show that $x$ admits a sequence of arc surgeries to $\partial D$ where $D\subset S$ is an embedded disk. Since arc surgeries can be undone by more arc surgeries, the result will follow.

    Let $x\subset S$ be a nullhomologous multicurve. There exists a multicurve $x^*\subseteq x$ bounding a subsurface $S^*\subset S$ with interior disjoint from $x$. This surface may be abstractly built from a 0-handle and some number of (possibly non-orientable) 1-handles. Performing arc surgeries along the co-cores of these 1-handles converts $x^*$ to the boundary of the 0-handle, i.e., an embedded disk. 
    
    Note that since $S^*$ is connected, our procedure will end with exactly one trivial loop in $S$. We can repeat the above process until every curve in $x$ bounds a disk, and then merge these components. 
\end{proof}

In what follows, it will be convenient for us to work with an abstract surface $S$ instead of embedded surfaces $\Sigma\subset S^4$. To this end, we specialize our vocabulary of bridge 4-sections to abstract surfaces. We will include the word \emph{abstract} to differentiate between the two settings. 

\begin{definition}\label{def:abstract_4-section}
    An \emph{abstract $4$-section} of a surface $S$ is a finite, connected, 4-valent graph $\Gamma\subset S$ with the following properties. 
    \begin{enumerate}
        \item We have $\Gamma=\Gamma_1\cup \Gamma_2\cup \Gamma_3\cup \Gamma_4$, where each $\Gamma_i$ is a subgraph of  $\Gamma$, and every vertex is the endpoint of an edge in each $\Gamma_i$,
        \item the union $\Gamma_i\cup \Gamma_{i+1}$ of two consecutive collections of edges bounds a disjoint union of polygonal disks in $S$ with interior disjoint from $\Gamma$.
    \end{enumerate}
    If $I$ is a subset of $\{a,b,c,d\}$ with $1\leq |I|\leq 3$, we will write $\Gamma_I=\cup_{i\in I}\Gamma_i$, and occasionally treat these subgraphs as multicurves when appropriate. 
\end{definition}

The canonical example of an abstract 4-section of $S$ corresponds to the cell decomposition of $S$ induced by a bridge 4-section of an embedding $S \hra S^4$. In analogy with this case, we now define perturbations of abstract 4-sections of $S$. There are three kinds, depending on the number of subgraphs $\Gamma_i$ involved. 

\begin{definition}\label{def:abstract_permutation}
    Let $\Gamma$ be an abstract 4-section $S$. Let $(a,b,c,d)$ be some cyclic permutation of $(1,2,3,4)$. A \emph{perturbation} of $\Gamma$ is one of the following local modifications of $\Gamma$. 
    
    \begin{enumerate}
        \item[(1)] An \emph{$a$-perturbation} of $\Gamma$ is the result of modifying $\Gamma$ in a region containing a small portion of an edge from $\Gamma_a$, as in the top frame of \autoref{fig:perturbation_abstract}. 

        \item[(2)] An \emph{$ab$-perturbation} of $\Gamma$ is the result of modifying $\Gamma$ along an embedded arc $\delta$ joining an edge from each of $\Gamma_a$ and $\Gamma_b$, as in the middle frame of \autoref{fig:perturbation_abstract}. 

        \item[(3)] An \emph{$abc$-perturbation} of $\Gamma$ is the result of modifying $\Gamma$ along an embedded arc $\rho$ joining an edge from each of $\Gamma_a$ and $\Gamma_c$, which meets an edge of $\Gamma_b$ in a single point, as in the bottom frame of \autoref{fig:perturbation_abstract}.
    \end{enumerate}
    If $I$ is a subset of $\{a,b,c,d\}$ with $1\leq |I|\leq 3$, we will generally refer to this operation as an \textit{$I$-perturbation}. 
\end{definition}

\begin{figure}[ht]
    \centering
    \includegraphics[width=.5\textwidth]{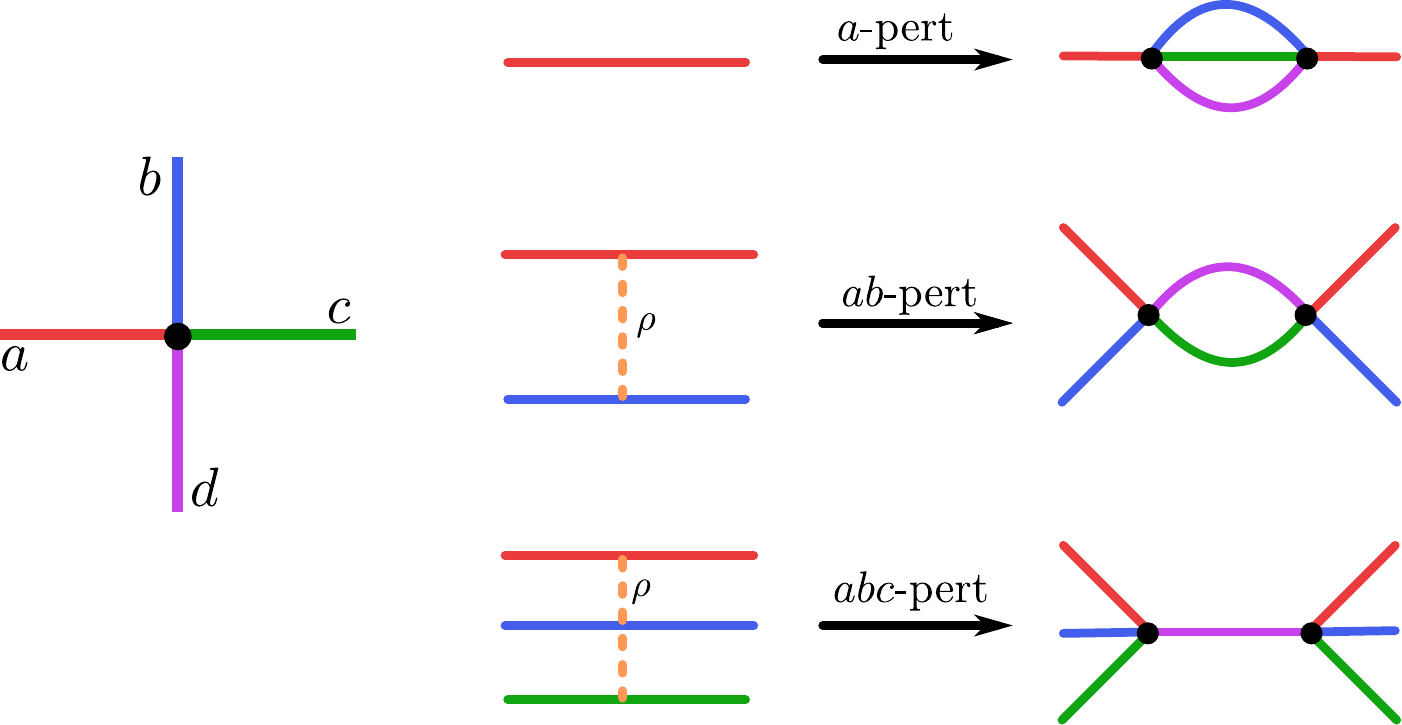}
    \caption{Three local models for a perturbation of an abstract 4-section, where $a,b,c$ and $d$ correspond to the red, blue, green, and purple subgraphs. Top: an illustration of an $a$-perturbation. Middle: an illustration of an $ab$-perturbation. Bottom: an illustration of an $abc$-perturbation.}
    \label{fig:perturbation_abstract}
\end{figure}

After each kind of local modification above, the result is clearly still an abstract 4-section. In fact, the simple closed curves determined by the spine change in a very controlled way. The following lemma follows immediately from the models in \autoref{fig:perturbation_abstract}. 

\begin{lemma}\label{lem:arc_surgery_changes_Gamma_curves}
    Let $\Gamma=\Gamma_1\cup \Gamma_2 \cup \Gamma_3 \cup \Gamma_4$ be the spine of an abstract $4$-section of $S$. Let $\Gamma'$ be an $I$-perturbation of $\Gamma$.
    \begin{enumerate}
        \item If $I=\{a\}$, then $\Gamma'_{ac}$ is isotopic to $\Gamma_{ac}$ and $\Gamma'_{bd}$ is isotopic to $\Gamma_{bd}\cup \partial D$, where $D\subset S$ is a small disk. 
        \item If $|I|=2$, then $\Gamma'_{ac}$ is isotopic to $\Gamma_{ac}$ and $\Gamma'_{bd}$ is isotopic to $\Gamma_{bd}$.
        \item If $I=\{a,b,c\}$ or $\{c,d,a\}$, then $\Gamma'_{ac}$ is isotopic to arc surgery $\Gamma_{ac}[\rho]$ of $\Gamma_{ac}$ along $\rho$ and $\Gamma'_{bd}$ is isotopic to $\Gamma_{bd}$.
    \end{enumerate}
\end{lemma}

On the other hand, modulo abstract perturbations, we can perform certain band surgeries on an abstract 4-section.

\begin{lemma}\label{lem:abstract_perturbations_trick}
    Let $\Gamma=\Gamma_1\cup \Gamma_2 \cup \Gamma_3 \cup \Gamma_4$ be an abstract $4$-section of $S$. Suppose that $\rho\subset S$ is an embedded arc with $\rho\cap \Gamma_{ac}=\partial \rho$. Then, after some perturbations of $\Gamma$, there is an abstract $4$-section $\Gamma'$ of $S$ such that $\Gamma'_{ac}=\Gamma_{ac}[\rho]$ is isotopic to an arc surgery along $\rho$ and $\Gamma'_{bd}$ is isotopic to $\Gamma_{bd}$. 
\end{lemma}

\begin{proof}
    Note that the interior of $\rho$ may cross $\Gamma_{bd}$. If $|\interior (\rho)\cap \Gamma_{bd}|=1$, then $\rho$ guides an $axc$-perturbation of $\Gamma$ for $x\in \{b,c\}$. By \autoref{lem:abstract_perturbations_trick}, this effects an arc surgery on $\Gamma_{ac}$. If $|\interior(\rho)\cap \Gamma_{bd}|>1$, we can do $b$- or $d$- perturbations as in \autoref{fig:arc_surgery_trick} to effectively contract $\rho$ while preserving the isotopy classes of $\Gamma_{ac}$ and $\Gamma_{bd}$. 

    \begin{figure}[ht]
        \centering
        \includegraphics[width=.7\textwidth]{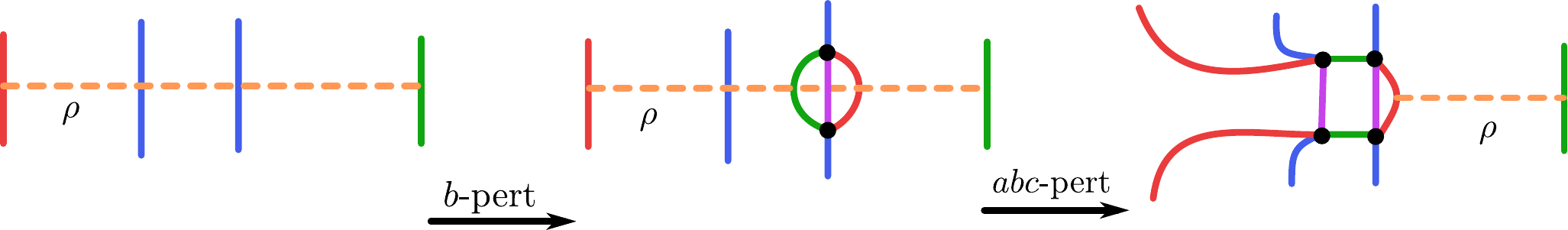}
        \caption{An illustration of the proof of \autoref{lem:abstract_perturbations_trick}. By performing repeated perturbations along an arc, we can find a new abstract 4-section which realizes arc surgery of $\Gamma_{ac}$ along $\rho$.}
        \label{fig:arc_surgery_trick}
    \end{figure}

    Once $|\interior (\rho)\cap \Gamma_{bd}|=1$, we proceed as in the previous case.
\end{proof}

\begin{example}\label{ex:arc_surgery_example}
    We illustrate the previous process with an example. \autoref{fig:arc_surgery_ex1} and \autoref{fig:arc_surgery_ex2} illustrate the result of two different arc surgeries on an abstract 4-section of the torus. 
    
    \begin{figure}[ht]
        \centering
        \includegraphics[width=.65\textwidth]{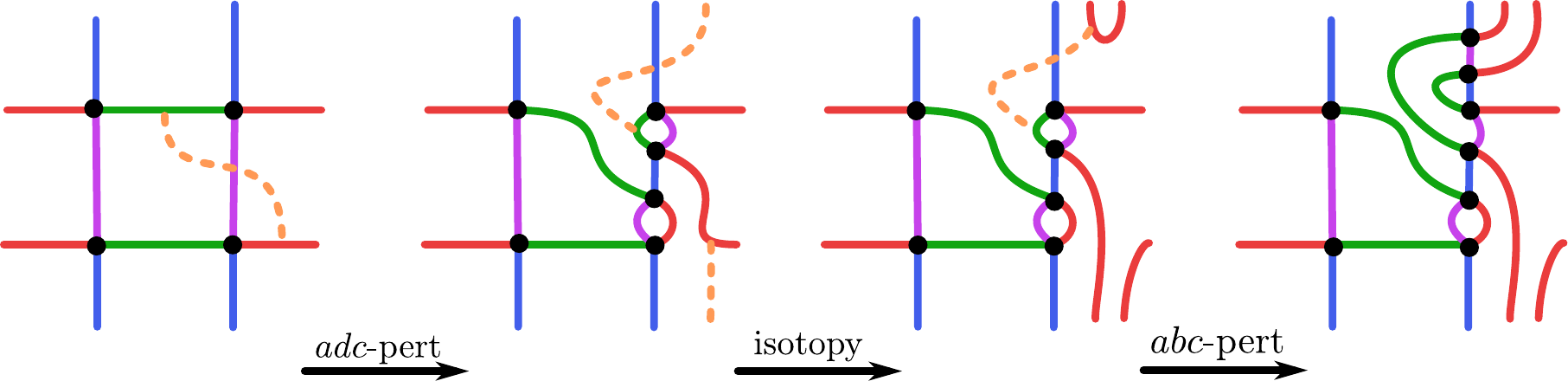}
        \caption{An illustration of a sequence of arc surgeries on an abstract 4-section of the torus. Each arc surgery is indicated by a dashed arrow. The process converts $\Gamma_{ac}$ from two copies of the $(0,1)$ curve to two copies of the  $(0,1)$ curve. }
        \label{fig:arc_surgery_ex1}
    \end{figure}

    \begin{figure}[ht]
        \centering
        \includegraphics[width=.85\textwidth]{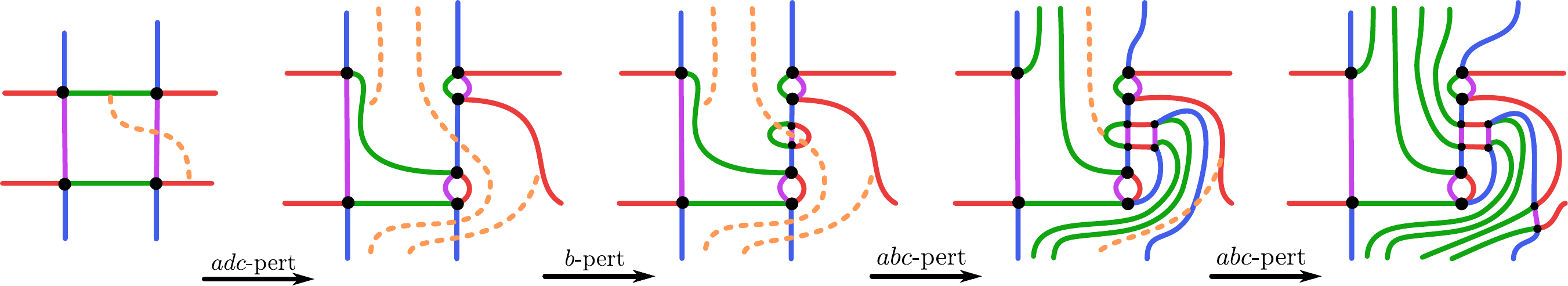}
        \caption{An illustration of a sequence of arc surgeries on an abstract 4-section of the torus. Each arc surgery is indicated by a dashed arrow. In this case, the process converts $\Gamma_{ac}$ from two copies of the $(0,1)$ curve to two copies of the  $(2,1)$ curve. }
        \label{fig:arc_surgery_ex2}
    \end{figure}
\end{example}

\subsection{Perturbations of both abstract and embedded surfaces}\label{sec:blending_abstract_and_embedded_quadrisections}

Let $\Tcal=(T_1,T_2,T_3,T_4)$ be a bridge 4-section of an embedded surface $\Sigma:S\hra S^4$. If $\Gamma$ is the spine of the abstract 4-section of $S$ given by $\Gamma_i=\Sigma^{-1}(T_i)$, we write $\Gamma=\Gamma_\Tcal$. By construction, $I$-perturbations of $\Tcal$ descend to $I$-perturbations of $\Gamma_\Tcal$. The following result shows that a kind of converse holds. 

\begin{lemma}\label{lem:abstract_perturbations_lift}
    Let $\Tcal$ be a bridge $4$-section of $\Sigma\subset S^4$ and let $\Gamma=\Gamma_{\Tcal}$. Suppose that $\Gamma'$ is an $I$-perturbation of $\Gamma$. Then there exists an $I$-perturbation of $\Tcal$, denoted by $\Tcal'$, such that $\Gamma_{\Tcal'}=\Gamma'$. 
\end{lemma}

\begin{proof}
    Let $x\in \{1,2,3,4\}\setminus I$. Let $e\subset \Gamma_x'$ be the new edge of $\Gamma'$ that is labeled $x$; such an edge exists by the definition of $I$-perturbation. Using the language in \cite[\S7]{aranda24}, the graph $\Gamma$ is the \emph{edge-compression} of $\Gamma'$ along $e$. Then Proposition 7.1 of \cite{aranda24} gives us the desired conclusion.     
\end{proof}

\begin{remark}\label{rem:unique_perturbation_lift}
    It is not clear (and possibly false) that every $I$-perturbation of $\Gamma_\Tcal$ lifts to a \emph{unique} $I$-perturbation of $\Tcal$. 
\end{remark}

\section{Existence of bridge \texorpdfstring{$4$}{4}-sections of \texorpdfstring{$3$}{3}-manifolds}\label{sec:existence}

In this section, we prove the main theorem of this paper, \autoref{thm:existance_bridge_4sections}, that any embedded 3-manifold in $S^5$ may be isotoped into bridge position as in \autoref{def:bridge_4section_3man}. The proof is broken down into three steps. First, we isotope the 3-manifold $Y^3\subset S^5$ into relative Morse position so that the equatorial $S^4$ in $S^5$ cuts $Y^3$ into two 3-dimensional handlebodies $H_\alpha$ and $H_\beta$. Such an embedding of $Y^3$ can be codified with a tuple $(\Sigma; \Da, \Db)$ where $\Sigma$ is a surface in $S^4$ and, for each $\varepsilon=\alpha,\beta$, $D_\varepsilon$ are compressing disks for the handlebody $H_\varepsilon$ embedded in $S^4$. 
Next, in \autoref{sec:bridge_position_heegaard_complexes}, we find a bridge 4-section of $\Sigma$ in $S^4$ that also contains the data of $\Da$ and $\Db$; we call such a 4-section is called a \emph{bridge position} for the Heegaard complex. To end, \autoref{sec:4sections_from_bridge_position} builds a quadrisection of an embedded 3-manifold using a Heegaard complex in bridge position as input data.

\subsection{Hyperbolic embeddings and Heegaard complexes of \texorpdfstring{$3$}{3}-manifolds}\label{sec:Heegaard_complexes}

A Heegaard complex is a convenient way of describing a (possibly non-orientable) 3-manifold $Y$ embedded in $S^5$. 

\begin{definition}\label{def:hyperbolic_morse}
    Let $h \colon S^5\subset \mathbb{R}^6\to \mathbb{R}$ be the natural Morse function with two critical points. We say that an embedding of a connected 3-manifold $Y\subset S^5$ is \emph{hyperbolic} if the critical values of $h\vert_Y$ appear with increasing index.
\end{definition}

The terminology is borrowed from the corresponding language for knotted surfaces. Note that such an embedding may have many handles of a given index, and in particular, many local minima. 

\begin{proposition}\label{prop:hyperbolic_embedding}
    Any embedding of a $3$-manifold $Y\subset S^5$ may be ambiently isotoped so that it is hyperbolic. 
\end{proposition}

\begin{proof}
    It is a standard fact (e.g., see \cite{perron75} or \cite[Theorem 3.4]{sharpe88}) that in codimension two, embedded lower index handles can be ambiently isotoped to appear below higher index handles. In particular, we can arrange that the handles of $Y$ induced by $h\vert_Y$ appear with increasing index. 
\end{proof}

We now describe how to record such an embedding.

\begin{definition}\label{def:heegaard_complex}
    An (oriented) \emph{Heegaard complex} is a triple $\left(\Sigma;D_\alpha, D_\beta\right)$ embedded in $S^4$ such that: 
    \begin{enumerate}
        \item $\Sigma\subset S^4$ is a closed, connected (orientable) surface, 
        \item For $\eps\in\{\alpha,\beta\}$, $D_\eps\subset S^4$ is a collection of embedded disks such that 
        \begin{enumerate}
            \item[(i)] $D_\eps \cap \Sigma=\partial D_\eps$,
            \item[(ii)] Each component $D\in D_\eps$ is \emph{framed}, that is, the framing $\overline{\partial D}$ on $\nu(D)\vert_{\partial D}\cong \partial D\times D^2$ obtained by restricting the unique framing of $\nu(D)$ agrees with the surface framing on $\nu(D)\vert_{\partial D}$,
            \item[(iii)] The result $\Sigma\mid D_\eps:=\left(\Sigma-\eta(D_\eps)\right)\cup \left(D_\eps \times \{-1,1\}\right)$ of surgering $\Sigma$ along $D_\eps$ is an unlink of 2-spheres in $S^4$.
        \end{enumerate}
    \end{enumerate}
\end{definition}

\begin{remark}\label{rem:framed_disks}
    Note that $D_\alpha$ and $D_\beta$ will generally intersect in $S^4$. Moreover, we do not require $\partial D_\alpha$ or $\partial D_\beta$ to be a minimal cut system for $\Sigma$, so the various components of $\partial D_\eps\subset \Sigma$ may be isotopic or trivial in $\Sigma$. The corresponding realizing 3-manifold may have many 0- and 3-handles. 
    
    Moreover, the condition in (ii) ensures that $\Sigma$ can be ambiently surgered along along $D_\eps$ to yield an embedded surface. Indeed, a section of $\nu(D)\vert_{\partial D}$ obtained from a pushoff of $\partial D$ on $\Sigma$ bounds an embedded disk in $\nu(D)$ disjoint from $D$. 
\end{remark}

\begin{remark}
    The surface $\Sigma$ in a Heegaard complex is a ribbon surface, i.e., it is obtained from unlinked $2$-spheres in (iii) by tubing along the cocores of $D$'s in (ii).
\end{remark}

\begin{remark}\label{rem:nonorientable_heegaard_complex}
    Though we will deal mainly with the orientable case, we can also define \emph{non-orientable Heegaard complexes}. In this case, we allow $\Sigma\subset S^4$ to be non-orientable in (1) above, but require it to have normal Euler number in $S^4$ equal to zero. In particular, it is abstractly homeomorphic to a sum of an even number of projective planes. In (i), we also require the boundary of a disk in $D_\alpha$ or $D_\beta$ to have an annular neighborhood in $\Sigma$, so that it is sensible to discuss compressing $\Sigma$ along $D_\eps$ in (iii). 
\end{remark}

From a Heegaard complex, we can completely recover an embedded 3-manifold. 

\begin{proposition}\label{prop:heegaard_complex_determines_embedding}
    A Heegaard complex $(\Sigma;D_\alpha,D_\beta)$ determines a unique embedding of a closed $3$-manifold $Y(\Sigma;D_\alpha,D_\beta)$ into $S^5$ up to isotopy. 
\end{proposition}

\begin{proof}
    Given a Heegaard complex $(\Sigma;D_\alpha,D_\beta)$, we can construct an embedded 3-manifold, which we call the \emph{realizing $3$-manifold}, as follows. Beginning with $\Sigma\times [-1,1,]\subset S^4\times [-2,2]$, attach thickened 3-dimensional 2-handles corresponding to $D_\alpha$ to $\Sigma\times \{-1\}\subset S^4\times \{-1\}$ and thickened 3-dimensional 2-handles corresponding to $D_\beta$ to $\Sigma\times \{1\}\subset S^4\times \{1\}$. Note that by condition (ii) in \autoref{def:heegaard_complex}, the disks $D_\alpha$ and $D_\beta$ may be used as the cores of ambiently attached 2-handles. After adding these handles, condition (iii) guarantees that the result is an unlink of 2-spheres in $S^4\times \{1\}$ and $S^4\times \{-1\}$. By \cite{powell24}, we can fill these unlinks of 2-spheres with a unique (up to isotopy) collection of boundary parallel 3-balls in $S^4\times [1,2]$ and $S^4\times [-2,-1]$ to produce a closed 3-manifold $Y(\Sigma,D_\alpha,D_\beta)\subset S^4\times [-2,2]\subset S^5$. 
\end{proof}

\begin{proposition}\label{prop:admit_heegaard_complex}
    Every connected embedded $3$-manifold $Y\subset S^5$ admits a Heegaard complex. 
\end{proposition}

\begin{proof}
    We may perturb the embedding of $Y\subset S^5$ so that if $h \colon S^5\to [-1,1]$ is the natural Morse function with two critical points, then $h\vert_S$ is also Morse. Furthermore, by \autoref{prop:hyperbolic_embedding}, we may assume that the critical values of $h\vert_S$ are isolated and that they appear with increasing index. We may also assume that all index 0 and 1 critical values are negative, and all index 2 and 3 critical values are positive. 

    Thus, $\Sigma=h\vert_S^{-1}(0)=\Sigma$ is a connected (ribbon) surface in $S^4$, and bounds a handlebody to both sides of $h^{-1}(0)$. In particular, the cores of the 1-handles and co-cores of the 2-handles can each be pushed into $h^{-1}(0)$. Consequently, the co-cores of the 1-handles (which we will call $D_\alpha$) and cores of the 2-handles (which we will call $D_\beta$) form two systems of disks with boundary lying on $\Sigma$. By construction, they satisfy conditions (i), (ii), and (iii) in \autoref{def:heegaard_complex} above. 
\end{proof}

\begin{remark}\label{rem:heegaard_disks}
    We can modify a Heegaard complex without substantially changing the corresponding embedding. For example, for $\eps \in {\alpha, \beta}$, we may add parallel copies of disks contained in $D_\eps$ to $D_\eps$; in particular, we may assume that $\partial D_\eps$ is trivial in $H_1(\Sigma;\mathbb{Z}_2)$. We can also slide $\eps$-disks over $\eps$-disks without changing the isotopy class of the embedding $Y\hra S^5$, since this corresponds to performing ambient handle slides of $Y$. 

    Lastly, starting from $(\Sigma;D_\alpha,D_\beta)$, one obtains a natural Heegaard splitting of $Y(\Sigma;D_\alpha,D_\beta)$ by forgetting the embedding of $\Sigma$ and deleting any homologically redundant curves in $D_\alpha$ and $D_\beta$. 
\end{remark}

\begin{question}
Suppose $Y\hookrightarrow S^5$ admits a genus-one Heegaard complex $(\Sigma;\Da,\Db)$ with $|\Da|=|\Db|=1$. What can we say about $Y$?
\end{question}

We now describe the effect of adding canceling pairs of handles to the embedding of $Y$ on its associated Heegaard complex. 

\begin{definition}\label{def:canceling_handles} 
    Let $(\Sigma;D_\alpha,D_\beta)$ be a Heegaard complex, with realizing 3-manifold $Y\subset S^5$.  
    \begin{enumerate}
        \item \textbf{Adding a canceling 1-/2-handle.} Let $t$ be an arc in $S^4$ with $t\cap \Sigma = \partial t$ and suppose that $t\subset E\subset D_\alpha$. If it is not already the case, we may arrange that $t\cap D_\beta=\emptyset$ by a small isotopy of $D_\alpha$. Let $\Sigma'$ be the result of tubing $\Sigma$ along $t$; note that by \cite[\S 2]{boyle88}, the result depends only the homotopy class of $t$ in $S^4\setminus \Sigma$, and whether the tube preserves the local orientations of $\Sigma$ at its endpoints. Let $D'_\beta=D_\beta \cup c$, where $c\subset \Sigma$ is a meridional disk corresponding to $t$. The disk $E$ is divided by $t$ into two smaller disks $E_1$ and $E_2$, and we let $D'_\alpha=\left(D_\alpha-E\right) \cup \left(E_1\cup E_2\right)$. The Heegaard complex $(\Sigma';D_\alpha',D_\beta')$ is called a \emph{1/2 stabilization} of $(\Sigma;D_\alpha,D_\beta)$.
        \item \textbf{Adding a canceling 2-/3-handle.} Let $U\subset S^4$ be an unknotted 2-sphere unlinked from $\Sigma$; that is, $U$ bounds a 3-ball in $S^4$ disjoint from $\Sigma$. Let $t$ be an arc in $S^4$ with interior disjoint from $U$ and $\Sigma$ and one endpoint on each surface. Let $\Sigma'$ be the result of tubing $\Sigma\cup U$ along $y$. Let $D'_\alpha=D_\alpha$ and $D'_\beta=D_\beta\cup c$, where $c$ is a meridional disk corresponding to $t$. In this case, the Heegaard complex $(\Sigma';D_\alpha',D_\beta')$ is called a \emph{2/3 stabilization} of $(\Sigma;D_\alpha,D_\beta)$. 
        \item \textbf{Adding a canceling 0-/1-handle.} Let $U$, $S^4$, $t$, and $\Sigma'$ be as in the previous case. Let $D'_\beta=D_\beta$ and $D'_\alpha=D_\alpha\cup c$, where $c$ is a meridional disk corresponding to $t$. In this case, the Heegaard complex $(\Sigma';D_\alpha',D_\beta')$ is called a \emph{0/1 stabilization} of $(\Sigma;D_\alpha,D_\beta)$. 
        \item \textbf{Handle slide.} Let $E, F\subset D_\alpha$ be two disks. Let $E'$ be a disk obtained by handle sliding $E$ over $F$. Let $D'_{\alpha}=(D_\alpha\setminus E)\cup E'$ and $D_\beta'=D_\beta$. Then we say that $(\Sigma';D_\alpha',D_\beta')$ is obtained from $(\Sigma;D_\alpha,D_\beta)$ by a handle slide. We define the $\beta$-side handle slide similarly.
        \item \textbf{Handle swim.} We say that $(\Sigma; D_\alpha', D_\beta)$ is obtained from $(\Sigma; D_\alpha', D_\beta)$ by a handle swim if there exist $E\subset D_\alpha$ and $E'\subset D_\alpha'$ such that $D_\alpha\setminus E=D_\alpha'\setminus E'$, and $E$ and $E$' are isotopic rel boundary in $S^4\setminus (\Gamma \cup D_\alpha)$, where $\Gamma$ is the result of a surgery on $\Sigma$ along a disk in $D_\alpha \setminus E$. We define the $\beta$-side handle swim similarly.
    \end{enumerate}
\end{definition}

\begin{theorem}\label{thm:canceling_handles_HC}
    Let $(\Sigma;D_\alpha, D_\beta)$ and $(\Sigma';D_\alpha', D_\beta')$ be Heegaard complexes for embedded $3$-manifolds $Y$ and $Y'$ in $S^5$, respectively. Then $Y$ and $Y'$ are isotopic if and only if the Heegaard complexes are related by a sequence of the moves in \autoref{def:canceling_handles}.
\end{theorem}

\begin{proof}
    The proof is essentially the same as Swenton's uniqueness theorem for \emph{banded unlink presentations} of surfaces in $S^4$~\cite{swenton2001calculus}; 
    Swenton's cup and cap moves correspond to additions of canceling pairs of handles in \autoref{def:canceling_handles}, and band slide and swim correspond to a local isotopy of $1$-handles. 
    
\end{proof}

\subsection{Bridge splittings of Heegaard complexes}\label{sec:bridge_position_heegaard_complexes}

In this section, we will prove that every Heegaard complex of $Y^3\subset S^5$ can be isotoped and stabilized into \emph{bridge position}, a structure analogous to a \emph{banded bridge splitting} of an unlink (see \cite[\S3]{meier17}). 

\begin{definition}\label{def:heegaard_complex_bridge_position}
    A Heegaard complex $(\Sigma;\Da, \Db)$ is in \emph{bridge position} with respect to a genus-zero 4-section of $S^5$ if there is a bridge 4-section $\Tcal=(T_1,T_2,T_3,T_4)$ of $\Sigma\subset S^4\subset S^5$ where $T_i\subset B_i$ satisfying the following conditions. 
    \begin{enumerate}
        \item $\Da \subset B_1\cup\overline{B}_3$ with $T_1\cup \T_3 =\partial \Da$, 
        \item $\Db \subset B_2\cup\overline{B}_4$ with $T_2\cup \T_4 =\partial \Db$.
    \end{enumerate}
\end{definition}

In the next subsection, we will show that if a Heegaard complex is in bridge position, it admits a natural 4-section in $S^5$. An example of a Heegaard complex for $S^3$ in bridge position is illustrated in \autoref{fig:heegaard_complex_example}.  

\begin{figure}[ht]
    \centering
    \includegraphics[width=.6\textwidth]{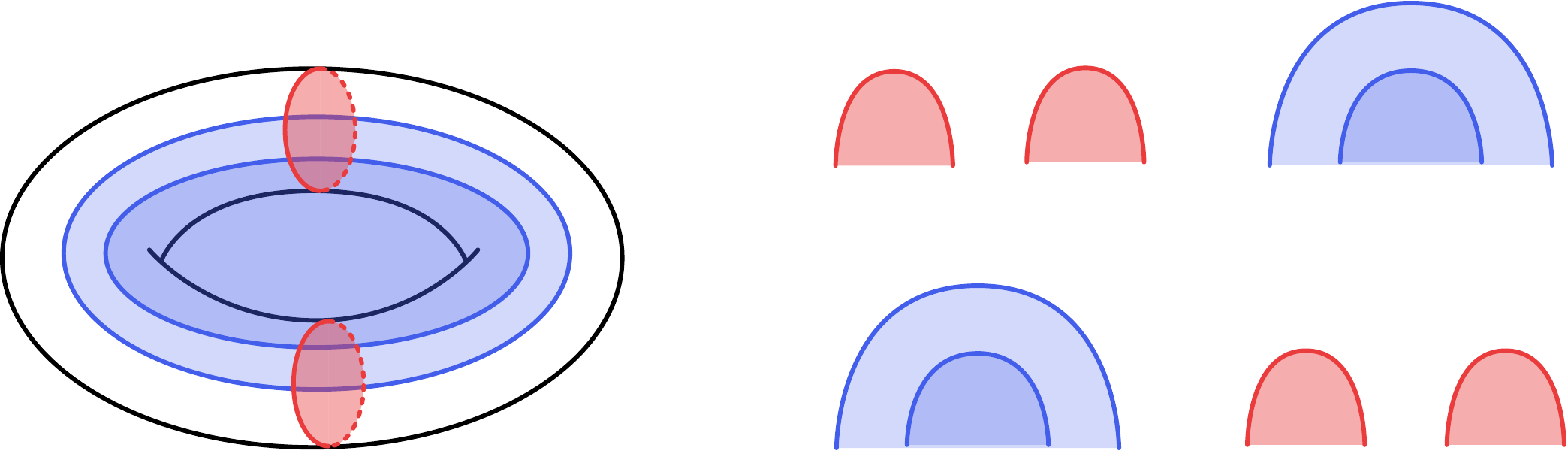}
    \caption{Left: a schematic Heegaard complex of $S^3\subset S^5$, notice that all the data in the Heegaard complex is actually embedded in 3-space. Right: a depiction of the intersection of $S^3$ with the four sectors of the genus-zero $4$-section of $S^5$. Note that the disks $\Da$ and $\Db$, shaded in the figure, lie in the cross-section 3-spheres obtained by gluing two 3-balls.}
    \label{fig:heegaard_complex_example}
\end{figure}

The following proposition shows that we can always arrange the boundary conditions in \autoref{def:heegaard_complex_bridge_position}.

\begin{proposition}\label{prop:existance_step1}
    Suppose that $Y\subset S^5$ is an embedded $3$-manifold, described by a Heegaard complex $(\Sigma;\Da, \Db)$. Then there is a bridge $4$-section $\Tcal=(T_1,T_2,T_3,T_4)$ of $\Sigma$ such that
    \begin{enumerate}
        \item $\partial \Da = T_1 \cup \T_3$, and 
        \item $\partial \Db = T_2 \cup \T_4$.
    \end{enumerate}
\end{proposition}

\begin{proof}
    By \autoref{rem:heegaard_disks}, we may assume that the curves $\partial \Da$ and $\partial \Db$ are null-homologous multicurves in $H_1(\Sigma;\mathbb{Z}_2)$. Let $\Tcal_0$ be a bridge 4-section for $\Sigma\subset S^4$ and let $\Gamma^0=\Gamma_{\Tcal_0}$ be the spine of the abstract 4-section for $\Sigma$. 
    
    By \autoref{lem:arc_surgery_nullhom}, there is a sequence of arc surgeries of $\Gamma^0$ taking $\Gamma^0_{13}$ to $\partial \Da$. By \autoref{lem:abstract_perturbations_trick}, these arc surgeries can be achieved via perturbations of $\Gamma^0$ that do not change the multicurve $\Gamma^0_{24}$ up to isotopy. By \autoref{lem:abstract_perturbations_lift}, we can lift each of these abstract perturbations to a perturbation of the bridge 4-section $\Tcal_0$. Let $\Tcal_1$ be the corresponding bridge 4-section with $\Gamma^1=\Gamma_{\Tcal_1}$. By construction, we have $\Gamma^1_{13}=\partial \Da$ and $\Gamma^1_{24}=\Gamma^0_{24}$. 
    Then, we can repeat the argument above, reversing the roles of $(13)$ and $(24)$ to obtain the desired bridge 4-section.
\end{proof}

Now, we will stabilize our surface $\Sigma$ to arrange that the disks $\Da$ and $\Db$ are embedded in the 3-sphere cross-sections $B_i\cup \overline{B}_{i+2}$.  

\begin{theorem}\label{thm:existance_bridge_position} 
    Every embedded $3$-manifold $Y\subset S^5$ admits a Heegaard complex in bridge position.
\end{theorem}

\begin{proof}
Consider $(\Sigma;\Da, \Db)$ and $\Tcal$ as in the conclusion of \autoref{prop:existance_step1}. In what follows, we will modify $\Sigma$ and $\Tcal$ to ensure that $\Da\subset B_1\cup \overline{B_3}$, without altering the boundary conditions $\partial \Da =T_1\cup \T_3$ and $\partial \Db = T_2\cup \T_4$. The theorem will follow by rerunning the argument with $\Db$ instead of $\Da$. 

Denote the sectors of the genus-zero 4-section of $S^4$ by $\{X_1,X_2,X_3,X_4\}$. Let $X_{123}=X_{1}\cup_{B_2} X_{2}$, with $\partial X_{123}= B_1\cup \overline{B}_3$. Push the interior of $\Da$ into the interior of  $X_{123}$ via an isotopy rel. $\partial \Da=T_1\cup \T_3=L_{13}$, so that $\partial D_\alpha$ bounds a collection of pairwise disjoint slice disks $\Da\subset X_{123}$. By a further isotopy of $\Da$ (rel. $\partial D_\alpha$), we may assume $D_\alpha$ is in Morse position with respect to the radial height function on $X_{123}$ induced from $S^5$, which is constant on $\partial X_{123}$ and has one minimum. By an isotopy, we can push a neighborhood of the local maxima of $\Da$ into $\partial X_{123}$; these are small 2-disks disjoint from $L_{13}$. Similarly, we may isotope the saddles of $\Da$ into $\partial X_{123}$, and view these as bands connecting the components of this link. 

Thus, there exists an unlink $\ell\subset B_1\cup \overline{B}_3$ disjoint from $L_{13}$ and a set of bands $\nu\subset B_1\cup \overline{B}_3$ such that the link $\left(L_{13}\cup \ell\right)[\nu]$ obtained by band surgery on $L_{13}\cup \ell$ along $\nu$ is an unlink. In fact, $\left(L_{13}\cup \ell\right)[\nu]$ bounds the sub-disks of $\Da$ corresponding to the local minima of $\Da\subset X_{123}$.

By transversality, we can assume that $\Db$ is disjoint from $\nu$ and $\ell$. Moreover, we can assume that $\ell$ and $\nu$ are transverse to the middle sphere $F$ of $S^3=B_1\cup_F\overline {B}_3$, and each component of $\ell$ is a one-bridge unknot with respect to the splitting $S^3=B_1\cup\overline {B}_3$. 

We are now ready to modify $\Sigma$. Let $U$ be an $|\ell|$-component unlink of 2-spheres in $S^4$ unlinked with $\Sigma$; i.e., so that $U$ bounds 3-balls disjoint from $\Sigma$. Choose $U$ so that $U\cap \left(B_1\cup \overline{B}_3\right)=\ell$. 

We can use the cores of $\nu$ as the guiding arcs for 1-handle attachments of $\Sigma\cup U$. By \autoref{thm:canceling_handles_HC}, the resulting surface can be completed to a Heegaard complex $(\Sigma'; \Da', \Db')$ for the same (up to isotopy) embedded 3-manifold $Y^3\subset S^5$. In fact, we know that same result shows that $\Db'$ is the union of $\Db$ and some meridians of the cores of $\nu$, and $\partial \Da'=\left(L_{13}\cup \ell\right) [\nu]$. 

We now pay attention to the bridge 4-sections of $\Sigma$ and $\Sigma'$. Let $\Tcal_0=(T^0_1,T^0_2,T^0_3,T^0_4)$ be a bridge 4-section of $U$ such that $\ell=T^0_1\cup \T^0_3$; this may be obtained from the union of $|\ell|$ copies of 1-bridge 4-sections of an unknotted 2-sphere. Let $\Tcal\cup \Tcal_0$ be the 4-section of $\Sigma\cup U$ obtained by taking the disjoint union of the respective pieces; this is also a bridge 4-section since $\ell=T^0_1\cup \T^0_3$ is unlinked with $L_{13}=T_1\cup \T_3$. Now, \autoref{lem:tubing_4sec_with_bands} states that $\Tcal\cup \Tcal_0$ can be modified to a bridge 4-section $\Tcal'$ of $\Sigma'$ with $L'_{13}=L_{13}[\nu]$ and $L'_{24}=L_{24}\cup c$ for some unknots $c$ corresponding to copies of meridians of cores of $\nu$. In other words, 
\[ 
\Da \subset B_1\cup \overline{B}_3, ~ \partial \Da'=T'_1\cup \overline{T'}_3, \text{ and } \partial \Db' = T'_2 \cup \overline{T'}_4.
\]    
Applying the same argument above to $\Db$ completes the proof. 
\end{proof}

\subsection{From bridge position to \texorpdfstring{$4$}{4}-sections}\label{sec:4sections_from_bridge_position}

In this final subsection, we will show that a Heegaard complex in bridge position, as in \autoref{def:heegaard_complex_bridge_position}, yields a bridge quadrisection of the underlying 3-manifold. 

Fix a Morse function $f:S^5\twoheadrightarrow [-2,2]$ with exactly two critical points of index 0 and 5. For $I\subset [-2,2]$, we use the notation $S^5_{I}=f^{-1}(I)$. For a subset $A\subset S^4_t$ with $t\in [s,r]$, we let $A[s,r]$ denote the vertical cylinder $A\times [s,r]$ obtained by pushing $A$ along the gradient flow of $f$ during time $[s,r]$. The symbol $A\{r\}\subset S^5_r$ will denote the image of the gradient flow at time $r$. 

\begin{figure}[ht]
    \centering
    \includegraphics[width=.7\textwidth]{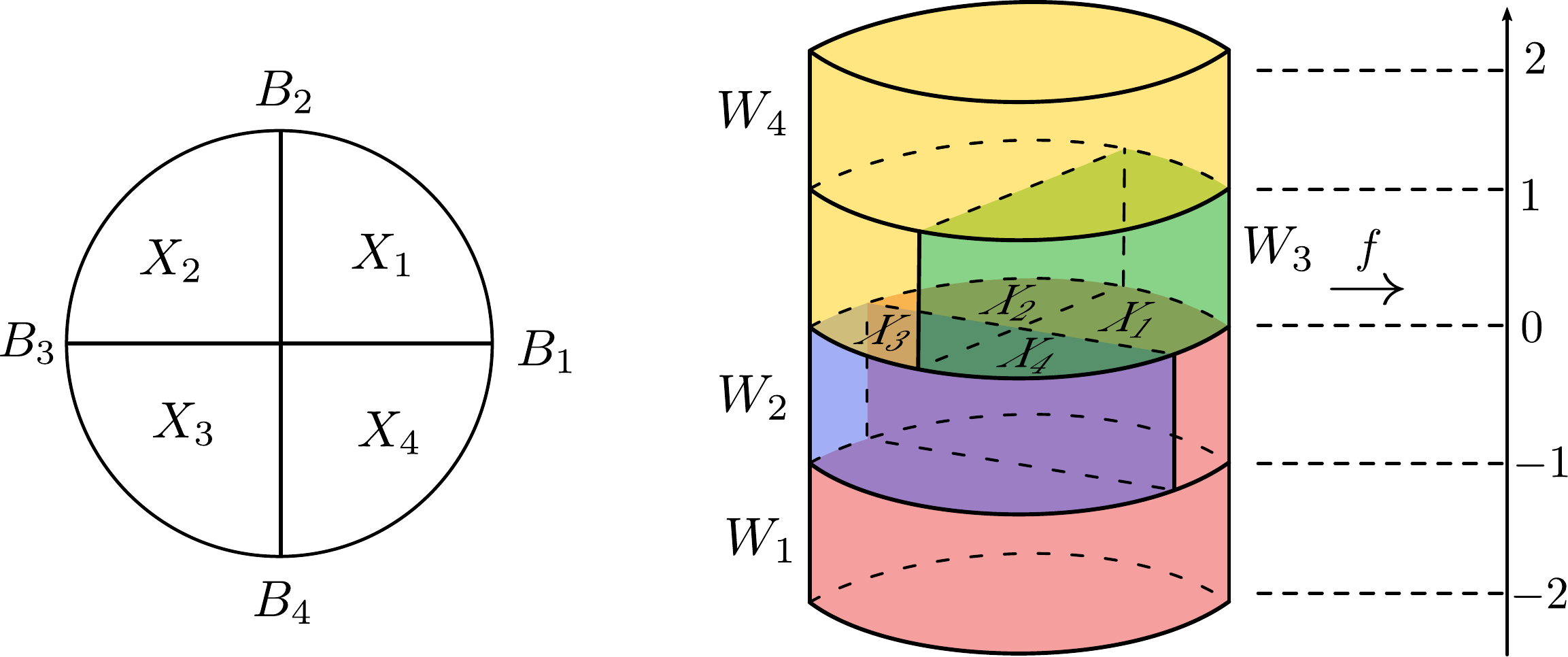}
    \caption{Quadrisecting the 5-sphere.}
    \label{fig:model_quadrisection}
\end{figure}

We first review how genus-zero quadrisections of $S^5$ are obtained from Morse functions as in \cite[Theorem 7.3]{aribi23}. Fix a genus-zero quadrisection of the level set $S^5_0$; we use the notation $S^4=X_1\cup X_2\cup X_3\cup X_4$ from \autoref{def:multisections_surfaces}. Let $W_1=\left(X_1\cup X_2\right)[-1,0]\cup S^5_{[-2,-1]}$, $W_2=\left(X_3\cup X_4\right)[-1,0]$, $W_3=\left(X_4\cup X_1\right)[0,1]$, and $W_4=\left(X_3\cup X_2\right)[0,1]\cup S^5_{[0,1]}$; see \autoref{fig:model_quadrisection} for reference. The splitting $S^5=W_1\cup W_2\cup W_3\cup W_4$ is a genus-zero quadrisection. It is worth noting that the pairwise intersection $W_1\cap W_2$ is the union of a collar of an $S^3$ with half a quadrisection of a regular level; i.e., $W_1\cap W_2=\left(B_1\cup B_3\right)[0,-1]\cup \left(X_3\cup X_4\right)\{-1\}$. A similar description holds for $W_3\cap W_4$. 

\begin{proposition}\label{prop:quadrisection_from_bridge_position}
    Let $(\Sigma;\Da,\Db)$ be a Heegaard complex for $Y$ in bridge position. Then there is a bridge quadrisection of $Y\subset S^5$. 
\end{proposition}

\begin{proof}
Let $(S^4,\Sigma)=\bigcup_{i=1}^4(X_i,D_i)$ be the bridge quadrisection satisfying conditions (1) and (2) in \autoref{def:heegaard_complex_bridge_position}. We start by embedding $\Sigma$ in $S^5_0$ and flowing $\Sigma=\bigcup_{i=1}^4 D_i$ through $[-1/2,0]$. We will attach the 3-dimensional 2-handles corresponding to $\Da$ at the level set $S^5_{-1/2}$. To do this, we consider a bicollar neighborhood $C=(B_1\cup B_3)\times [-1,1]$ of the 3-sphere $(B_1\cup B_3)\{-1/2\}$ inside $S^5_{-1/2}$; see \autoref{fig:model_quadrisection} for notation. We label the interval $[-1,1]$ so that $(B_1\cup B_3)\times\{-1\}\subset X_2\cup X_3$ and $(B_1\cup B_3)\times\{1\}\subset X_1\cup X_4$. The 3-dimensional 2-handle with cores equal to $\Da$ can be seen inside $C$ as the product $h_\alpha=\Da\times [-1,1]$. After compressing $\Sigma$ along $\Da$, the resulting surface link can be described as 

\begin{align*}
\Sigma|{\Da} &
= \Sigma-\left(\Da \times [-1,1]\right) \cup \left(\Da \times \{-1,1\}\right)
\\
= & \left[\wt D_2\cup \wt D_3 \cup \left(\Da\times \{-1\}\right)\right]
\cup 
\left[\wt D_1\cup \wt D_4 \cup \left(\Da\times \{+1\}\right)\right],
\end{align*}

where $\wt D_i$ are the smaller disks obtained from $D_i$ after removing what is inside $C$. Note that the unions of the square brackets form closed surfaces, separated by the 3-sphere $(B_1\cup B_3)\times \{0\}$. Now, as $\Sigma\mid{\Da}$ is an unlink of 2-spheres, there exist two sets of pairwise disjoint 3-balls $E_-$ and $E_+$ embedded in $S^5_{-1/2}$ such that 
\[ 
\partial E_- = \wt D_2\cup \wt D_3 \cup \left(\Da\times \{-1\}\right),  ~ ~ 
\partial E_+ = \wt D_1\cup \wt D_4 \cup \left(\Da\times \{+1\}\right),
\]
and $E_\pm$ are disjoint from $(B_1\cup B_3)\times \{0\}$. Notice that the 3-dimensional handlebody $H_\alpha\subset Y$ determined by $\Da$ is isotopic (rel. boundary) to the union 
$H_\alpha = \Sigma[-1/2,0]\cup \left( h_\alpha^2 \cup E_- \cup E_+\right)$, where $\left( h_\alpha \cup E_- \cup E_+\right)$ is a subset of $S^5_{-1/2}$. One can check that $H_\alpha$ already intersects the sectors of the genus-zero quadrisection of $S^5$ in trivial disk systems of the correct dimension. 

Note that the situation is completely symmetric for $\Db$ and $H_\beta$: we flow $\Sigma$ through time $[0,1/2]$, we attach the 2-handles of $\Db$ along a collar of $(B_2\cup B_4)\{1/2\}$, and find 3-balls $E_\pm$ bounded by $\Sigma\mid \Db$ disjoint from $(B_2\cup B_4)\{1/2\}$. We conclude that $Y=H_\alpha\cup_\Sigma H_\beta$ is in bridge quadrisected position. 
\end{proof}

\begin{theorem}\label{thm:existance_bridge_4sections}
    Every knotted $3$-manifold $Y$ in $S^5$ admits a bridge quadrisection.
\end{theorem}
\begin{proof}
    By \autoref{thm:existance_bridge_position} the embedding of $Y$ in $S^5$ can be described with a Heegaard complex $(\Sigma;\Da, \Db)$ that is in bridge position. Then \autoref{prop:quadrisection_from_bridge_position} gives us the desired bridge quadrisection of $(S^5,Y)$. 
\end{proof}

\section{Examples of simple embeddings of 3-manifolds}\label{sec:unknotted_examples}

\subsection{Low-complexity diagrams}

To build examples of bridge quadrisected 3-manifolds in $S^5$, one needs to find tuples of tangles $(T_1,T_2,T_3,T_4)$ such that each triplet $(T_i,T_j,T_k)$ is a triplane diagram representing an unlink of 2-spheres. For instance, if $\Tcal=(T_1,T_2,T_3)$ is a triplane diagram for an $m$-component unlink of 2-spheres, then $\widetilde{\Tcal}=(T_1,T_2,T_3,T_3)$ satisfies the desired conditions. Uninterestingly, $\widetilde{\Tcal}$ describes an unlink of 3-spheres in $S^5$: this claim can be checked by observing that the associated Heegaard complex is equal to adding 0/1 and 2/3 pairs of canceling handles to a Heegaard complex with empty disk sets.

\subsubsection{Quadrisections with low-bridge index}

Bridge quadrisections with at most two bridges can be completely classified using the work in Sections 4.1-4.3 of \cite{meier17}. There is exactly one 1-bridge quadrisection, in which each tangle in the spine is the unique 1-stranded trivial tangle. Tuples $(T_1,T_2,T_3,T_4)$ for 2-bridge quadrisections have the property that the rational slope of each $T_i$ is either 0 or $\infty$. The following also holds. 

\begin{proposition}
    Up to PL-homeomorphism, the only $3$-manifolds in $S^5$ admitting $2$-bridge quadrisections are the unknotted $S^3$ and the unlinked $S^3\sqcup S^3$.
\end{proposition}

To classify 3-bridge quadrisections, we introduce some vocabulary. Let $\Tcal$ and $\Tcal'$ be two bridge quadrisections of 3-manifolds $Y$ and $Y'$ in $S^5$. 
We can form either their \emph{connected sum} $Y\# Y'$ or their \emph{distant sum} $Y\sqcup Y'$ obtained by connect summing their ambient 5-manifolds along 5-balls away from or inside $Y$ and $Y'$, respectively. At the level of bridge quadrisections, one can build a bridge quadrisection for $Y_1\sqcup Y_2$ by taking the connected sum along one puncture of the central 2-sphere away from the tangles of $Y$ and $Y'$. If we choose a 5-ball neighborhood of a puncture of the central 2-spheres, we obtain a bridge quadrisection for $Y\# Y'$; see \cite[\S 2.2]{meier17} or \cite[Remark 2.6]{aribi23}.

If $(B,T)$ is a trivial tangle, a \emph{c-disk} is a properly embedded disk $D$ in $B$, transverse to $T$, with boundary a non-trivial loop in the punctured $\partial B$ and $|D\cap T|\leq 1$. The classic work of Birman and Hilden explains a correspondence between genus-two handlebodies and 3-bridge trivial tangles \cite{BH75, Haas89_genustwo}. This correspondence relates meridians of a handlebody with c-disks for a 3-bridge tangle~\cite[\S 4.3.3]{pants}. Curves in punctured spheres bounding c-disks in all tangles of a 4-plane diagram indicate that the underlying bridge quadrisection is a sum of lower bridge quadrisections. This observation is key in  \autoref{prop:3-bridge_classification}, which is a 5-dimensional analog of Theorem 1.8 of \cite{meier17}.

\begin{proposition}\label{prop:3-bridge_classification}
Every $3$-bridge quadrisection of $(S^5,Y^3)$ is either a distant sum or connected sum of lower bridge quadrisections. 
\end{proposition}

\begin{proof}
Let $(B_1,T_1)\cup (B_2,T_2)\cup (B_3,T_3)\cup (B_4,T_4)$ be the spine of the 3-bridge quadrisection of $(S^5,Y)$. The 2-fold branched cover of each 3-ball $B_i$ branched along the tangle $T_i$ is a quadruple of 3-dimensional handlebodies $(H_1,H_2,H_3,H_4)$ with common boundary a genus-two surface. Thus, we get a genus-two quadrisected 5-manifold $Z$. By an upcoming result of Meier, Moussard, and Zupan, genus-two quadrisections of 5-manifolds are connected sums of genus-one splittings~\cite{Meier_Moussard_Zupan_personal}. This means that there is a separating curve in the genus-two surface bounding a meridian disk in all handlebodies. Such a curve descends to a non-trivial loop in the six-puncture sphere $\partial B_i$ bounding a c-disk for the tangle $T_i$. Hence, the bridge 4-section of $(S^5,Y)$ is either a distant sum or a connected sum of lower bridge quadrisections, as desired. 
\end{proof}

\subsubsection{Crossingless \texorpdfstring{$4$}{4}-plane diagrams.}
Similar to knots in 3-dimensions, one can filter 3-manifolds by the number of crossings in their 4-plane diagrams. Manifolds with the smallest crossing number (zero) are \emph{unknotted} by \autoref{prop:crossingless_4planes}.
In particular, 4-plane diagrams of non-trivial lens spaces $L(p,q)$ must have crossings, since these manifolds do not embed in $S^4$~\cite{Hantzsche_obstruction_no_Lpq_embeds_in_S4}.

\begin{proposition}\label{prop:crossingless_4planes}
    If a $3$-manifold $Y^3\subset S^5$ admits a crossingless quadrisection diagram, then $Y^4$ can be isotoped into $S^4$.
\end{proposition}
\begin{proof}
    We can embed the union of the 3-balls of the quadrisection in $S^3$ since it is crossingless. The 3-balls cut $S^3$ into four 3-balls $Z_{12},Z_{13},Z_{24},Z_{34}$ and each trivial disk system either embeds in $Z_{ij}$ or embeds in the spine itself. In conclusion, the Heegaard surface for $Y$ and $D_{\alpha}$ embed in $S^3.$ By definition of a quadrisection, each triple union of tangles is a collection of trivial 2-spheres and we can cap them off with four trivial 3-ball systems $B^3_{123}, B^3_{124},B^3_{234}$ and $B^3_{134}$. 

    As shown in \autoref{fig:lieinoneside}, the 3D projection of $B^3_J$ may a priori intersect $Z_{I\not\subset J}$. However, after an isotopy preserving the knot-type of $Y^3$ (as shown schematically in \autoref{fig:lieinoneside}), $B^3_I$ can be made to lie completely in $Z_{I\subset J}$. After performing these isotopies for $B^3_{I}$'s, we can embed $Y$ in $S^4$ as follows. We consider an embedding in $S^3\times [0,1],$ where the Heegaard surface is present in $S^3\times t$ for all  $t \in [0,1]$. Then,  $B^3_{123}\cup B^3_{234}$ is embedded in $S^3\times \{0\}$ and $B^3_{124}\cup B^3_{134}$ is embedded in $S^3\times \{1\}$.
\end{proof}
\begin{figure}[ht!]
    \centering    \includegraphics[width=.3\textwidth]{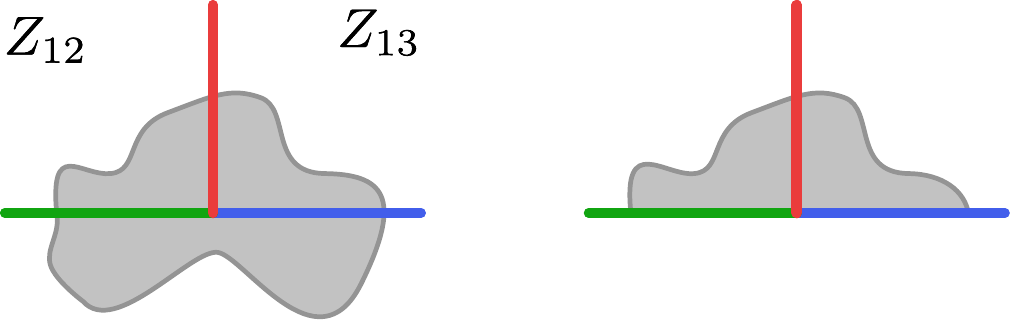}
    \caption{Treating a triple of crossingless tangles as a trisection, it is possible to cap off with trivial disk systems to form a closed surface, which is a trivial link of 2-spheres. Since the surface-link is trivial, we can fill in with a collection of 3-balls, which a priori lies in the union of the 3 balls $Z_{12}\cup Z_{13}\cup Z_{24} \cup Z_{34}$. The intersection $B^3_{123}$ with $Z_{24} \cup Z_{34}$ is a collection of 3-balls that trace an isotopy so that $B^3_{123}$ lies in $Z_{12}\cup Z_{13}$.}
    \label{fig:lieinoneside}
\end{figure}

\begin{question}
    What is the smallest crossing number of a non-trivial 3-knot?
\end{question}

\subsection{Lens spaces} 

Lens spaces are the closed 3-manifolds admitting genus-one Heegaard splittings. \autoref{fig:lens_space} shows a bridge 4-section $\Tcal=(T_1,T_2,T_3,T_4)$ of an embedding of $L(p,1)$. The bridge quadrisected surface by $\Tcal$ is a torus, as the abstract bridge 4-section corresponding to $\Tcal$ is shown in the left panel of \autoref{fig:lens_space}. Notice that the \textcolor{red}{red}-\textcolor{ForestGreen}{green} and \textcolor{blue}{blue}-\textcolor{purple}{purple} curves form the standard genus-one Heegaard diagram of $L(p,1)$, where each curve appears twice. 

\begin{figure}[ht]
    \centering
    \includegraphics[width=1\textwidth]{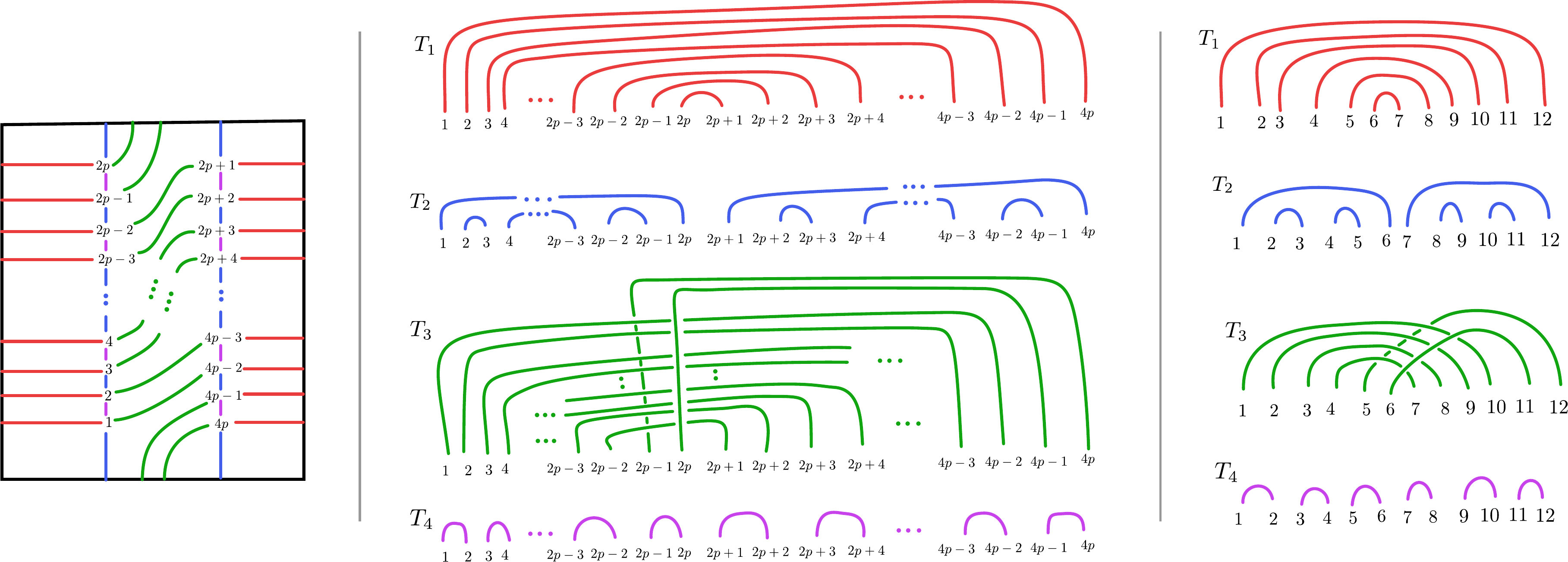}
    \caption{Left: an abstract 4-section of a torus. The \textcolor{blue}{blue}-\textcolor{purple}{purple} curves are meridians while the \textcolor{red}{red}-\textcolor{ForestGreen}{green} have slope $p/1$. Middle: a bridge 4-sected embedding of the lens space $L(p,1)$ in $S^5$, $p\geq 1$. Right: the particular case $p=3$.}
    \label{fig:lens_space}
\end{figure}

To formally check that $\Tcal$ from \autoref{fig:lens_space} determines an embedded 3-manifold, one needs to check that each tuple $(T_i, T_j, T_k)$ is a triplane diagram for an unlink of 2-spheres. We do this for the $p=3$ case and leave it as an exercise for the reader to generalize the figures for arbitrary $p\geq 2$: \autoref{fig:lens_space_check2} and \autoref{fig:lens_space_check1} show how to perform mutual braid moves to each tuple $(T_i, T_j, T_k)$ to obtain crossingless triplane diagrams, which describe unlinks of unknotted surfaces by Proposition 4.4 of \cite{meier17}. To end, an Euler characteristic computation checks that the respective surfaces are 2-spheres.

\begin{proposition} 
Let $\Tcal$ be a $4$-plane diagram describing an embedded lens space $L(p,q)$ in $S^5$, where $p\geq 2$. Then, 
\[ 
2p\leq \text{bridge}(\Tcal).
\]
In particular, the $4$-sections of $L(p,1)$ in \autoref{fig:lens_space} have the smallest possible bridge index. 
\end{proposition}
\begin{proof}

Choose a cyclic ordering $(T_1,T_2,T_3,T_4)$ of $\Tcal$; this yields a 4-plane diagram for an embedded surface $\Sigma\subset S^4$. From \autoref{prop:Heegaard_splitting_from_4section}, the tuple $(\Sigma;\alpha, \beta)$ is an extended\footnote{Recall that this means that the curves may be linearly dependent in $H_1(\Sigma;\Z)$.} Heegaard diagram for an abstract $L(p,q)$, where $\alpha=T_1\cup \T_3$ and $\beta=T_2\cup \T_4$ as subsets of $\Sigma$. Since lens spaces are orientable $\Sigma$ is an orientable surface. An orientation of $\Sigma$ translates to a consistent choice of signs for the $2b$ endpoints $\{p_1,\dots, p_{2b}\}$ of the tangles in $\Tcal$; $b=\text{bridge}(\Tcal)$. Thus, the algebraic intersection of curves in $\alpha$ and $\beta$ is equal to the signed count of the common punctures in $\{p_i\}_{i=1}^{2b}$. In particular, the number of intersections between the $\alpha$ and $\beta$ curves is at most the number of punctures of $\Tcal$.

Consider the matrix of algebraic intersections $Q=\left(\alpha_i\cdot \beta_j\right)_{i,j}$, where $\{\alpha_i\}_{i=1}^n$ and $\{\beta_j\}_{j=1}^m$ are the connected components of $\alpha$ and $\beta$. Since the fundamental group of $L(p,q)$ is cyclic, the Smith normal form of $Q$ must have exactly one non-zero entry (equal to $p$). In fact, the greatest common divisor of the entries in $Q$ is equal to $p$~\cite{norman2012SmithForms}. Take a pair of curves $\alpha_i$ and $\beta_j$ with non-zero algebraic intersection; i.e., $|\alpha_i\cdot \beta_j|\geq p$. 
Since $\alpha$ and $\beta$ were obtained from a bridge 4-section, we know that $\alpha$ and $\beta$ are nullhomologous sets in $H_1(\Sigma)$. Thus, there must be subsets of curves $\alpha_I\subset \alpha\setminus\alpha_i$ and $\beta_J\subset \beta\setminus \beta_j$ with $[\alpha_i]=[\alpha_I]$ and $[\beta_j]=[\beta_J]$ in $H_1(\Sigma)$. In particular, there exist $i_0\in I$ and $j_0\in J$ such that the quantities $\alpha_{i_0}\cdot \beta_j$, $\alpha_i\cdot \beta_{j_0}$, and $\alpha_{i_0}\cdot \beta_{j_0}$ are non-zero. Hence, 
\[
2b= |\alpha\cap\beta| \geq |\alpha_{i}\cdot \beta_{j}| + |\alpha_{i_0}\cdot \beta_{j}| + |\alpha_{i}\cdot \beta_{j_0}| + |\alpha_{i_0}\cdot \beta_{j_0}| \geq p + p + p + p.
\]
\end{proof}

\begin{figure}[ht]
    \centering
    \includegraphics[width=.85\textwidth]{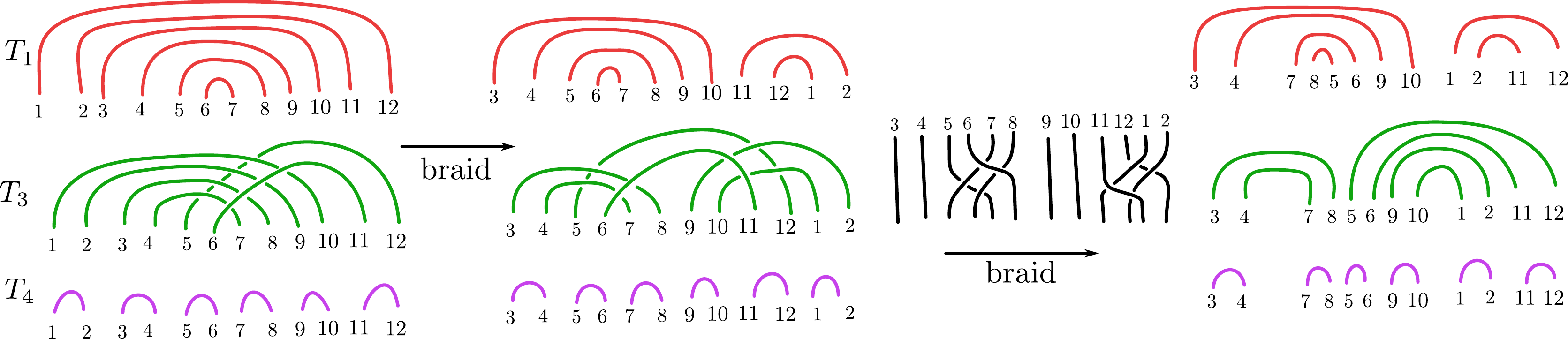}
    \caption{The tangles $(\textcolor{red}{T_1},\textcolor{blue}{T_2},\textcolor{ForestGreen}{T_3},\textcolor{purple}{T_4})$ from \autoref{fig:lens_space}, and mutual braid moves turning $(\textcolor{red}{T_1},\textcolor{ForestGreen}{T_3},\textcolor{purple}{T_4})$ into a crossingless triplane diagram.}
    \label{fig:lens_space_check2}
\end{figure}

\begin{figure}[ht]
    \centering
    \includegraphics[width=.85\textwidth]{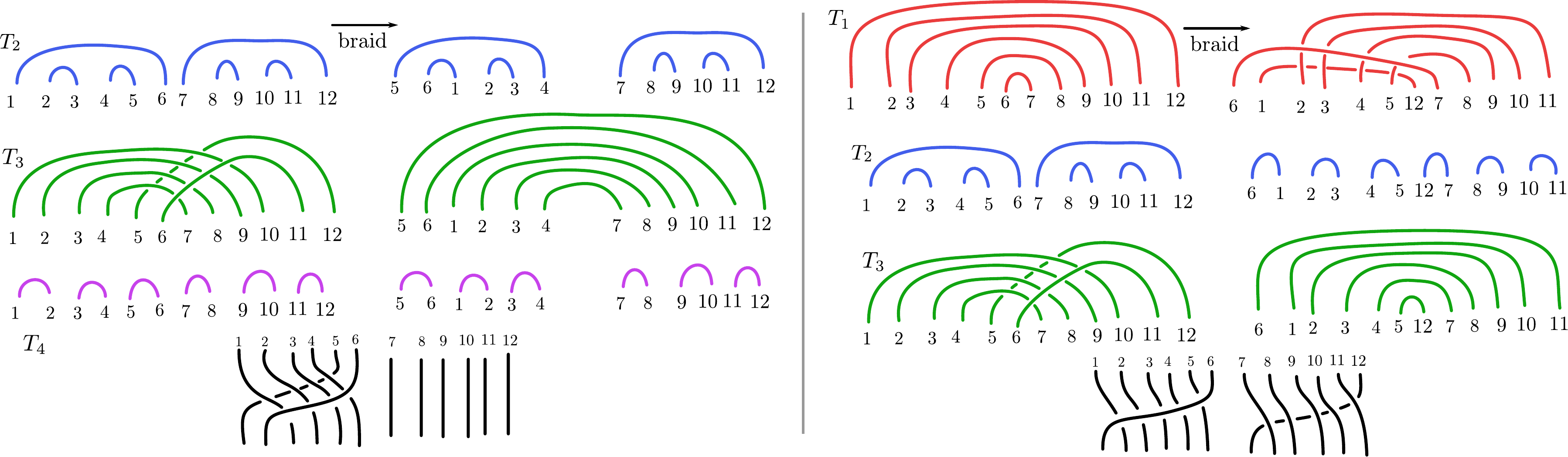}
    \caption{The tangles $(\textcolor{red}{T_1},\textcolor{blue}{T_2},\textcolor{ForestGreen}{T_3},\textcolor{purple}{T_4})$ from \autoref{fig:lens_space}. Left: mutual braid moves turning $(\textcolor{blue}{T_2},\textcolor{ForestGreen}{T_3},\textcolor{purple}{T_4})$ into a crossingless triplane diagram. Right: mutual braid moves turning $(\textcolor{red}{T_1},\textcolor{blue}{T_2},\textcolor{ForestGreen}{T_3})$ into a triplane diagram equal to the mirror image of $(\textcolor{ForestGreen}{T_3},\textcolor{purple}{T_4},\textcolor{red}{T_1})$.}
    \label{fig:lens_space_check1}
\end{figure}

\subsection{Embeddings of \texorpdfstring{$3$}{3}-manifolds in  \texorpdfstring{$\R^3\times \R$}{R3 x R}}\label{sec:example_embeddings_in_R3xR}

We now consider Heegaard complexes $(\Sigma;\Da,\Db)$ \emph{embedded in 3-space}; that is, each one of $\Sigma$, $\Da$, and $\Db$ embeds in a fixed $\R^3$ in the equatorial $S^4\subset S^5$. An example of such a Heegaard complex for the 3-torus in $\R^4$ is shown in \autoref{fig:Example_T3_1}. One can check that the underlying embedded 3-manifolds can be isotoped into $\R^4\subset S^4$ with the property that the projection onto the fourth coordinate $f:Y\subset \R^4\to \R$ is a Morse function. Agol and Freedman used the curve complex to find an obstruction for such an embedding to exist~\cite{Agol21_Embedding_Heegaard}. 

In what follows, we will explain how to find a bridge 4-section of $Y\subset \R^4$ from a Heegaard complex embedded in 3-space. Let $(\Sigma;\Da,\Db)$ be a Heegaard complex with $\Sigma$, $\Da$, and $\Db$ embedded in a fixed $\R^3$; denote the intersections of $\Sigma$ with the disk sets by $\alpha=\Da\cap \Sigma$ and $\beta=\Db\cap \Sigma$. 
First, alter the disk sets so that $\Sigma\setminus \left(\alpha\cup \beta\right)$ is a disjoint union of 2-disks. One way to achieve this is to artificially add bigons between $\alpha$ and $\beta$ curves so that they fill the surface. Then, consider $\alpha'=\partial \eta(\alpha)$ and $\beta'=\partial\eta(\beta)$ be the new multicurves resulting from doubling each curve in $\alpha$ and $\beta$. Note that each intersection point between $\alpha$ and $\beta$ turns into a small 4-gon. We color the arcs of $\alpha'\cup \beta'$ that connect the intersections $\alpha'\cap\beta'$ in one of four colors as follows: arcs in $\alpha'$ (resp. $\beta'$) that lie in the new 4-gons are green (resp. purple), and the rest of the arcs are red (resp. blue); see \autoref{fig:Example_T3_1} for reference. To end, drag the curves $\alpha'\cap \beta'$, without altering the intersection pattern, so that the intersection points $\alpha'\cap \beta'$ lie in a plane $P\subset \R^3$ tangent to $\Sigma$. By a small isotopy of $\Sigma$ we can choose $P$ such that $P\cap \Sigma=\alpha'\cap \beta'$ and $\Sigma$ lies in one side of $P$. This way, the colored arcs determine a tuple of tangles $(T_1,T_2,T_3,T_4)$ given by the ordering (\textcolor{red}{red}, \textcolor{blue}{blue}, \textcolor{ForestGreen}{green}, \textcolor{purple}{purple}). In \autoref{fig:Example_T3_1}, the plane $P$ is parallel to the paper, and the surface, together with the tangles $T_i$, is pushed into the paper. 

\begin{figure}[ht]
    \centering
    \includegraphics[width=.7\textwidth]{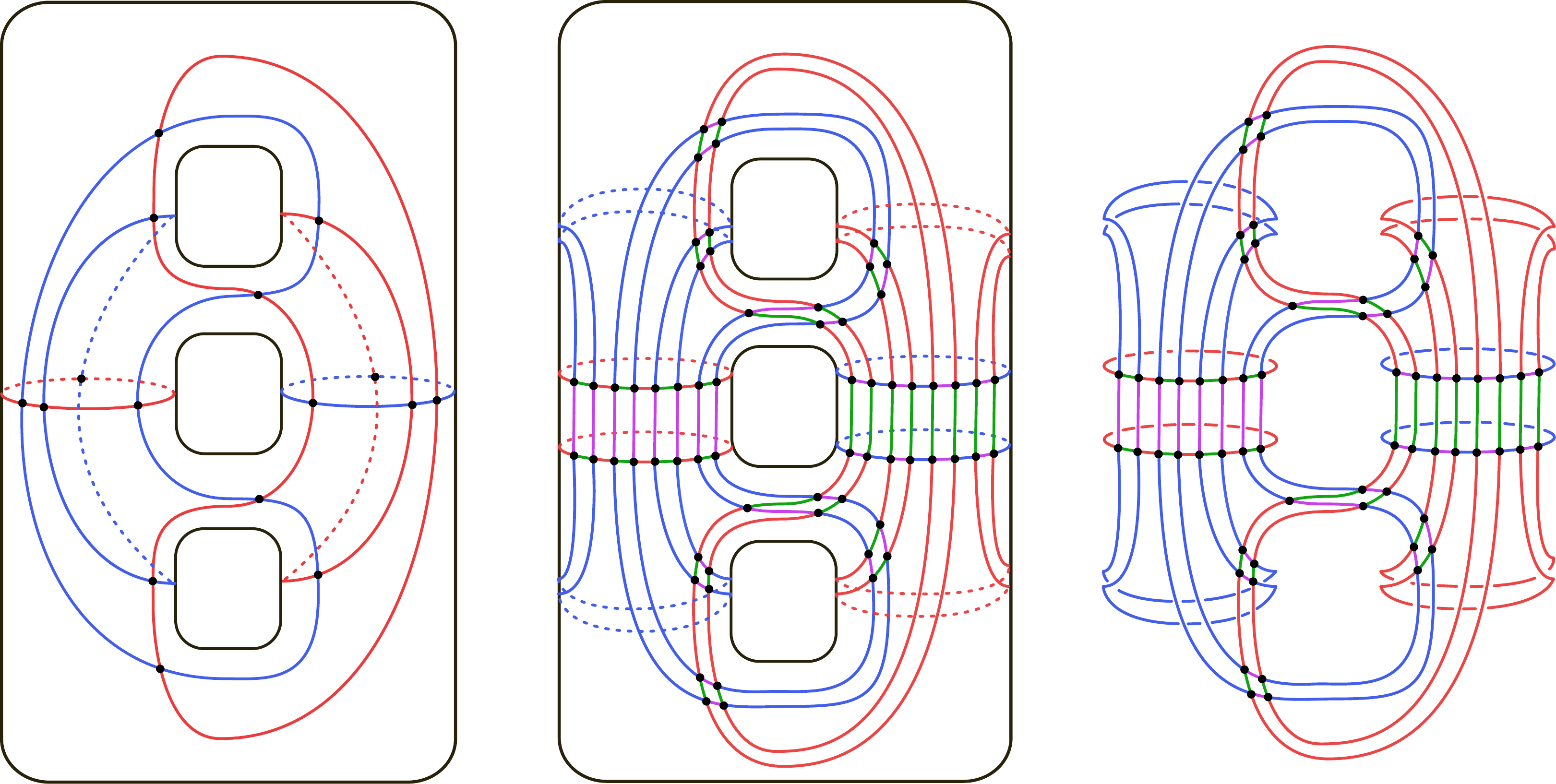}
    \caption{Left: a Heegaard diagram $(\Sigma;\alpha, \beta)$ of the 3-torus. Note that, looking at the way $\Sigma$ is drawn in $\R^3$, each curve bounds a compressing disk for $\Sigma$ embedded in $\R^3$. Thus, we have a Heegaard complex for a 3-torus in $\R^4=\R^3\times \R$. Middle: doubling each curve in $\alpha\cup \beta$ yields an abstract 4-section for $\Sigma$. This is drawn so that all the punctures lie in the same plane $P\subset \R^3$; the rest of $\Sigma$ and the arcs lie behind $P$. Right: after forgetting $\Sigma$, colored arcs can be thought of as tangles with endpoints in $P$. According to \autoref{prop:example_embeddings_in_R3xR}, this is the spine for a 32-bridge 4-section of $T^3$.}
    \label{fig:Example_T3_1}
\end{figure}

\begin{figure}[ht]
    \centering
    \includegraphics[width=.5\textwidth]{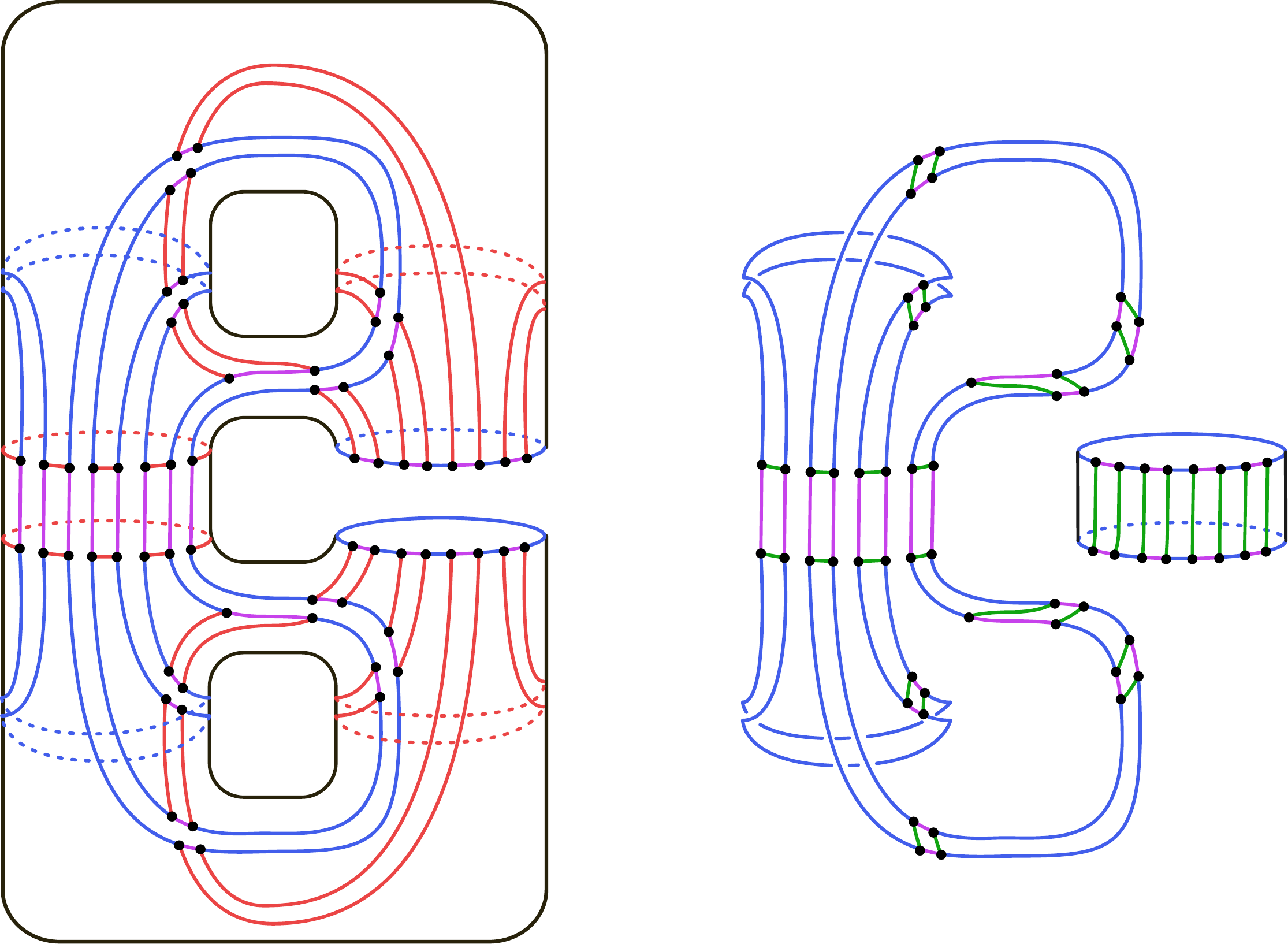}
    \caption{Consider the tangles $(\textcolor{red}{T_1},\textcolor{blue}{T_2},\textcolor{ForestGreen}{T_3},\textcolor{purple}{T_4})$ from \autoref{fig:Example_T3_1}. The triplets $(\textcolor{red}{T_1},\textcolor{blue}{T_2},\textcolor{purple}{T_4})$ (shown in the left panel) and $(\textcolor{blue}{T_2},\textcolor{ForestGreen}{T_3},\textcolor{purple}{T_4})$ (shown in the right panel) are subsets of surfaces embedded in $\R^3$. The left surface is equal to $\Sigma$ compressed along the disk set $\Da$, and the right surface is equal to the boundary of a tubular neighborhood of $\Da$.}
    \label{fig:Example_T3_2}
\end{figure}

\begin{proposition}\label{prop:example_embeddings_in_R3xR}
The tuple $\Tcal=(T_1,T_2,T_3,T_4)$ is the spine of a bridge $4$-section of $Y\subset S^5$.
\end{proposition}
\begin{proof}
Let $(\Sigma;\widetilde{D}_\alpha, \widetilde{D}_\beta)$ be the Heegaard complex of $Y$ obtained by doubling each disk; i.e., $\widetilde{ \mathcal{D}}_* = \mathcal{D}_*\sqcup \mathcal{D}_*$. 
By construction, unions of consecutive tangles bound disk components of $\Sigma\setminus (\alpha'\cup \beta')$. The union of non-consecutive pairs, (red, green) and (blue, purple), is equal to the boundaries of the disks $\widetilde{D}_\alpha$ and $\widetilde{D}_\beta$, respectively. Hence, each pair of tangles glues up to unlinks.

To end, we need to check that each triplet of tangles represents a sublink unlink of 2-spheres obtained by compressing $\Sigma$ along $\widetilde{D}_\alpha$ or $\widetilde{D}_\beta$. This is shown in \autoref{fig:Example_T3_2} for the 3-torus, where we observe that $(T_1,T_2,T_4)$ (no green) is a triplane for the surface $\Sigma|\Db$, and $(T_2, T_3, T_4)$ (no red) describes the boundary of a neighborhood of $\Db$. Thus, the disjoint union of the surfaces described by $(T_1,T_2,T_4)$ and $(T_2, T_3, T_4)$ are isotopic to $\Sigma|\widetilde{D}_\beta$. The same holds for the union of $(T_1,T_2,T_3)$ (no purple), $(T_1, T_3, T_4)$ (no blue), and $\Sigma|\widetilde{D}_\alpha$. Thus, the result follows from \autoref{lem:spine_determines_3man}.
\end{proof}

\begin{remark}
    The cautious reader may notice that the tangles built in this subsection may not always be trivial tangles. That said, as explained in \autoref{remark:0-sector_warning}, the tuple still describes a bridge 4-sected Heegaard surface in $S^4$ and 3-manifold in $S^5$. This issue can be solved by performing enough 0-sector perturbations as explained in that remark. The tangles in \autoref{fig:Example_T3_1} are trivial; thus, the bridge number of the 3-torus in $S^4\subset S^5$ is at most 32. 
\end{remark}

\section{Examples of knotted embeddings of 3-manifolds}\label{sec:knotted_examples}

To quadrisect more interesting 3-manifold embeddings, we first discuss how to construct Heegaard complexes for families of 3-knots. 
This is done in \autoref{subsec:doubles} and \autoref{subsec:spuns} for the double of ribbon handlebodies and $S^2$-spun knots, respectively. 
Before getting into the examples, in \autoref{subsec:tricks} we will present lemmas that serve as ``tricks'' that are helpful for isotoping Heegaard complexes into bridge position in practice.

\subsection{Trick lemmas}\label{subsec:tricks}

Recall the notation from \autoref{def:multisections_surfaces} for the pieces in a bridge 4-section of $(S^4,\Sigma)=\bigcup_{i=1}^4 (X_i,D_i)$. For $i\in \Z_4$, denote $X_{i,i+1,i+2}=X_{i}\cup_{B_{i+1}} X_{i+1}$, with $\partial X_{i,i+1,i+2}= B_i\cup \overline{B}_{i+2}$.

\begin{lemma}\label{lem:extension_lem_HC}
Let $(\Sigma;D_\alpha, \Db)$ be a Heegaard complex for $Y^3\subset S^5$. Let $\Tcal$ be a $4$-plane diagram of $\Sigma$ with 
\begin{enumerate}
    \item $\Da\subset B_1\cup \overline{B}_3$ and $\partial \Da \subset T_1\cup \T_3$, and
    \item $\Db\subset B_2\cup \overline{B}_4$ and $\partial \Db \subset T_2\cup \T_4$.
\end{enumerate}
Suppose that $(T_i, T_j, T_k)$ is a triplane diagram for an unlink of $2$-spheres for all $\{i,j,k\}\subset\{1,2,3,4\}$. Then $\Tcal$ is the spine of a bridge $4$-section of $Y^3\subset S^5$.
\end{lemma}

\begin{proof}
Let $\Sigma_{ijk}\subset S^4$ be the unlink of 2-spheres described by the triplane $(T_i,T_j,T_k)$. Let $D_{13}\subset B_1\cup \overline{B}_3$ be a collection of embedded disks satisfying $\partial D_{13}=T_1\cup \T_3$ and $\Da\subset D_{13}$. Such a set of disks exists since $(T_1,T_2,T_3)$ is a triplane diagram. In fact, since the 2-cells in a bridge trisected surface are boundary-parallel disks, we can assume that $D_{13}$ is a subset of $\Sigma_{123}$. For the same reason, $D_{13}\subset \Sigma_{341}$ and $\Sigma_{341}$ are obtained by gluing $D_{13}$ to $\Sigma\cap X_{341}$. Although this is not a transverse intersection for $\Sigma_{123}$, $\Sigma_{341}$ and $\Sigma$, this shows that compressing $\Sigma$ along $D_{13}$ yields the distant sum of $\Sigma_{123}\subset X_{123}$ and $\Sigma_{341}\subset X_{341}$; i.e., $\Sigma\mid D_{13}=\Sigma_{123} \sqcup \Sigma_{341}$ is an unlink of 2-spheres. 

Suppose that $\Da\neq D_{13}$ and let $D'_{13}=D_{13}-\Da$. Note that $\partial D'_{13}$ is a collection of simple closed curves embedded in the 2-spheres $\Sigma\mid \Da$. So there exist disks $B\in D'_{13}$ and $E\subset \Sigma\mid \Da$ with $\partial B=\partial E$ and $E$ having interior disjoint from $D'_{13}$. In particular, compressing $\Sigma \mid \Da$ along $B$ will create a 2-sphere component $S$ that is isotopic to a component of $\Sigma \mid \Da \mid D'_{13}= \Sigma \mid D_{13}$. Since we know that $\Sigma\mid D_{13}$ is an unlink of 2-spheres, $S$ must be unknotted and $\Sigma\mid (\Da\cup \{B\})$ is an unlink of 2-spheres. This shows that $\left(\Sigma; \Da\cup \{B\}, \Db\right)$ is also a Heegaard complex for $Y^3$; in fact, adding $B$ to $\Da$ corresponds to adding a 2/3-canceling pair of handles to the Heegaard complex as in \autoref{def:canceling_handles} . 

Thus, we have shown that we can add some disks of $D'_{13}$ to $\Da$ while preserving the condition of being a Heegaard complex for $Y_3$. After finite iterations of this argument, we can conclude that $\Da=D_{13}$. The same argument gives us that $\partial \Db=T_2\cup \T_4$; finishing the proof of this lemma. 
\end{proof}

\begin{remark}[Moving disks around]\label{rem:bubling_disks}
Let $\Tcal$ be a 4-plane diagram for a surface $\Sigma\subset S^4$ and $D$ a compressing disk for $\Sigma$. Suppose that $D$ is embedded in the spine of the genus-zero 4-section of $S^4$ and the boundary of $D$ is a subset of the spine of $\Tcal$; i.e., $D\subset \bigcup_{i=1}^4 B_i$ and $\partial D\subset \bigcup_{i=1}^4 T_i$. In each 3-ball $B_i$, $D\cap B_i$ is the union of polygonal disks with boundary the union of arcs alternating between strands in $T_i$ and arcs in $\partial B_i$ as in \autoref{fig:double_ribbon_lemma_1} and \autoref{fig:double_ribbon_lemma_2}. The subdisks that are bigons 
correspond to \emph{bridge disks} for strands in $T_i$ like in \autoref{fig:double_ribbon_lemma_1} (B). 

We can exploit the bridge 4-section of $\Sigma$ to move subdisks of $D\cap T_i$ around. \autoref{fig:bubling_disks} shows an isotopy of a bigon subdisk $E$ in $D\cap T_i$ near the boundary parallel disks of $F\cap X_{i+1}$ that replaces $E$ with other subdisks inside $B_i\cup B_{i+1}$. At the level of the abstract 4-section of $\Sigma$, such an isotopy of $D$ corresponds to an isotopy of $\partial D$ through a 2-cell of $\Sigma$. The technical condition needed for such an isotopy to exist is for the bigon in $D\cap B_i$ to be contained in a disk bounded by $T_i\cup \T_{i+1}$~\cite[Theorem 1.1]{Lee_Bridgedisks}.
\end{remark}

\begin{figure}[ht]
    \centering
    \includegraphics[width=.95\textwidth]{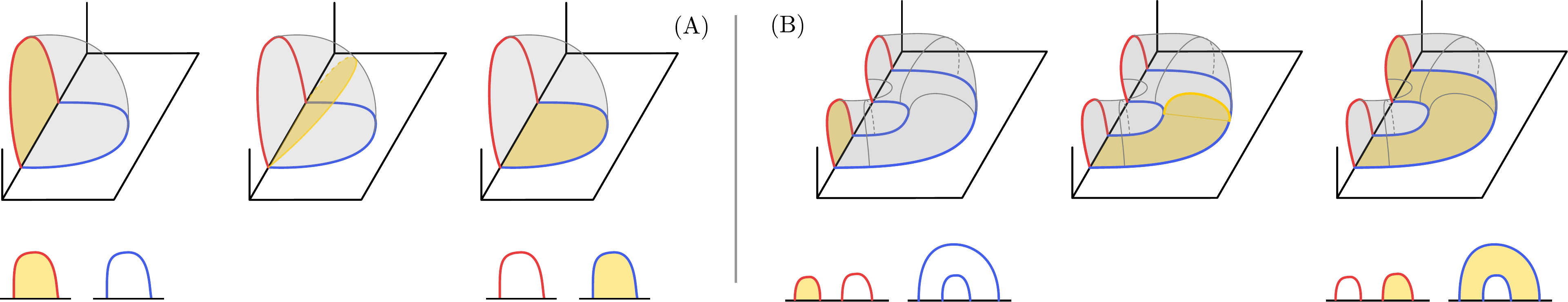}
    \caption{Top: each panel is a local picture of a boundary parallel disk of $F\cap X_{i+1}$ near the boundary $\partial X_{i+1}$. Sequences (A) and (B) depict an isotopy of the yellow surface in $S^4$ obtained by sliding the boundary of the surface through the trivial disk $F\cap X_{i+1}$. Bottom: the effect of this isotopy is to trade a bridge disk for the red tangle with a combination of spanning surfaces for the red and blue tangles.}
    \label{fig:bubling_disks}
\end{figure}

Suppose that some $\Da$ disks in a Heegaard surface are subsets of the 3-sphere cross-section $B_2\cup \overline{B}_4$. The following lemma will enable us also to see the disks in $D_\alpha$ inside the 3-sphere cross-section $B_1\cup \overline{B}_3$. To ease the presentation, we write (and use) the statement only for the case where the boundary of such a disk is a 1-bridge unknot. We leave it to the reader to write down the more general statement. 
    
\begin{lemma}\label{lem:moving_disk}
Let $\Tcal$ be a $4$-plane diagram for $F$. Assume that $T_2\cup \T_4$ contains a $1$-bridge unknot $\ell_{24}$ bounding a disk that bounds a disk $D\subset B_2\cup \overline{B}_4$.
Then, after perturbations of $\Tcal$ as in \autoref{fig:double_ribbon_lemma_1} (B)-(E), the cross-sections of the resulting $4$-plane diagram $\Tcal'$ have the properties that
\begin{enumerate}
\item as curves in $F$, $\ell_{24}$ is isotopic to components $\ell'_{13}\subset T'_1\cup \overline{T'_3}$ and $\ell'_{24}\subset T'_2\cup \overline{T'_4}$,
\item $D$ is isotopic in $S^4-F$ to a disk $D'_{24}\subset B_2\cup \overline{B}_4$ bounded by $\ell'_{24}$, and 
\item $D$ is isotopic in $S^4-F$ to a disk $D'_{13}\subset B_1\cup \overline{B}_3$ bounded by $\ell'_{13}$. 
\end{enumerate}
Furthermore, if $\Tcal$ is obtained by tubing a $4$-plane diagram $\Tcal_0$ 
with meridian equal to $D$, then $\Tcal'$ is obtained by tubing a $4$-plane diagram $\Tcal'_0$ with meridian $D'_{13}$, where $\Tcal'_0$ is obtained by 0-perturbations of $\Tcal_0$.
\end{lemma}

\begin{proof}
Panels (B) to (E) of \autoref{fig:double_ribbon_lemma_1} contain a sequence of perturbations, which do not change the 4-sected surface, turning $\Tcal$ in panel (B) into $\Tcal'$ in panel (E). 
Let $\ell'_{24}\subset T'_2\cup \overline{T'_4}$ and $\ell'_{13}\subset T'_1\cup \overline{T'_3}$ be the loops in $F$ passing through the punctures 1-8 and 6-7 punctures of $\Tcal'$, respectively. The bottom row of the figure depicts the strands of the tangles in the 4-sections as a subset of an abstract copy of $F$; note that the subsurface of $F$ represented is an annular neighborhood of $\ell_{24}$. In particular, condition (1) holds. 
Throughout the perturbations from $\Tcal$ to $\Tcal'$, the disk $D \subset B_2\cup \overline{B}_4$ is isotoped into a disk $D'_{24}\subset B_2\cup \overline{B}_4$ shown in panel (E); see the shaded disks in \autoref{fig:double_ribbon_lemma_1}.
%
\begin{figure}[ht]
    \centering
    \includegraphics[width=.9\textwidth]{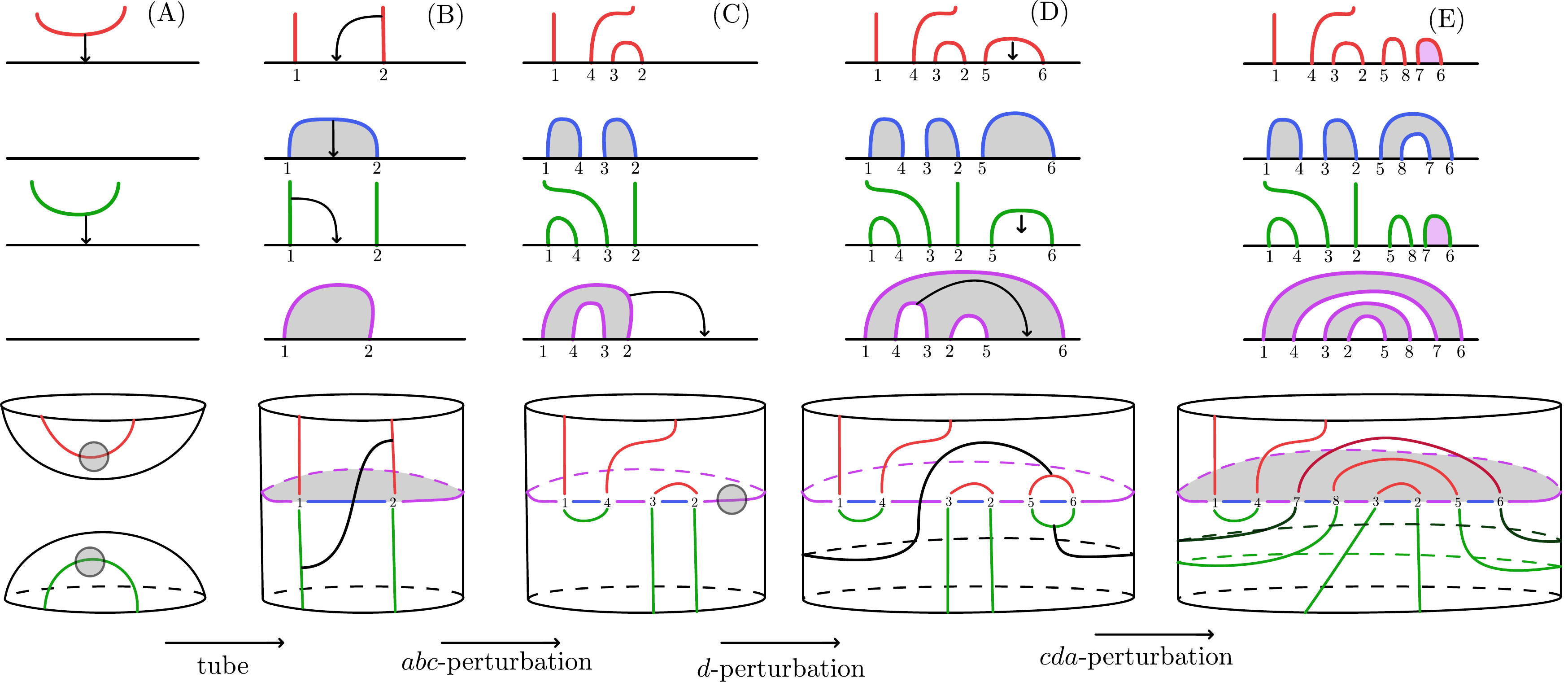}
    \caption{In (B)-(E): a sequence of perturbations of a bridge 4-sected surface in a neighborhood of a 1-bridge unknotted component of $\textcolor{blue}{T_2}\cup \textcolor{purple}{\T_4}$. The effect of this sequence, shown in panel (E), is a new 1-bridge unknotted component of $\textcolor{red}{T_1}\cup \textcolor{ForestGreen}{\T_3}$ that appears with the property that the disk it bounds is isotopic in $S^4$ to the 4-bridge unknotted component of $\textcolor{blue}{T_2}\cup \textcolor{purple}{\T_4}$. See \autoref{lem:moving_disk}. 1-bridge unknots as in (B) often appear as meridians of tubes shown in (A).}
    \label{fig:double_ribbon_lemma_1}
\end{figure}

For condition (2), note that in \autoref{fig:double_ribbon_lemma_2} we see a sequence of isotopies taking $\left(\ell'_{24}, D'_{24}\right)$ to $\left(\ell'_{13}, D'_{13}\right)$. Between each panel, we push the loop-disk pair through marked bicolored 2-cells of $F$; see \autoref{rem:bubling_disks}. 
The last part of the lemma follows from \autoref{fig:double_ribbon_lemma_3}, where $\Tcal'_0$ is the 4-section in panel (B) and $\Tcal_0$ is the 4-section in both panel (D) and \autoref{fig:double_ribbon_lemma_1} (A). 
\end{proof}
\begin{figure}[ht]
    \centering
    \includegraphics[width=.9\textwidth]{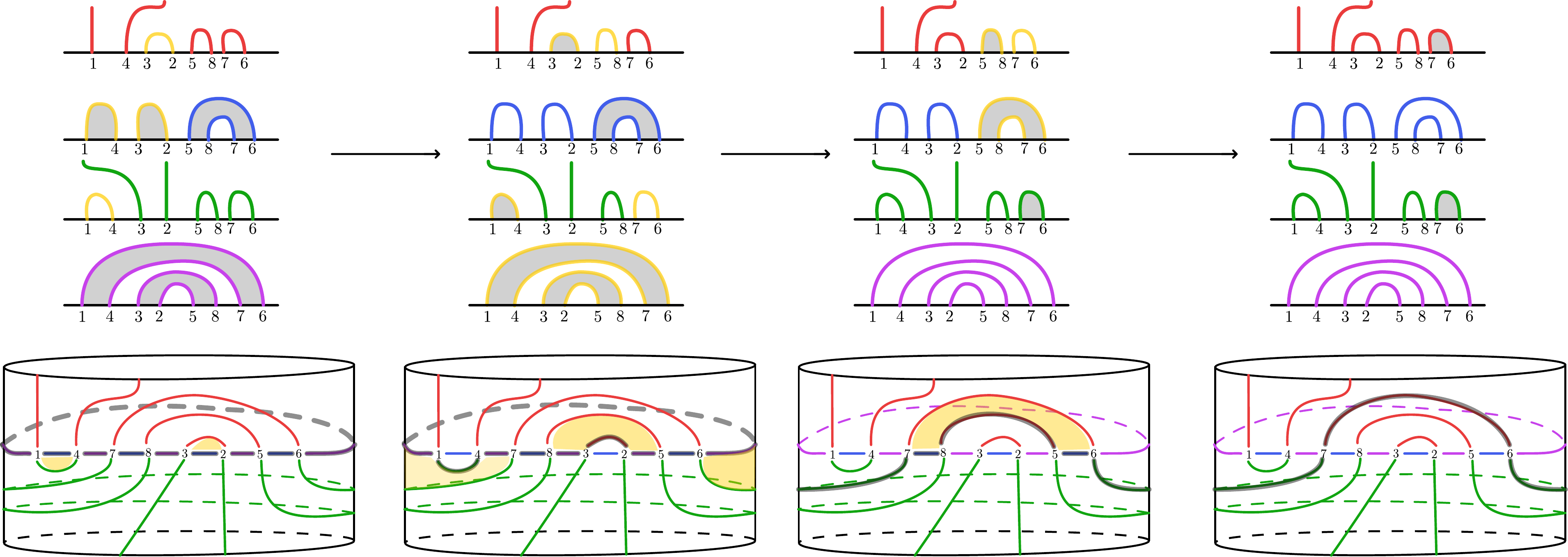}
    \caption{Sequence of isotopies of the gray shaded disk in the left panel using the modifications in \autoref{rem:bubling_disks} and \autoref{fig:bubling_disks}. The yellow shaded arcs in each panel bound subdisks of the bridge 4-sected surface $F$ we use to isotope our gray disk. The bottom row keeps track of boundary of the gray disk in an abstract 4-section of $F$. }
    \label{fig:double_ribbon_lemma_2}
\end{figure}
\begin{figure}[ht]
    \centering
    \includegraphics[width=.75\textwidth]{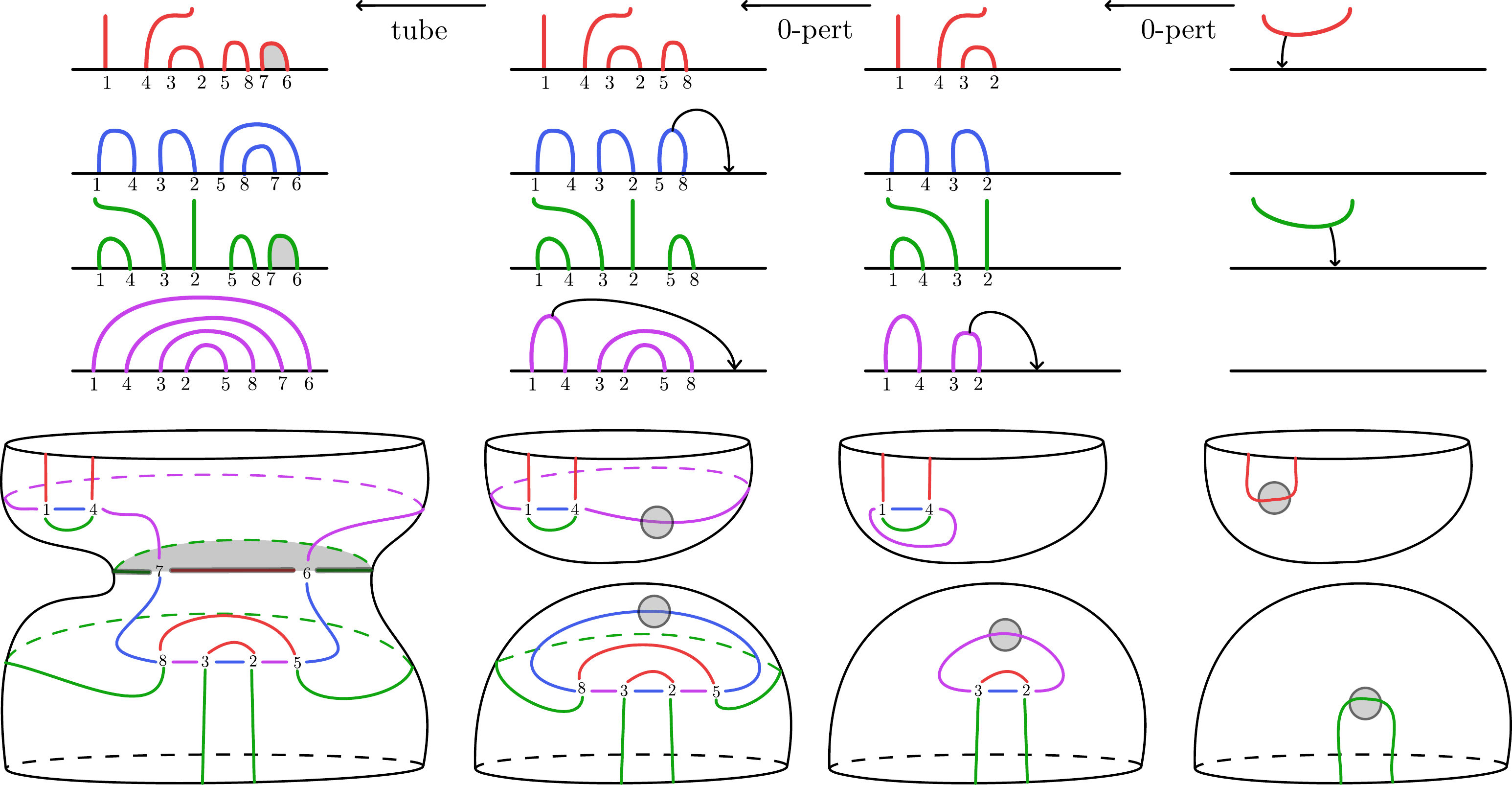}
    \caption{Alternative sequence of detubings and deperturbations transforming the 4-plane diagram in \autoref{fig:double_ribbon_lemma_1} (E) into that in \autoref{fig:double_ribbon_lemma_1} (A). Notice that the meridian of the tube used in this sequence is the pink-shaded disk. The bottom row keeps track of the abstract 4-section of $F$ throughout the sequence. }
    \label{fig:double_ribbon_lemma_3}
\end{figure}

\subsection{Doubles of ribbon \texorpdfstring{$3$}{3}-manifolds}\label{subsec:doubles}
Let $F$ be a connected ribbon surface in $S^4$. By definition, $F$ is built by adding finitely many tubes to an unlink of 2-spheres in $S^4$. In particular, $F$ bounds a 3-dimensional 1-handlebody in $B^5$ with the 2-spheres bounding the 0-handles and the tubes determining the 1-handles. Denote by $D_F$ the closed 3-manifold obtained by ``doubling'' the handlebody $H$ bounded by $F$, so that $D_F$ is an embedding of $\#_{g(F)} S^1\times S^2$ if $F$ is orientable and $\#_{1-\chi(F)} S^1\widetilde{\times} S^2$ if $F$ is non-orientable. 

To describe a bridge 4-section of $Y_F\subset S^5$, we first need to see a Heegaard complex of $Y_F$ ``inside'' a 4-plane diagram of $F$. To do this, one can consider an $m$-bridge 4-plane diagram in which all tangles are identical; such a diagram represents an $m$-component unlink of 2-spheres. 
The cores of the tubes that form $F$ can be isotoped to lie in the 3-sphere cross-section $B_1\cup \overline{B}_3$, where the link $T_1\cup \T_3$ lies. \autoref{fig:Kinoshita_Terasaka_2knot} (A) shows an example of this situation. 
We now modify the bridge quadrisection so the core of each tube intersects the bridge surface transversely at a single point; denote by $\Tcal_0$ the resulting 4-plane diagram for $\displaystyle \sqcup_{i=1}^m S^2$. This can be achieved by 0-sector perturbations along subarcs of the cores that shrink the tubes; see \autoref{fig:Kinoshita_Terasaka_2knot} (A)-(C). 
We now tube $\Tcal_0$ as in \autoref{sec:tubings}. According to \autoref{lem:tubing_4sec_with_bands}, the resulting 4-plane diagram $\Tcal_1$ describes $F$ and the new 1-bridge unknots in $T_2\cup \T_4$ bound a set of disks $\Da$ equal to meridians of the new tubes; see \autoref{fig:Kinoshita_Terasaka_2knot} (D) or the left panel of \autoref{fig:double_kinoshita_terasaka_disk}. In particular, since $Y_F$ is a doubled 3-manifold, the tuple $(F; \Da, \Db=\Da)$ is a Heegaard complex for $Y_F$. 

\begin{proposition}\label{prop:4section_double}
Let $F$ be a ribbon surface in $S^4$ and let $\Tcal_1$, $\Da$, be as above. For each $1$-bridge unknot in $\Da$, modify the $4$-plane diagram $\Tcal_1$ as in \autoref{fig:double_ribbon_lemma_1} (B)-(E). The resulting tangle, denoted by $\Tcal_2$, is the spine of a bridge $4$-section of $Y_F\subset S^5$. 
\end{proposition}
\begin{proof}
By \autoref{lem:moving_disk}, the disk sets $\Da$ and $\Db$ can be seen as subsets of $B_1\cup \overline{B}_3$ and $B_2\cup \overline{B}_4$, respectively, with $\partial \Da\subset T_1\cup \T_3$ and $\partial \Db \subset T_2\cup \T_4$. By \autoref{lem:extension_lem_HC}, it remains to show that removing any tangle from $\Tcal_2$ yields a triplane diagram of an unlink of 2-spheres. 

For $r=0,1,2$, denote the tangles of $\Tcal_r$ by $(T^r_1,T^r_2, T^r_3, T^r_4)$. We color-code our figures in the order $(\text{red}, \text{blue}, \text{green}, \text{purple})$. Panels (A), (B), and (E) of \autoref{fig:double_ribbon_lemma_1} are local models for $\Tcal_0$, $\Tcal_1$, and $\Tcal_2$ near the disks in $\Da$, respectively. By construction, $\Tcal_0$ is obtained by 0-perturbations of a 4-plane of the form $(T,T,T,T)$. This implies that removing a tangle from $\Tcal_0$ yields a triplane diagram for an unlink of 2-spheres. 
If we remove the red tangle $T^r_1$ from $\Tcal_r$, \autoref{fig:double_ribbon_lemma_1}, read from right-to-left, becomes a sequence of deperturbations taking $(T^2_2, T^2_3, T^2_4)$ into $(T^0_2, T^0_3, T^0_4)$, which a diagram for the unlink of 2-spheres. The same argument holds if we remove the green tangle $T^2_3$ from $\Tcal_2$. 
If we remove the blue tangle $T^r_2$ from $\Tcal_r$, \autoref{fig:double_ribbon_lemma_3} becomes a sequence of
deperturbations from $(T^2_1, T^2_3, T^2_4)$ to $(T^0_1, T^0_3, T^0_4)$. \autoref{fig:double_ribbon_deperts_no_purple} shows sequences of deperturbations from $(T^2_1, T^2_2, T^2_3)$  to distant sum of the triplane $(T^0_1, T^0_2, T^0_3)$ with 1-bridge trisected 2-spheres. 
\end{proof} 


\begin{figure}[ht]
    \centering
    \includegraphics[width=.8\textwidth]{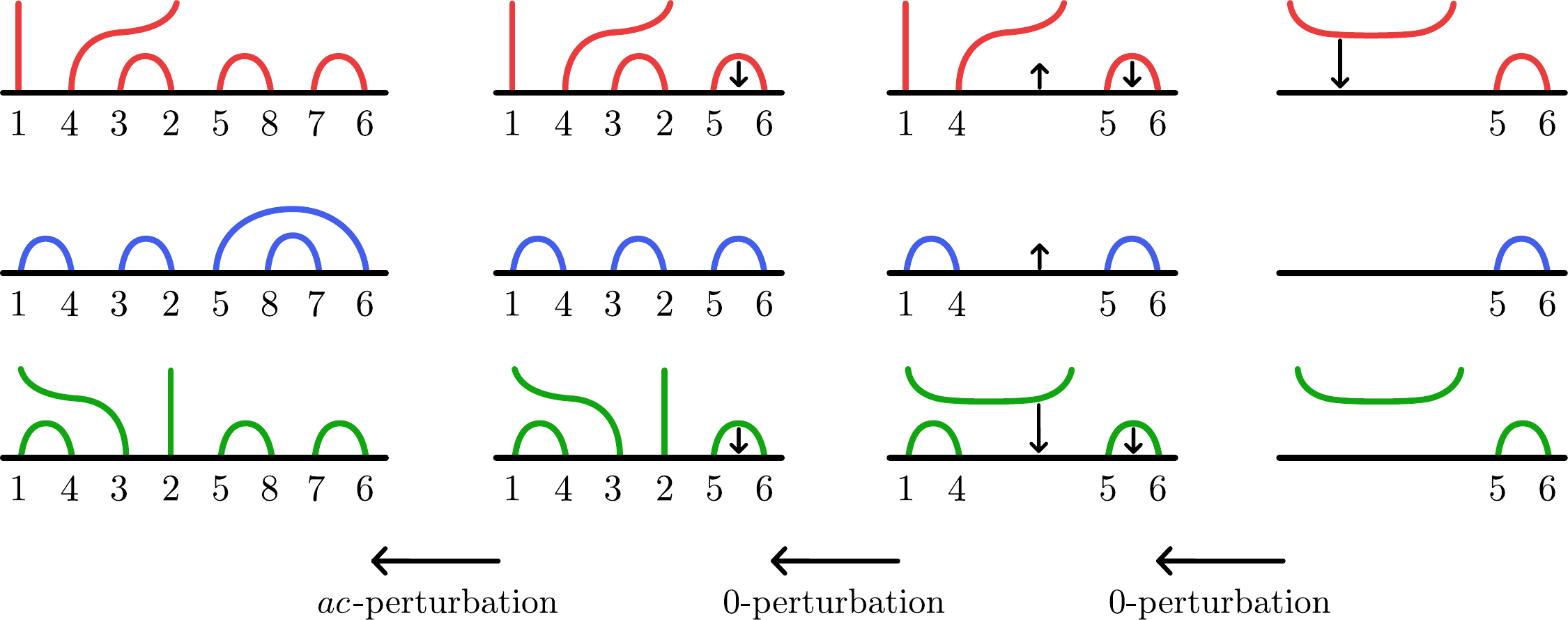}
    \caption{After removing the \textcolor{purple}{purple} tangle from $\Tcal_2$ in \autoref{fig:double_ribbon_lemma_1} (E), the left triplane deperturbs to a distant sum of the triplane $(\textcolor{red}{T^0_1},\textcolor{blue}{T^0_2},\textcolor{ForestGreen}{T^0_3})$ from \autoref{fig:double_ribbon_lemma_1} (A) with 1-bridge trisected 2-spheres with endpoints $\{5,6\}$.}
    \label{fig:double_ribbon_deperts_no_purple}
\end{figure}

\begin{example}\label{example:KT_double_double}
    Let $F_{KT}$ be the double of the ribbon disk of the Kinoshita-Terasaka ribbon presentation of $11_{n_{42}}$, which is a ribbon surface. \autoref{ex:KT_double_as_tubes} explained how to obtain a 4-plane diagram for $F_{KT}$ where the meridian of the tube bounds a 1-bridge unknot component of $\textcolor{blue}{T_2}\cup \textcolor{ForestGreen}{\T_4}$. In \autoref{fig:double_kinoshita_terasaka_disk}, we perform the procedure from \autoref{prop:4section_double} to obtain the spine of a bridge 4-section for the 3-knot $D(F_{KT})$. This is a diagram for a trivial 3-knot since $11_{n_{42}}$ is smoothly superslice~\cite{Livingston15_superslice}. We encourage the reader to find other ribbon surfaces $F$ for which $D_F$ is not trivial.
\end{example}

\begin{figure}[ht]
    \centering
    \includegraphics[width=.5\textwidth]{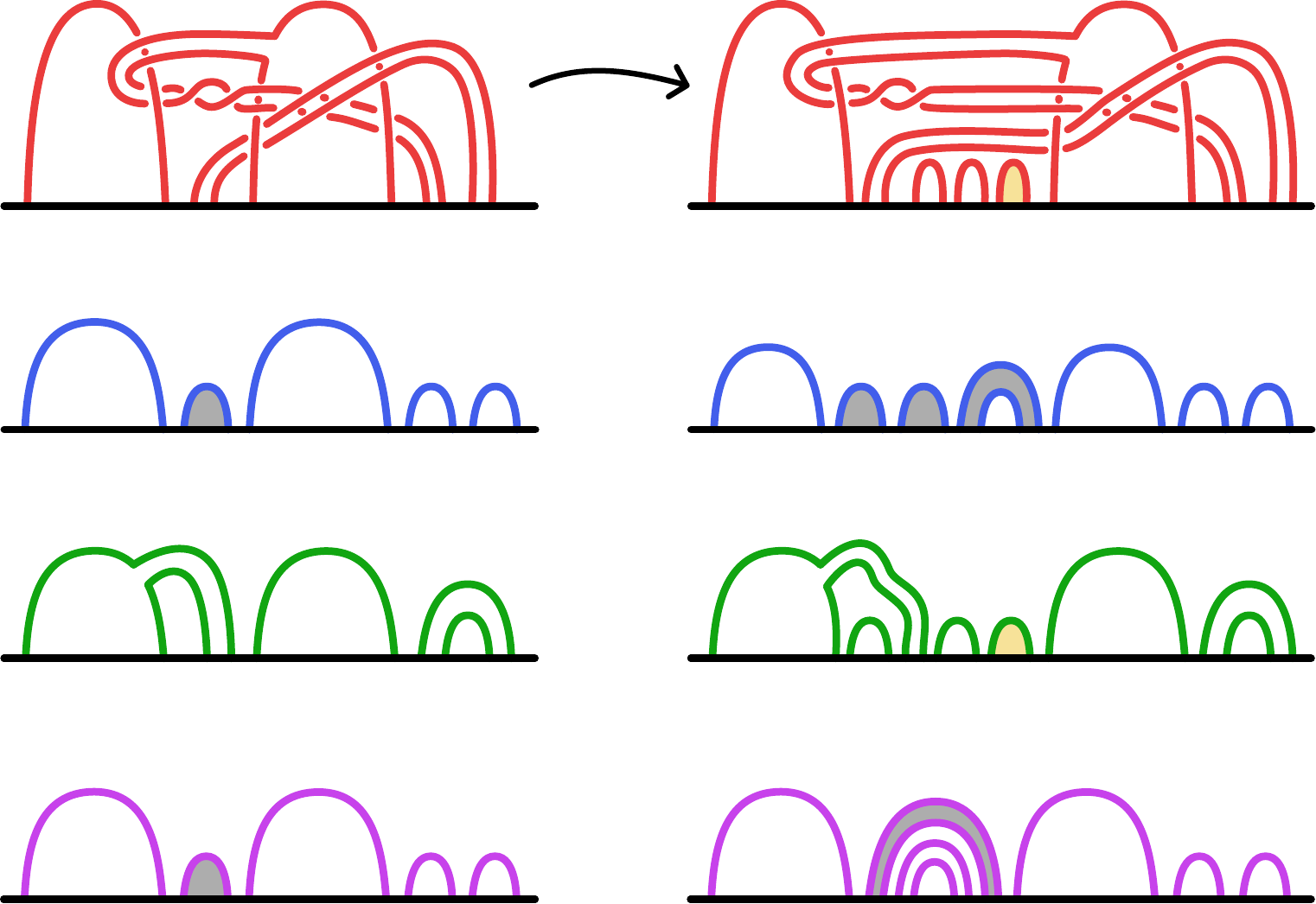}
    \caption{Left: a 4-plane diagram for the ribbon 2-knot $F_{KT}$. The shaded regions depict the meridian of the tube that undoes $F_{KT}$. Right: the spine of a bridge 4-section of the 3-knot $D(F_{KT})$; see \autoref{example:KT_double_double}. }
    \label{fig:double_kinoshita_terasaka_disk}
\end{figure}

\subsection{Spun \texorpdfstring{$3$}{3}-manifolds}\label{subsec:spuns}

Spinning is a process for building knotted objects from knots in lower dimensions, initiated by Artin a century ago~\cite{Artin_spinning}. To build 3-knots in $S^5$, we can spin knots in either $S^3$ or $S^4$, using the observation that $B^6$ can be written as $B^3\times B^3$ and $B^4\times B^2$~\cite{friedman_spinning}.

\subsubsection{Spinning knots}

Let $(S^3, K)$ be a knot, and let $(B^3, K^\circ)$ be the tangle obtained by removing a small open ball centered at a point in $K$. The boundary of $K^\circ$ is two points $\{N,S\}$ in $S^2$. Define the $S^2$-spin of $K^1$, denoted by $S^2(K)$, to be the embedded 3-manifold given by 
\[ 
\left(S^5, S^2(K)\right) = 
\left( B^3\times S^2, K^\circ\times S^2\right) \cup \left(S^3\times B^2, \{N,S\} \times B^2\right).
\] 
Alternatively, one can parametrize $S^2(K^1)\subset \R^5$ as follows: if $K^\circ$ has coordinates $\left(x(t), y(t), z(t)\right)$ with $x(t)\geq 0$, then $S^2(K)$ has spherical coordinates $(t, \theta, \phi)$ given by
\begin{equation}\label{eq:spherical_coords}
S^2(K)(t, \theta, \phi) = 
\left( 
x(t)\cos{\theta}\sin{\phi}, y(t), z(t), x(t)\sin{\theta}\sin{\phi}, x(t)\cos{\phi}
\right).
\end{equation}
Note that $S^2$-spinning works for links as well. If $K$ is connected, the $S^2$-spin is an embedded 3-sphere. If $K$ has more than one component, $S^2(K)$ is a link of one knotted 3-sphere and $(|K|-1)$ embeddings of $S^1\times S^2$. \autoref{prop:diagram_S2_spun} gives a procedure to find a 4-plane diagram for $S^2(K)$ by observing that $S^2(K)$ is the double of a ribbon 3-manifold in $B^5$.

\begin{figure}[ht]
    \centering
    \includegraphics[width=1\textwidth]{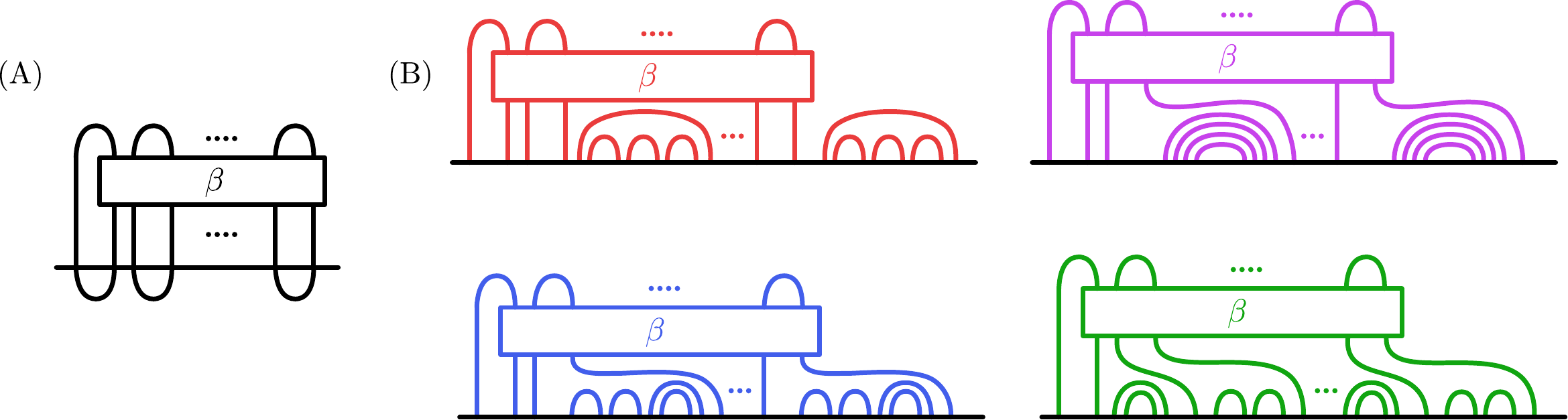}
    \caption{In (A): a plat projection of a link $K\subset S^3$. In (B): a quadrisection diagram for $S^2(K)\subset S^5$.}
    \label{fig:S2_Spun_3} 
\end{figure}

\begin{proposition}\label{prop:diagram_S2_spun}
    Let $K\subset S^3$ be a $b$-bridge knot or link. Then $S^2(K)$ admits a $(5b-4)$-bridge quadrisection with spine described in \autoref{fig:S2_Spun_3}.
\end{proposition}

\begin{proof}
We give the proof for $b=2$. 
Draw $K^\circ$ as a tangle with one local minima below two local maxima with heights 1, 2, and 3, respectively. Consider $S^2(K)\subset\R^5$ described by the spherical coordinates in \autoref{eq:spherical_coords}. Let $\pi \colon \R^5\to \R^3$ be the projection $\pi(x_1, \dots, x_5) = (x_1, x_4, x_5)$. The restriction of $\pi$ to $S^2(K)$ has three 2-spheres with critical values\footnote{In spherical coordinates, the determinant of the Jacobian of $\pi\mid_{S^2(K)}$ is $x^2x'\sin{\phi}$ with $0<\phi<\pi$.} of radii equal to 1, 2, and 3, corresponding to 
the three critical points of $K^\circ$; see~\autoref{fig:S2_Spun_1}. In particular, if $\alpha\subset \R^3$ is an unbounded path that starts at the origin and is transverse to each 2-sphere centered at the origin, then the preimage $h^{-1}(\alpha)$ is a tangle equivalent to $K^\circ$. 

\begin{figure}[ht]
    \centering
    \includegraphics[width=.75\textwidth]{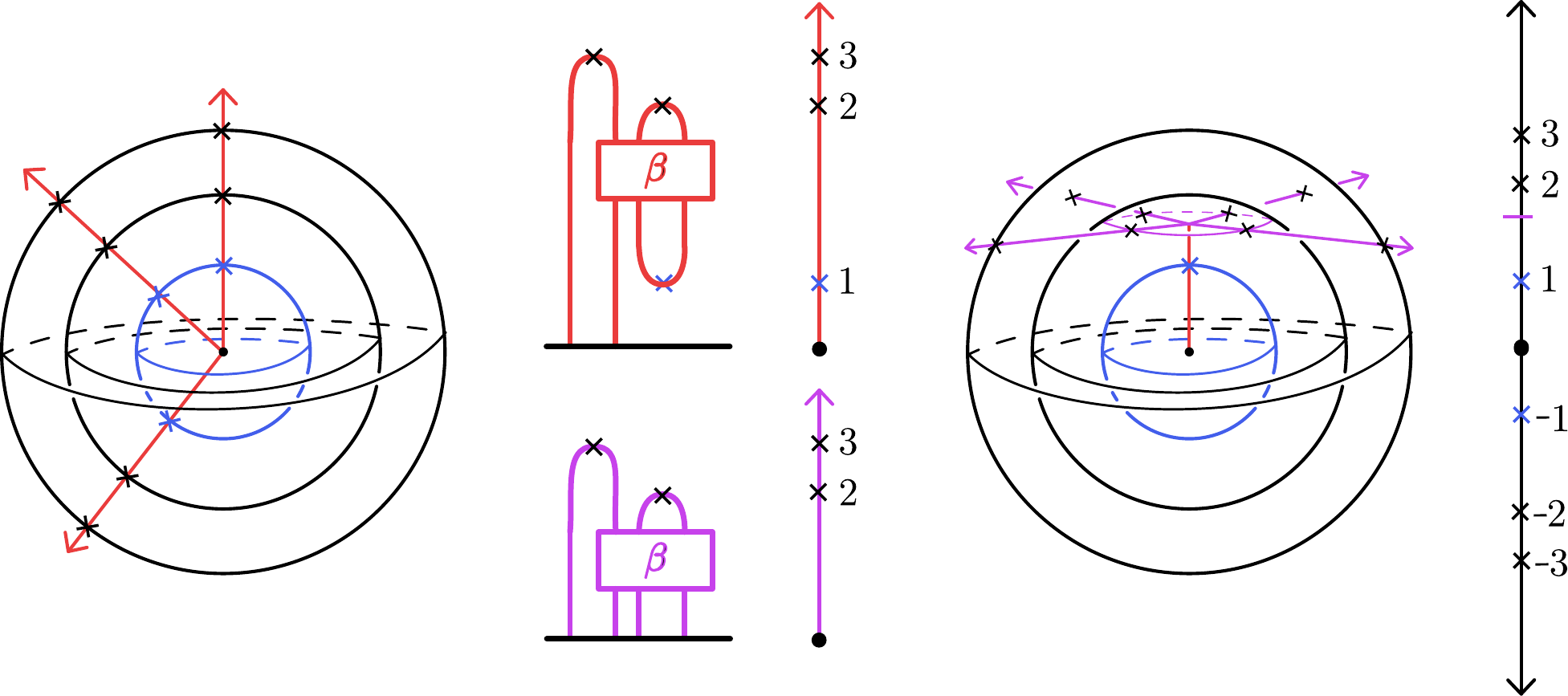}
    \caption{The critical values of the function $\pi\mid_{S^2(K)}$. The preimage of every red ray is equal to the red tangle $K^\circ$. The purple rays lie in a plane $x_3=1.5$ and are based in the $x_3$-axis; their preimage under $h$ is the purple tangle $K^\circ_{1.5}$ shown in the middle panel.}
    \label{fig:S2_Spun_1}
\end{figure}

Let $K^\circ_r$ be the subset of points in $K^\circ$ with height at least $r$, and let $\R^4_r=\{\vec{x}\in \R^5 : x_5=r\}$. For $r\in (0,3)\setminus \{1,2\}$, $K^\circ_r$ is a properly embedded tangle inside a 3-ball. From the discussion in the previous paragraph, we see that $S^2(K)\cap \R^4_r$ is the surface obtained by spinning the tangle $K^\circ_r$ around its boundary; see \autoref{fig:S2_Spun_1}. Moreover,
$\left\{
S^2(K)\cap \R^4_r : t\in (-\infty, \infty)
\right\}$ 
is a one-parameter family of surfaces in $\R^4$ tracing $S^2(K)$.  

The movie from $0$ to $\infty$ is the same as from $0$ to $-\infty$. As $r$ goes from $-\infty$ to $-1.5$, we see two births. As $t$ goes from $-1.5$ to $0$, the two endpoints of $K^\circ_r$ collide at $r=-1$ to form the local minimum of $K^\circ_r$. The respective spun surfaces $S^2(K)\cap \R^4_r$ change by a 1-handle addition along an arc that connects such endpoints. Panels (A) and (B) of \autoref{fig:S2_Spun_2} show banded unlink diagrams for $S^2(K)\cap \R^4_{-1.5}$ and $S^2(K)\cap \R^4_{0}$, respectively. The 4-plane diagrams of these surfaces describing the same 1-handle addition are shown in \autoref{fig:S2_Spun_2} (C-D). The shaded bigons in panel (D) glue to a disk $D$, which is the meridian of the 1-handle. In conclusion, the 4-sected surface in panel (D), together with disk sets $\Da=\Db=\{D\}$, forms a Heegaard complex for $S^2(K)$. Note that $S^2(K)$ is a double 3-manifold as in~\autoref{subsec:doubles}. To end, \autoref{prop:4section_double} describes how to modify panel (D) to get the spine for a bridge 4-section of $S^2(K)$. The final result is shown in \autoref{fig:S2_Spun_2} (E). 
\end{proof}

\begin{figure}[ht]
    \centering
    \includegraphics[width=.75\textwidth]{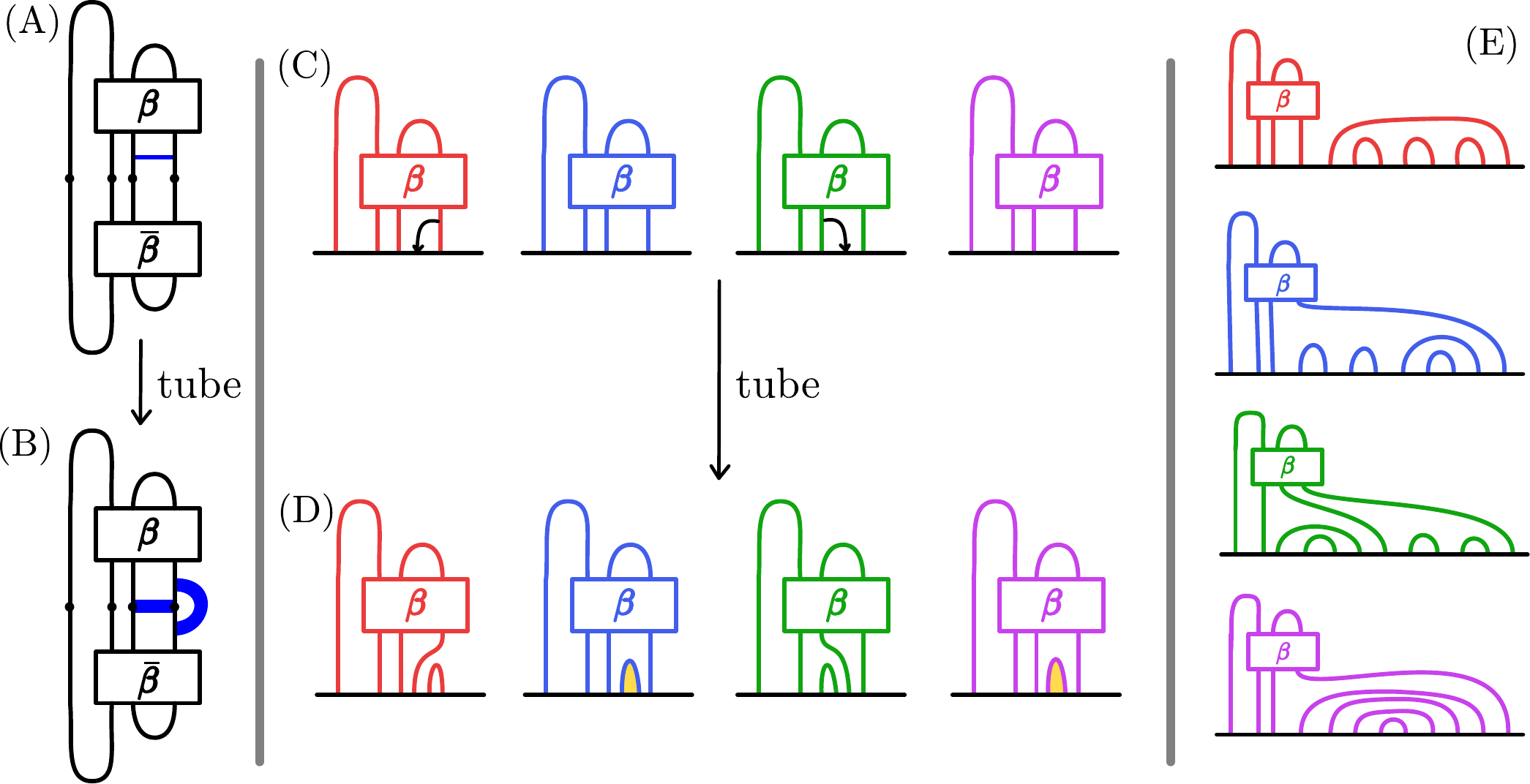}
    \caption{In (A): a banded unlink diagram, and in (C): a 4-plane diagram representing the surface $S^2(K)\cap \R^4_{-1.5}$, obtained by spinning the tangle $K^\circ_{-1.5}$ along its boundary. In (B): a banded unlink diagram and, in (D): a 4-plane diagram representing $S^2(K)\cap \R^4_{0}$, equal to the spin of $K$. These two surfaces differ by a 1-handle addition shown. In (D): a Heegaard complex for $S^2$-spun of $K$; the disk sets $\Da=\Db$ contain one copy of the shaded disk in the blue and purple tangles. In (E): a Heegaard complex for $S^2(K)$ in bridge position. }
    \label{fig:S2_Spun_2}
\end{figure}

\subsubsection{Spinning surfaces}

    Given an embedded surface $(S^4,F)$, let $(B^4,F^\circ)$ be the 2-tangle resulting from removing a small 4-ball centered in $F$. Then, the $S^1$-spin of $F$ is defined by 
    \[ 
    \left(S^5, S^1(F) \right) = 
    \left(B^4\times S^1, F\times S^1\right) \cup \left(S^3\times B^2, \partial F^\circ\times B^2\right).
    \]
    Spinning an orientable surface will yield a link of connected sums of some copies of $S^1\times S^2$ (or $S^3$). If $F$ is given by a banded unlink diagram with more than one maximum, the $S^1$-spin may not be a doubled 3-manifold in $S^5$. This obstacle prevents us from using the ideas in \autoref{subsec:doubles} to find a bridge quadrisection.

\begin{problem}
    Find a bridge quadrisection diagram for $S^1(F)$. 
\end{problem}

\section{Applications and future directions}\label{sec:future_directions}

We end this section with potential directions and applications of the theory of quadrisected embeddings of 3-manifolds in $S^5$. These are by no means complete; see \cite{Joseph_classical_surface, Joseph_Solids} for more ideas.

\subsection{Branched covers}\label{subsec:branched_covers}
An attractive feature of multisections of a 3-manifold $Y$ in $S^5$ is their convenient connection to the multisected closed 5-manifolds that arise as branched covers of $S^5$ along $Y$. The following proof is probably well known to experts, but we include a proof for the reader's convenience.

\begin{proposition}
    Consider a bridge quadrisection of a 3-manifold \[ 
(S^5,Y)=(W_1, E_1)\cup (W_2, E_2)\cup (W_3, E_3)\cup (W_4, E_4) , 
\] and a branched covering $f:Z\rightarrow S^5$ along $Y$, where $Z$ is a connected orientable closed 5-manifold. Then, \[ 
Z=f^{-1}(W_1)\cup f^{-1}(W_2)\cup f^{-1}(W_3)\cup f^{-1}(W_4), 
\] is a quadrisection of $Z.$
\end{proposition}

\begin{proof}
    Let $\widetilde{M}$ denote $f^{-1}(M).$ It is well-known that the branched cover $\widetilde{\Sigma}$ of $S^2$ along $2b$ bridge points is a closed orientable surface of genus $g$, where $g$ is calculated in \autoref{prop:Riemann-Hurewicz}. We next argue that the branched cover of an $n$-ball along a trivial $(n-2)$ ball tangle is an $n$-dimensional 1-handlebody.

    Let $(W_i,E_i)$ be a trivial $(n-2)$-ball tangle in an $n$-ball. We decompose $W_i$ into two pieces $W_i^1$ and $W_i^2$, where $W_i^1$ is a disjoint union $n$-balls, each containing a component of the trivial $(n-2)$-ball tangle. Define $W_i^2 = W\backslash W_i^1.$ Then, $\widetilde{W_i^1}$ is a disjoint union of $n$-balls. The submanifold $W_i^2$ does not contain the branched locus, so $\widetilde{W_i^2}$ is also a disjoint union of $n$-balls. The process of identifying  $W_i^2$ and $W_i^1$ together to form $W_i$ lifts to attaching $n$-dimensional 1-handles to disjoint union of $n$-balls. More precisely, the boundary of each component of $W_i^1$ is an $(n-1)$ spheres that can further be decomposed into two parts: the part that contains the boundary of the trivial $(n-2)$ ball tangle, and the part that meets $W_i^2$ in an $(n-1)$-ball. The branched covering $f$ restricted to this latter piece is again an $(n-1)$-ball. In conclusion, the manifold $\widetilde{W_i^2}$ is attached to each component of $\widetilde{W_i^1}$ along a disjoint collection of $(n-1)$-balls. As a 1-handlebody is constructed by attaching 1-handles to a disjoint collection of $n$-balls, the claim is verified.
\end{proof}
Thus, bridge 4-sections lift to 4-sected closed 5-manifolds. \autoref{fig:S2xS3_2fold_cover} demonstrates this process, which gives rise to Figure 6 from Section 4.4 of \cite{aribi23}. From \autoref{prop:Heegaard_splitting_from_4section}, we know we are looking at an embedding of $S^1\times S^2$, which is unknotted by \autoref{prop:crossingless_4planes}.

By a famous result of Alexander, every closed, connected, oriented $n$-manifold is a branched cover of the $n$-sphere~\cite{Alexander_Branched_Covers}. Thus, \autoref{thm:existance_bridge_4sections} can be used as an alternative proof of the existence of quadrisections of smooth 5-manifolds \cite{aribi23}. However, if the number of sheets is constrained to a certain range, the problem becomes more difficult. In particular, the following question is open. 

\begin{question}\label{prob:montesionsanalog}
Does every closed $5$-manifold admit a degree-5 cover of the 5-sphere?    
\end{question}

The main theorem of \cite{Blair24_BranchedCovers} solved a version of this question in dimension four using the tool of trisections. \autoref{thm:existance_bridge_position} then offers a potential tool to solve \autoref{prob:montesionsanalog} on relating 5-manifolds as branched covers of $S^5$ approached with 4-plane diagrams.

\begin{figure}[ht]
    \centering
    \includegraphics[width=.7\textwidth]{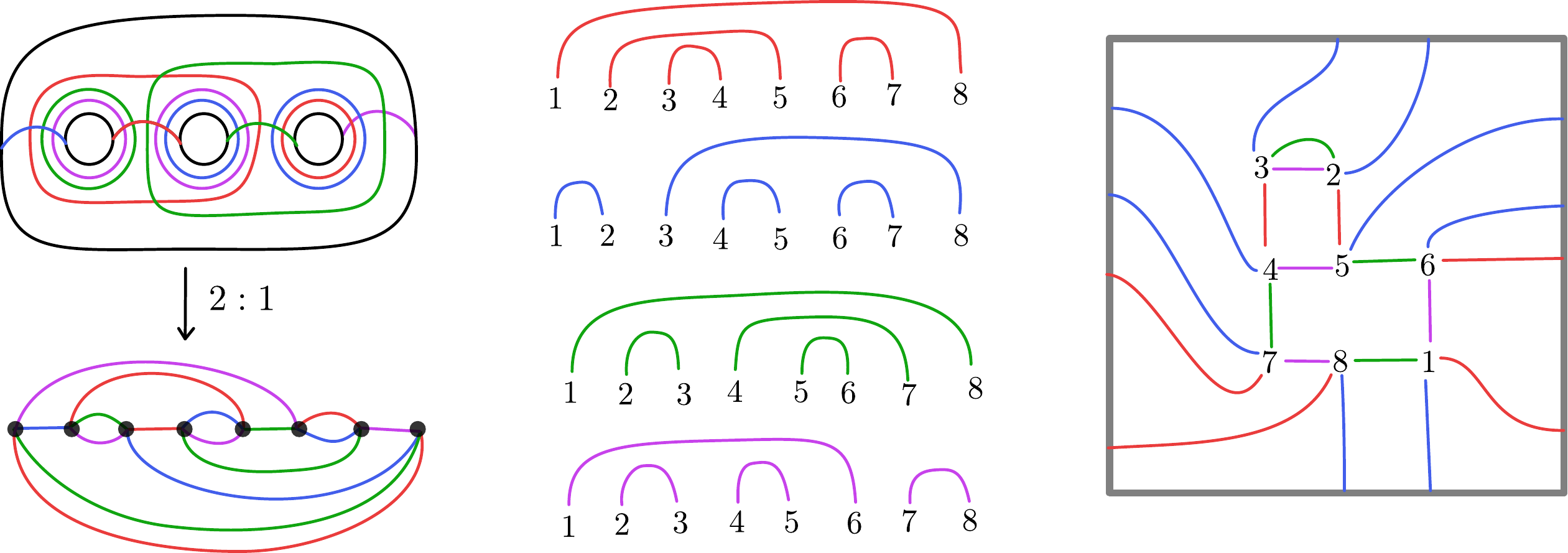}
    \caption{An illustration of $S^2\times S^3$ as a 2-fold cover of $S^5$ branched along an unknotted $S^1\times S^2$. Left: a genus-two 4-section diagram as a 2-fold cover of a 2-phere branched along eight points. Middle: the colored curves descend to arcs, which correspond to a tuple of trivial tangles $(\textcolor{red}{T_1},\textcolor{blue}{T_2},\textcolor{ForestGreen}{T_3},\textcolor{purple}{T_4})$ shown in the middle panel. Right: embedding the graph $\textcolor{red}{T_1}\cup\textcolor{blue}{T_2}\cup\textcolor{ForestGreen}{T_3}\cup\textcolor{purple}{T_4}$ in a torus. From \autoref{prop:Heegaard_splitting_from_4section}, we see that this tuple of tangles is the spine of an embedding of $S^1\times S^2$ in $S^5$.}
    \label{fig:S2xS3_2fold_cover}
\end{figure}

If one is interested in measuring how complicated a 3-manifold in a 5-manifold is, we give a lower bound for the bridge index of a knotted 3-manifold in terms of its branched cover, which follows from the Riemann-Hurwitz formula. Since a permutation can be written uniquely as a product of disjoint cycles up to ordering, the number of disjoint cycles $cyc(\rho)$ of a permutation $\rho$ is well-defined. Noticing that a branched cover is determined by $\rho:\pi_1(\Sigma \backslash P)\rightarrow S_n$, we will write the formula in terms of $cyc(\sigma)$. 

\begin{proposition}\label{prop:Riemann-Hurewicz}
    Suppose $Z$ is a closed $5$-manifold which is an $n$-fold branched cover over $S^5$ along a knotted $3$-manifold $Y$. Suppose also that $Y$ admits a $b$-bridge $4$-section. Let $g(Z)$ denote the $4$-section genus of $Z.$ Then, $g(Z) \leq 1-n+\frac{1}{2}\sum_{i=1}^{2b}(n-cyc(\rho(x_i)))$. In particular, when the cover is cyclic, we have
    \[
    g(Z) \leq 1-n+b(n-1).
    \] 
\end{proposition}

Given a 4-sected oriented connected 3-manifold $Y$, an Euler characteristic calculation gives us that $\chi(Y) = 0 = 2b-4b+6c-4d$ as follows. There are $2b$ 0-handles corresponding to the bridge points. There are $4b$ 1-handles corresponding to the number of bridge arcs in the four trivial tangles. We can form the Heegaard complex by attaching $6c$ 2-handles. Recall that the Heegaard complex is a Heegaard surface for $Y$ equipped with the compressing disks that determine the handlebodies. In other words, after compressions, we can attach $4d$ 3-handles to form $Y$. Therefore, the knowledge of the number of handles needed to build $Y$ bounds the bridge number from below.

For a more tractable lower bound, we can turn to homology groups and branched coverings. Let $\Sigma_2(Y)$ denote the 2-fold branched cover of $S^5$ along $Y$ and $\beta_i$ denote the $i$th-Betti numbers. Modifying the algorithm presented in \cite{cahn23}, we provide a Sage code that takes a colored bridge $4$-sected diagram as the input and produces the homology groups of the branched cover as the output. Detailed examples are presented in \autoref{sec:appendix}. The computer code can be found at \cite{github_link}.

\subsection{Group quadrisections of \texorpdfstring{$3$}{3}-manifold complements}

\textit{Group trisections}, originally due to \cite{abrams2018group}, are algebraic objects which capture all of the smooth information of a trisected $4$-manifold. Roughly, a group trisection is a cube of groups which is the result of applying van Kampen's theorem to the pieces of a trisected $4$-manifold, but perhaps surprisingly, given such a cube of groups which satisfy the requirements to be a group trisection, a trisected $4$-manifold corresponding to this data can be recovered \cite{abrams2018group}. 

In \cite{BKKLR} the second author, along with Kirby, Klug, Longo, and Ruppik, extended the notion of group trisections to the case of knotted surfaces in $4$-manifolds, and gave an explicit construction for how to find diagrams for the manifolds determined by the algebraic information of a group trisection. It is natural to ask whether this theory can be extended to group quadrisections of $5$-manifolds and $3$-manifold complements.

Given a quadrisected $S^5$ into four $5$-balls $W_i$, we can apply van Kampen's theorem to each of the four $4$-spheres $X_i = \partial W_i$ to produce four group trisections of $S^4$. Similarly, given a bridge quadrisected $3$-manifold in $S^5$, we can apply van Kampen's theorem to each of the four $4$-spheres $X_i = \partial W_i$ with a finite collection of disjoint, embedded, unknotted $2$-spheres removed, to produce four group trisections of this space. In either case, we informally call the collection of these four group trisections a \textit{group quadrisection}, which can be clustered in a 4-dimensional cube as in \autoref{fig:hypercube}.

\begin{figure}[ht]
    \centering
    \includegraphics[width=1\textwidth]{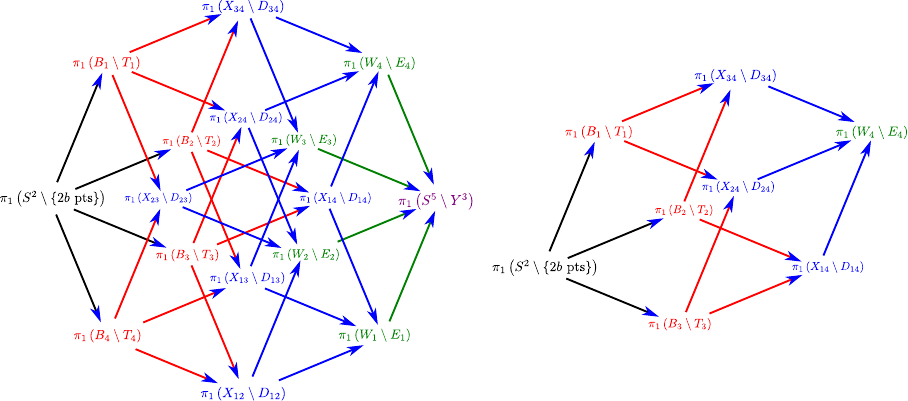}
    \caption{Consider a bridge quadrisection $(S^5,Y^3)=\bigcup_{i=1}^4 (W_i, E_i)$. The labeling convention in \autoref{def:multisections_3mans} yields a hypercube of epimorphisms determined by tuples of homomorphisms (in black) from the fundamental group of a $2b$-punctured surface to the fundamental group of tangle complements $B_i\setminus T_i$. Left and right diagrams correspond to the group quadrisection and group trisection for $\pi_1(S^5\setminus Y^3)$ and $\pi_1(W_4\setminus E_4)\cong \pi_1(\partial W_4\setminus \partial E_4)$, respectively.}
    \label{fig:hypercube}
\end{figure}

In order to recover a bridge quadrisected 3-manifold in $S^5$ from a group quadrisection, one would need to start with an algebraic object defined independently of any quadrisected manifold. However, it is crucial that the four group trisections push out to a group which is the fundamental group of the complement of an unlink of 2-spheres in $S^4$, i.e., a free group, and recognizing when a finite presentation is a free group is hard. Note that while \cite{BKKLR} shows that a group trisection uniquely determines a trisected knotted surface in a $4$-manifold, and furthermore produces a diagram of the surface, actually determining what these manifolds are in practice is difficult. Thus, we propose the following problem.

\begin{problem} \label{prob:group}
    Determine what algebraic conditions could be imposed on a group quadrisection in order to recover a bridge quadrisected 3-manifold in $S^5$.
\end{problem}

On the one hand, one might simply require that the four group trisections push out to free groups, as our desired groups are certainly free. However, it is unclear whether this condition is sufficient. By shifting the parameters of the trisections (and thus the Euler characteristic), we can ensure that we are building spheres, but it may not be obvious whether the embedding of the resulting $3$-manifold built from this process is smooth. As a consolation prize, we could, however, cone each component and build a PL-embedded $3$-manifold.

Although recognizing when a group trisection represents unknotted 2-spheres in $S^4$ is generally difficult, note that if we we start with four such trisections which are compatible in such a way as necessary to create a group quadrisection (meaning the parameters align as needed between the four trisections), then the work in this present paper combined with \cite{BKKLR} implies that this group quadrisection determines a $3$-manifold in $S^5$. 

This discussion also motivates the following problem, assuming a successful resolution to \autoref{prob:group}.

\begin{problem}
Use group quadrisections to compute the second homotopy groups of 3-knot complements. A paper by Lomonaco outlines this process \cite[\S6]{Lomonaco81_5dim_knot_theory}.
\end{problem}

\subsection{Braiding \texorpdfstring{$3$}{3}-manifolds}\label{subsec:braiding_3-manifolds}

The first and fifth authors, together with Carter and Courtney, showed how to braid a bridge trisected surface~\cite{aranda25}. Concretely, they gave procedures for turning a triplane diagram of a knotted surface in $S^4$ into a braid chart, a 2-dimensional version of a movie of braids. To braid a triplane diagram, the authors used modifications called 0-sector perturbations that do not change the underlying surface in $S^4$. These modifications make sense for 4-plane diagrams and also preserve the isotopy class of the bridge quadrisected 3-manifold in $S^5$. Thus, the proofs of the braiding methods in  \cite[\S 4]{aranda25} hold for 3-manifolds in $S^5$. The following theorem is immediate. 

\begin{theorem}\label{thm:rainbow}
Every orientable $3$-manifold embedded in $S^5$ admits a description as a \emph{rainbow diagram}; that is, a tuple $(T_1,T_2,T_3,T_4)$ of $b$-string tangles \emph{braided} with respect to a fixed axis such that each pair $T_i\cup\T_{i+1}$ is braid isotopic to a crossingless braid with some Markov stabilizations.
\end{theorem}

The invested reader may wish to upgrade \autoref{thm:rainbow} so that each triplet $(T_i,T_j,T_k)$ is isotopic to \emph{braided perturbations} of the crossingless rainbow diagram~\cite[\S 3.1]{aranda25}. If true, then a notion of a braid chart for a knotted 3-manifold seems within reach. 

\begin{question}
    Is there a notion of a braid chart for a 3-knot? Do rainbow diagrams for 3-knots yield braid charts as in Section 6 of \cite{aranda25}?
\end{question} 

Lomonaco's movies of movies of knots are another alternative to the Heegaard complexes in \autoref{sec:Heegaard_complexes}. One could ask whether the 4-section bridge of a knotted 3-manifold $Y$, or a rainbow diagram instead, could be used to find descriptions of $Y$ as a movie of movies of knots~\cite{Lomonaco81_5dim_knot_theory}. The four-dimensional version of the question below is true for braided diagrams of surfaces in the 4-ball~\cite[\S 1.3]{aranda25}. 

\begin{question}
    Do braided 4-plane diagrams describe geometrically meaningful (transverse) 3-manifolds in $S^5$ with respect to the standard contact structure?
\end{question}

\subsection{Towards uniqueness of bridge quadrisections}\label{subsec:uniqueness_bridge_3mans}

In order to define new 3-knot invariants using bridge quadrisections or Heegaard complexes, one may wish to have a complete set of moves connecting any two descriptions of isotopic embeddings. The moves between Heegaard complexes were discussed in \autoref{sec:Heegaard_complexes}; see \autoref{def:canceling_handles} and \autoref{thm:canceling_handles_HC}. 


We discuss three types of moves between bridge quadrisections that do not change the isotopy class of an embedded 3-manifold. We first fix some notation. Let $\Tcal=(T_1,T_2,T_3,T_4)$ be a fixed cyclic ordering of the spine of a $b$-bridge quadrisected 3-manifold. It follows from \autoref{prop:Heegaard_splitting_from_4section} and \autoref{def:heegaard_complex_bridge_position} that $\Tcal$ defines a Heegaard complex in bridge position where $\Sigma$ is the $b$-bridge quadrisected surface in $S^4$ described by $\Tcal$ and $\Da$ and $\Db$ are embedded disks in 3-space bounded by the links $T_1\cup \T_3$ and $T_2\cup \T_4$, respectively. 

The first two moves are interior Reidemeister moves and mutual braid moves, which correspond to isotopies of the surface $\Sigma$ that do not change the number of bridge points~\cite{meier17, aranda24}. \emph{Interior Reidemeister moves} are isotopies of the tangles $T_i$ that fix their endpoints and \emph{mutual braid moves} are the result of appending the same $2b$-stranded braid to each tangle of $\Tcal$. These moves induce isotopies of their associated Heegaard complexes. 

The third move, called \emph{3-manifold perturbation},\footnote{We are working on the name.} corresponds to the addition of canceling pairs to the associated Heegaard complex. In short, such a move will behave like a multiple-sector perturbation of a 4-dimensional bridge 4-section from \autoref{lem:sector-perturbation} for any permutation of the indices $(1,2,3,4)$. Recall the two types of tangle modifications from \autoref{sec:calculus_embedded_surfaces}. 

Fix $0\leq k \leq 3$ and let $I$ be a subset of $\Z_4$ with $k$ elements. For each $i\in I$, let $\rho_i$ be a band inducing a modification of type 1 in the tangle $T_i\subset B_i$. Let $\Tcal'=(T_1',T_2',T_3',T_4')$ be the result of band surgery on each tangle of $\Tcal$. 

\begin{lemma}\label{lem:3man_perturbation}
Suppose that for each pair $i\neq j$ in $I$, there is a $2$-sphere in $B_i\cup \overline{B}_j$ that intersectcs $\partial B_i=\partial B_j$ in one loop, and contains both the band $\rho_i\cup \overline{\rho}_j$ and exactly one component of $T_i\cup \T_j$. Then the tuple $\Tcal'$ represents the same $3$-manifold as $\Tcal$. We call $\Tcal'$ a \emph{$k$-sector $3$-manifold pertubation of $\Tcal$.} 
\end{lemma}
\begin{proof}
The 2-sphere condition on $\rho_i\cup \overline{\rho}_j$ is equivalent to the assumption that for each pair $i,j\in I$, the link $T_i\cup\T_j$ can be isotoped, via interior Reidemeister moves and mutual braid moves, to a diagram as in \autoref{fig:sphere_condition}. We briefly explain how a 3-manifold perturbation changes the associated Heegaard complex depending on the value of $k$.  

\begin{enumerate}
\item If $|I|=1$, the Heegaard complex changes by the addition of a 0/1-canceling pair since a 1-bridge 4-sected unknotted 2-sphere gets tubed to $\Sigma$ along the core of the band $\rho_i$. 
\item If $|I|=2$ and $I$ contains consecutive indices (i.e., $I=\{1,2\}$), then the new Heegaard complex is isotopic to $(\Sigma;\Da,\Db)$.
\item Suppose that $|I|=2$ and $I$ contains two opposite indices, say $I=\{1,3\}$. Then performing band surgery on $\Tcal$ splits one disk of $\Da$ into two subdisks while tubing the surface $\Sigma$ along the band $\rho_1\cup \overline{\rho}_3$; see \autoref{lem:tubing_4sec}. Effectively, this adds a 1/2-canceling pair to the Heegaard complex. 
\item To end, suppose that $I=\{1,2,3\}$. By \autoref{lem:sector-perturbation}, performing band surgery on $\Tcal$ does not change the isotopy class of $\Sigma$. At the same time, one disk of $\Da$ is split in half by the band $\rho_1\cup \overline{\rho}_3$. Hence, the Heegaard complex changes by the addition of a 1/2-canceling pair. 
\end{enumerate}
\end{proof}

\begin{conjecture}\label{conj:uniqueness_bridge_3mans}
Any two $4$-plane diagrams describing isotopic $3$-manifolds in $S^5$ are related by a finite sequence of interior Reidemeister moves, mutual braid moves, and $3$-manifold perturbations. 
\end{conjecture}

To prove the uniqueness conjecture above, one needs to overcome two challenges: 
(1) find a combination of moves on 4-plane diagrams that resembles a handle slide, and (2) prove that any two 4-plane diagrams of isotopic surfaces in $S^4$ admit a common perturbation. 

\appendix
\section{Computing homology groups of branched covers} \label{sec:appendix}

In this appendix, we compute classical invariants from algebraic topology of branched covers taking quadruplane diagrams as inputs. Given a closed 5-manifold which is 4-sected with the parameters $(\Sigma;H_{\alpha},H_{\beta},H_{\gamma},H_{\delta})$, we let $L_i = \ker(\iota) \colon H_1(\Sigma) \rightarrow H_1(H_{\mu})$ for $\mu = \alpha,\dots,\delta$. By a result of \cite{aribi23}, the homology groups can be determined solely from these Lagrangians.

\begin{theorem}[\cite{aribi23}]\label{th:homology}
 The homology of $W$ is the homology of the complex
 $$0\to\mathbb{Z}\xrightarrow{0}\bigoplus_{i=1}^n\left(\cap_{j\neq i}L_j\right)\xrightarrow{\delta_{n-1}}\dots\xrightarrow{\delta_{\ell+1}}\bigoplus_{|I|=\ell}\left(\cap_{i\in I}L_i\right)\xrightarrow{\delta_\ell}\dots\xrightarrow{\delta_2}\bigoplus_{i=1}^nL_i\xrightarrow{\delta_1} H_1(\Sigma)\xrightarrow{0}\mathbb{Z}\to 0,$$ where $\delta_j$ maps are defined as \begin{center}
$\delta_\ell\colon \bigoplus_{|I|=\ell}\left(\bigcap_{i\in I} L_i\right)\longrightarrow
\bigoplus_{|I|=\ell-1}\left(\bigcap_{i\in I} L_i\right), \delta_{\ell}(c) = (-1)^{\lvert\{\,s\in I\mid s<j\,\}\rvert}\,c.$
\end{center}
\end{theorem}

Therefore, our first goal is to obtain the Lagrangians from a bridge 4-section. The input we need is a homomorphism $\rho\colon \pi_1(S^2\backslash \{p_0,p_1,\ldots,p_{2b-1}\})\rightarrow S_n$ that extends over the quadrapod. That is, for each $i\in \{1,2,3,4\},$ there exists a map $\rho_i$ that makes the following diagram commute.

\begin{center}
    \begin{tikzcd}[row sep=2.5em, column sep=3.5em]
\pi_1(S^2\setminus\{2b\ \text{pts}\})
  \arrow[r,"\iota_i"]
  \arrow[dr,"\rho"']
& \pi_1(D^3\setminus T_i)
  \arrow[d, dashed, "\rho_i"]
\\
& S_n
\end{tikzcd}
\end{center}
 We will let $S,\Sigma'$, and $\Sigma$ denote the bridge sphere downstairs, the center surface of the unbranched cover, and the quadrisection surface in the branched cover, respectively.

\begin{framed}
\noindent\textbf{Algorithm for computing the homology of a branched cover.}
\vspace{3mm}

\begin{enumerate}[leftmargin=0.75in]
\item[\textbf{Step 1.}] Write down a presentation for the fundamental group of the bridge sphere $\pi_1(S)$. We can always assume it has the form 
\begin{center}
    $\pi_1(S)=\langle x_0, x_1, \ldots , x_{2b-1} \;\ | \;\ x_0x_1 \cdots x_{2b-1}\rangle$. 
\end{center}

\item[\textbf{Step 2.}] Write down a presentation for the fundamental group of the punctured surface $\pi_1(\Sigma')$ in the unbranched cover $X'$. Label the basepoints of $\Sigma'$ as $P_1,P_2,\ldots, P_n$. We will actually get the presentation from a space homotopy equivalent to $X'$ formed by attaching edges from one vertex $V$ positioned disjoint from $X'$ to the basepoints and then attaching in some 2-cells following the following instructions. Choose a path $\gamma_j$ from $P_1$ to $P_j$ such that (1) $\gamma_j$ is a lift of a
word $\tau_j$ in the generators $x_i$ and (2) the union of the paths $\gamma_j$ is a tree $T$ with vertices $P_i$ and
with edges a subset of the $\widehat{x_i^j}$.  Each two cell fills the triangle $[P_i
, V, P_j ]$ with edges $e^{-1}_i,e_j,e^{-1}$
for each edge $e$ of
the tree $T$. Each \textbf{generator} $x_i^j$ has the form $\gamma_j * \widehat{x_i^j}*\gamma_{\rho(i)(j)}^{-1}.$ There are two types of \textbf{relations}. The first type is the claw relations corresponding to the 2-cells attached. The second type is dictated by $\rho,$ which will be worked out in detail in the next example.

\item[\textbf{Step 3.}]  Write down a presentation for the fundamental group of the closed surface $\pi_1(\Sigma)$ in the branched cover. This amounts to adding one relation for each
disjoint $k$-cycle in the permutation $\rho(x_i)$.

\item[\textbf{Step 4.}] By now, we are done at the central surface level. Next, take presentations of the group of a trivial tangle complements. We again go in stages to obtain presentations for the lift of these tangle complements to the unbranched cover and then to the branched cover.

\item[\textbf{Step 5.}] Determine a basis for each Lagrangian $L_i = \ker(\iota) \colon H_1(\Sigma) \rightarrow H_1(H_{\mu})$.

\item[\textbf{Step 6.}]
Apply \autoref{th:homology}. 
\end{enumerate}
\end{framed}


We now walk through an example in detail.

\begin{example}
    Consider the 6-bridge 4-section in \autoref{fig:S2_Spun_2}, where $\beta = \sigma_1^3$ using standard Artin generators  with the following homomorphism $\rho$ sending $x_0,x_1 \mapsto (1,2)$ and $x_2,x_3,\ldots, x_{11} \mapsto (1,3)$.

\begin{enumerate}[leftmargin=0.75in]

\item[\textbf{Step 1.}]  We start with a group presentation of the punctured sphere $\pi_1(S^2\backslash \{12 \;\ \text{points}\})=\langle x_0, x_1, \ldots , x_{11} \;\ | \;\ x_0x_1 \cdots x_{11}\rangle$. 

\item[\textbf{Step 2.}] The 3-sheeted cover is generated by $\lbrace x_0^1, x_1^1, \ldots , x_{11}^1, x_0^2, x_{1}^2, \ldots , x_{11}^2,x_{1}^3, \ldots , x_{11}^3\rbrace$. We can use $x_0^1=1$ and $x_2^1=1$ as claw relations. The remaining relations depend on $\rho$:
\begin{center}
$w_1=x_0^1x_1^2x_2^1x_3^3x_4^1x_5^3x_6^1x_7^3x_8^1x_9^3x_{10}^1x_{11}^3$

$w_2=x_0^2x_1^1x_2^2x_3^2x_4^2x_5^2x_6^2x_7^2x_8^2x_9^2x_{10}^2x_{11}^2$

$w_3=x_0^3x_1^3x_2^3x_3^1x_4^3x_5^1x_6^3x_7^1x_8^3x_9^1x_{10}^3x_{11}^1$
\end{center}
\noindent In conclusion, 
\begin{center}
    $\pi(\Sigma')=\langle x_0^1, x_1^1, \cdots , x_{11}^1, x_0^2, x_{1}^2, \cdots , x_{11}^2,x_{1}^3, \cdots , x_{11}^3 \;\ | \;\ x_0^1,x_2^1,w_1,w_2,w_3\rangle$.
\end{center}

\item[\textbf{Step 3.}] We now get the presentation for the closed central surface in the branched cover, which is a quotient of $\pi(\Sigma')$, where each disjoint cycle in $\rho(x_i)$ contributes an additional relation. The presentation of $\pi(\Sigma)$ we get from this method is 
\begin{align*}
    \langle x_0^1, x_1^1, \ldots , x_{12}^1, x_0^2, x_{1}^2, \ldots , x_{12}^2,x_{1}^3, \ldots , x_{12}^3 \;\ | \;\ x_0^1,x_2^1,w_1,w_2,w_3,x_0^1x_0^2,x_0^3,x_1^1x_1^2,x_1^3,x_j^1x_j^3x_j^2  \rangle,
\end{align*}
where $3\leq j \leq 11.$

\item[\textbf{Step 4.}] We start with presentations for the tangle complements via Wirtinger presentations applied to the quadruplane diagram.
\begin{align*}
\pi_1(B^3\backslash T_1) &= \langle x_0,\ldots,x_{11} \;\ | \;\ x_4x_{11},x_5x_6,x_7x_8,x_9x_{10},x_1x_2x_1x_2^{-1}x_1^{-1}x_3,x_0x_1x_2x_1x_2x_1^{-1}x_2^{-1}x_1^{-1} \rangle \\
\pi_1(B^3\backslash T_2) &= \langle x_0,\ldots,x_{11} \;\ | \;\ x_3x_{4},x_5x_6,x_7x_{10},x_8x_{9},x_1x_2x_1x_2^{-1}x_1^{-1}x_{11},x_0x_1x_2x_1x_2x_1^{-1}x_2^{-1}x_1^{-1} \rangle \\
\pi_1(B^3\backslash T_3) &= \langle x_0,\ldots,x_{11} \;\ | \;\ x_2x_{5},x_3x_4,x_7x_8,x_9x_{10},x_1x_6x_1x_6^{-1}x_1^{-1}x_{11},x_0x_1x_6x_1x_6x_1^{-1}x_6^{-1}x_1^{-1} \rangle \\
\pi_1(B^3\backslash T_4) &= \langle x_0,\ldots,x_{11} \;\ | \;\ x_3x_{10},x_4x_9,x_5x_8,x_6x_{7},x_1x_2x_1x_2^{-1}x_1^{-1}x_{11},x_0x_1x_2x_1x_2x_1^{-1}x_2^{-1}x_1^{-1} \rangle
\end{align*}

Now we lift these presentations to the unbranched cover dictated by $\rho$. For instance, consider the relation $x_1x_2x_1x_2^{-1}x_1^{-1}x_3$. This relation lifts to three relations. Let's demonstrate by following $\rho$ through the sheets. For example, if we start at sheet 1, we get
\begin{center}
$1 \overset{\rho(x_1)}{\mapsto} 2 \overset{\rho(x_2)}{\mapsto} 2\overset{\rho(x_1)}{\mapsto} 1 \overset{\rho(x_2^{-1})}{\mapsto} 3 \overset{\rho(x_1^{-1})}{\mapsto} 3 \overset{\rho(x_3)}{\mapsto} 1.$ 
\end{center}
Note that $\rho(x_i)$ is a transposition, so $\rho(x_i^{-1})= \rho(x_i)$. Thus, the corresponding lifted relation is
\begin{center}
$x_1^1x_2^{2}x_1^{2}(x_2^{3})^{-1}(x_1^{3})^{-1}x_3^{3}.$
\end{center}

Performing similar steps for the other relations, we obtain the following presentations for the four handlebodies in the spine of the closed manifold branched cover.
\begin{center}
$H_{\alpha} = \langle x_4^1, x_5^1, x_7^1, x_9^1 \rangle$

$H_{\beta} = \langle x_3^1, x_5^1, x_7^1, x_8^1 \rangle$

$H_{\gamma} = \langle x_3^1, x_6^1, x_7^1, x_9^1 \rangle$

$H_{\delta} = \langle x_3^1, x_4^1, x_5^1, x_6^1 \rangle$
\end{center}
This agrees with the fact that the 4-section of the central surface in the branched cover has genus 4.

\item[\textbf{Step 5.}] Now, we look at inclusion maps and Lagrangians. Abelianizing $\pi_1(\Sigma),$ we get that $H_1(\Sigma)$ has basis $\{x_3^1, x_4^1, x_5^1, x_6^1,x_7^1, x_8^1, x_9^1, x_{10}^1\}$. The output of the code displays the kernels of the inclusion maps $L_{\mu}$ as follows.
\begin{align*}
\text{Red} &= [
    (1, 0, 0, 0, 0, 0, 0, 0),
    (1, 0, 1, -1, 1, -1, 1, -1),
    (0, 0, 1, -1, 0, 0, 0, 0),
    (0, 0, 0, 0, 1, -1, 0, 0)
] \\
\text{Blue} &= [
    (1, -1, 0, 0, 0, 0, 0, 0),
    (0, 0, 1, -1, 0, 0, 0, 0),
    (0, 0, 0, 0, 1, 0, 0, -1),
    (0, 0, 0, 0, 0, 1, -1, 0)
] \\
\text{Green} &= [
    (1, -1, 0, 0, 0, 0, 0, 0),
    (0, 0, 1, 0, 0, 0, 0, 0),
    (1, -1, 1, 0, 1, -1, 1, -1),
    (0, 0, 0, 0, 1, -1, 0, 0)
] \\
\text{Purple} &= [
    (1, 0, 0, 0, 0, 0, 0, -1),
    (0, 1, 0, 0, 0, 0, -1, 0),
    (0, 0, 1, 0, 0, -1, 0, 0),
    (0, 0, 0, 1, -1, 0, 0, 0)
] 
\end{align*}

\item[\textbf{Step 6.}] The red matrix translates to 
\begin{center}
    $L_{\alpha}= \mathbb{Z}[x_3^1,x_3^1+x_5^1-x_6^1+x_7^1-x_8^1+x_9^1-x_{10}^1,x,x_5^1-x_6^1,x_7^1-x_8^1]$
\end{center}
 in compact form, for instance. In conclusion,  
the homology groups are $H_1 \cong H_4 \cong 0$ and $H_2\cong H_3\cong \mathbb{Z}$. Indeed, these are the homology groups of $S^2\times S^3.$
\end{enumerate}
\end{example}

\sloppy 
\printbibliography[title={References}]
\end{document}